\documentclass[11pt,oneside,english]{amsart}

\usepackage{hyperref}
\hypersetup{
    colorlinks=true,
    linkcolor=blue,
    filecolor=magenta,      
    urlcolor=cyan,
}

\usepackage{amsfonts}
\usepackage{amsmath}
\usepackage{amssymb}
\usepackage{amsthm,color}
\usepackage{enumerate}
\usepackage{mathtools}
\usepackage{enumitem}
\mathtoolsset{showonlyrefs}
 \usepackage{transparent}
 \usepackage{stackengine,tikz}

\usepackage{breqn}
\usepackage{tikz}
\usepackage{tikz-3dplot}

\usepackage{longtable}

\usepackage{soul}

\tdplotsetmaincoords{70}{110}

\usepackage{color,graphicx,enumerate,wrapfig,amssymb,mathtools}
\usepackage{epsfig}
\usepackage{pxfonts}
\usepackage{graphicx}
\usepackage{eucal}
\newcommand{\R}{{\mathbb R}}

\newcommand{\N}{{\mathbb N}}

\newcommand{\cM}{{\mathcal M}}
\newcommand{\cL}{{\mathcal L}}
\newcommand{\cS}{{\mathcal S}}
\newcommand{\cK}{{\mathcal K}}
\newcommand{\cA}{{\mathcal A}}

\newcommand{\cH}{{\mathcal H}}

\newcommand{\cF}{{\mathcal F}}

\newcommand{\e}{\varepsilon}
\newcommand{\al}{\alpha}
\newcommand{\be}{\beta}
\newcommand{\de}{\delta}

\newcommand{\la}{\lambda}

\newcommand{\vp}{\varphi}

\newcommand{\dist}{\operatorname{dist}}

\newcommand{\loc}{\operatorname{loc}}

\newcommand{\supp}{\operatorname{supp}}
\newcommand{\D}{\nabla}

\numberwithin{equation}{section}

\def\Xint#1{\mathchoice
  {\XXint\displaystyle\textstyle{#1}}%
  {\XXint\textstyle\scriptstyle{#1}}%
  {\XXint\scriptstyle\scriptscriptstyle{#1}}%
  {\XXint\scriptscriptstyle\scriptscriptstyle{#1}}%
  \!\int}
\def\XXint#1#2#3{{\setbox0=\hbox{$#1{#2#3}{\int}$}
    \vcenter{\hbox{$#2#3$}}\kern-.5\wd0}}

\def\dashint{\Xint-}

\newtheorem{theorem}{Theorem}
\theoremstyle{plain}

\newtheorem{corollary}{Corollary}

\newtheorem{lemma}{Lemma}

\newtheorem{proposition}{Proposition}
\newtheorem{remark}{Remark}

\numberwithin{equation}{section}

\newcommand{\mean}[1]{\langle{#1}\rangle}
\usepackage{amsmath}
\usepackage{amsfonts}
\usepackage{amssymb}
\usepackage{caption}
\usepackage{subcaption}
\usepackage{enumerate}

\usepackage{geometry}
\makeatletter
\@namedef{subjclassname@2020}{%
  \textup{2020} Mathematics Subject Classification}
\makeatother

\author{Layan El Hajj}
\address[Layan El Hajj]{Carnegie Mellon University in Qatar, Doha, Qatar}
\email{lelhajj@andrew.cmu.edu}

\author{Seongmin Jeon}
\address[Seongmin Jeon]
{Department of Mathematics Education \newline 
\indent Hanyang University \newline 
\indent 222 Wangsimni-ro, Seongdong-gu, Seoul 04763, Republic of Korea} 
\email[Seongmin Jeon]{seongminjeon@hanyang.ac.kr}

\author{Henrik Shahgholian\,
}
\address[Henrik Shahgholian]{Department of Mathematics, KTH Royal Institute of Technology, SE-10044 Stockholm, Sweden  
} \email{henriksh@kth.se}

\thanks{The work of L.E. Hajj was  supported by the Qatar Foundation through Carnegie Mellon University in Qatar’s Seed Research program. S. Jeon was supported by the National Research Foundation of Korea(NRF) grant funded by the Korea government(MSIT) (RS-2025-24803159 and RS-2026-25474502).  H. Shahgholian was supported by Swedish Research Council (grant nr. 2025-03740).}

\begin{document}

\title[systems  with free boundaries  ]
{Existence theory for non-variational  systems \\ with free boundaries }

\date{\today}

\begin{abstract}
We study the existence of solutions for systems of both elliptic and parabolic partial differential equations with potentially singular right-hand sides and free boundaries, posed in general smooth domains.

Our results are established within a framework of "meta-theorems." This approach hinges on specific strong properties of the operators and their solutions (in approximate smooth settings) to guarantee the existence of a limit as the approximation parameter tends to zero. The primary challenge lies in applying these meta-theorems to prototype cases, which requires verifying that the necessary strong properties hold. For our analysis, we focus on fully nonlinear and $p$-Laplacian operators  and a mixing of these operators, in both elliptic and parabolic contexts.
While we focus on these specific cases, the meta-theorems remain valid for any other operators that satisfy the required properties.

Beyond the complex proofs of our meta-theorems, and their applications to specific operators, a major challenge is the technical handling of the $p$-parabolic case, which requires proving  the regularity of  solutions of the $p$-parabolic equation with singular or degenerate right-hand side--addressed in the Appendix--along with  several (new) properties, which is missing in the literature.

\end{abstract}

\maketitle
\setcounter{tocdepth}{2}

\tableofcontents

\section{Introduction}

\subsection{Background}

Free boundary problems for elliptic systems have in recent years been at the center of considerable attention. 
Their origins can be traced back to the pioneering work of Giaquinta and Giusti \cite{GiaGiu83, GiaGiu84}, 
and likely even earlier contributions. 
In a variational setting, the minimization of energy functionals for vector-valued functions with non-analytic forcing terms naturally leads to the formation of free boundaries.

As a representative example, consider the functional
\[
 \int_\Omega (|\nabla \mathbf{u}|^2 + |\mathbf{u}|^\gamma) \, dx,
\]
where $0 < \gamma < 2$ and $\mathbf{u} = (u^1, \dots, u^m)$. This functional naturally gives rise to free boundaries when boundary values are sufficiently small. Specifically, for small values of $|\mathbf{u}|$, the potential term $|\mathbf{u}|^\gamma$ decays more slowly than the quadratic gradient energy $|\nabla \mathbf{u}|^2$.\footnote{Formally, this can be verified by comparing the energy of a non-zero solution on a ball $B \subset \Omega$ with that of the zero solution, showing that the energy of the latter is strictly lower for sufficiently small configurations.}

To see this, consider the scaling $\mathbf{u}_\varepsilon(x) = \varepsilon \mathbf{u}(x)$. The energy components scale as:
\[
\int |\nabla \mathbf{u}_\varepsilon|^2 = \varepsilon^2 \int |\nabla \mathbf{u}|^2 \quad \text{and} \quad \int |\mathbf{u}_\varepsilon|^\gamma = \varepsilon^\gamma \int |\mathbf{u}|^\gamma.
\]
Since $\gamma < 2$, the potential term dominates for small amplitudes as $\varepsilon \to 0$. Consequently, maintaining a small but strictly positive solution is energetically more expensive than setting $\mathbf{u} = 0$ over a larger portion of the domain, thereby eliminating both the gradient and potential contributions. Thus, it is energetically favorable for the minimizer to develop regions where $\mathbf{u}$ vanishes identically.

These zero regions are separated from the positive phase by a \emph{free boundary}, whose location is not prescribed a priori but determined by energy balance. These examples show that singular reactions naturally produce free boundaries at the scalar level. One purpose of the present paper is to develop an existence theory for coupled systems in which analogous free boundary behavior may occur, even in non-variational settings. In coupled systems, however, the mechanism is more delicate, since the source term in one equation depends on the unknown from the other equation, and may itself become singular or degenerate near the free boundary. To prove an existence theorem for systems in general, one would need to take into account the presence of a free boundary; especially when the right hand side of an equation becomes singular. For example, for the above functional when $0 < \gamma < 1$, the functional is not convex, and the corresponding Euler-Lagrange equation has a singular right hand side. Nevertheless, even when $1 \leq \gamma <2$ the right hand side of the Euler-Lagrange equation acts as an absorbing term, creating dead cores, whose boundary is a free boundary. 
Things become more complicated when we consider equations that are non-variational  or do not admit a functional that represents the Euler–Lagrange equation; As a prototype non-variational model, one may consider the coupled system
\begin{equation}\label{eq:system0}
    \Delta u = v^a, \qquad \Delta v= u^b,
\end{equation}
in a bounded domain with prescribed boundary values, where the exponents are chosen in a range for which the right-hand sides may be singular or strongly degenerate near zero.
Even for such simple-looking systems, the interaction between the two components and the possible presence of free boundaries make the existence theory highly nontrivial.
In such settings, when the free boundary is absent and the equations are reasonably smooth, one may consider iteration techniques or fixed point theory. Both approaches require certain uniform estimates to achieve the goal. Such estimates are usually very hard to obtain, as the norms may accumulate in the iteration and eventually blow up in the limit. In singular regimes, this becomes
particularly delicate, since the right-hand side may deteriorate precisely near the set where
one expects a free boundary. This difficulty motivates the abstract framework developed in the present paper.
Recent work has highlighted several mechanisms by which coupled equations may generate dead cores and interacting free boundaries; see, for instance, \cite{AraTey24, AraTey25}. On the other hand, the available existence theory remains largely model-dependent and often relies on a special structure. In particular, the work \cite{ELSH25} treats the spherically symmetric  case for a related class of systems. The aim of the present paper is to provide a more flexible existence framework, covering both elliptic and parabolic equations and allowing for fully nonlinear as well as $p$-Laplacian type operators.

\subsection{Main results in this paper}

In this paper, we consider general non-variational systems with a partially singular right-hand side and free boundaries. We prove existence of solutions for   the two-component case of a system of the form:\footnote{In this paper, we do not specify the notion of solution for the partial differential equation part of the system. This allows us to adapt our theorems to various notions of solutions, provided that certain properties (solvability, comparison, convergence) are fulfilled. These requirements will be clarified as we proceed. See \cite{AllenKriventsovShahgholian2023} for a similar type of ideas.
}
    \begin{align}\label{eq:sol-2}
        \begin{cases}
            \cL_1 u=f_1k_1(v)h_1(u)&\text{in }\Omega\cap\{u>0\},\\
            \cL_2 v=f_2k_2(u)h_2(v)&\text{in }\Omega\cap\{v>0\},\\
            u=g_1, v=g_2&\text{on }\partial\Omega,\\
            u=|\D u|=0&\text{on }\Omega\cap\partial\{u>0\},\\
            v=|\D v|=0&\text{on }\Omega\cap\partial\{v>0\},
        \end{cases}
    \end{align}
and its parabolic counterpart. Here the nonlinearities $h_i$ may be singular at the origin, while the coupling functions $k_i$ are merely assumed to be non-decreasing and continuous. We also assume throughout the paper that for $j=1,2$
$$
g_j \in C(\partial\Omega; \mathbb{R}_+), \qquad 
\| g_j \|_{L^\infty (\overline \Omega ) } \leq M_0, 
$$
for a constant $M_0 > 0$. For technical reasons, 
in some special cases,  we may assume $g_j$ are Lipschitz; see Remark \ref{rem:DirD}.

Our first main contribution is an abstract existence theorem formulated in both elliptic and parabolic settings. Rather than proving the existence directly for each operator separately, we  prove a pair of meta-theorems, one elliptic (see Theorem \ref{thm:meta}) and one parabolic (see Theorem \ref{thm:meta-par}) that guarantee the existence of a solution pair $(u, v)$ to the system under the condition that the right-hand side behaves in a reasonably controlled way (as stated in equations \eqref{eq:rhs-1},\eqref{eq:rhs-2}) and three  properties of the operators $L_i$ : 
\begin{enumerate}
\item a solvability property; \ref{cond-sol}/\ref{cond-sol-par} ;
\item a comparison principle; \ref{cond-comp}/\ref{cond-comp-par};
\item a convergence property under uniform limits; \ref{cond-conv}/\ref{cond-conv-par}.
\end{enumerate}
The key technical device is a monotone iteration scheme. Monotone sequences of regular scalar solutions are constructed and arranged so that one component is monotone increasing and the other monotone decreasing. This yields pointwise convergence of the full sequences. The uniform regularity  then gives locally uniform convergence, which allows us to pass to the limit by the stability property.
Our second contribution is to verify the abstract hypotheses for two important classes of
operators: fully nonlinear uniformly elliptic/parabolic operators and the $p$-Laplacian and $p$-parabolic operators.\footnote{In section \ref{sec:mix} we shall discuss mixing of these operators.} 
In this way, the meta-theorems become concrete existence results for broad families of singular free boundary systems.
We also study the geometry of the coincidence sets. Under additional non-degeneracy and Hopf-type assumptions, we prove that their connected components (or their time slices in the parabolic case)
either coincide or are disjoint. This is carried out in both the elliptic and parabolic frameworks.
The regularity obtained in our construction is inherited from the corresponding scalar theory used in the approximation scheme. The solutions $(u,v)$ inherit from the scalar theory 
H\"older regularity up to the free boundary in the elliptic case, and local parabolic
H\"older regularity $p$-Laplacian case
(see Appendix \ref{appen:par-p-Lapl-reg}).

\medskip

\begin{remark}\label{rem:DirD}{\bf The Dirichlet data:} 
Since our focus is primarily on the free boundary problem—specifically the last two lines in \eqref{eq:sol-2}—the Dirichlet data on $\partial \Omega$ is of secondary concern, as it constitutes a standard PDE problem, though its well-posedness must still be verified. Nevertheless, the nonlinearity on the right-hand side of our system introduces technical challenges; in many instances involving Dirichlet data, barrier arguments are required. Such barriers are difficult to construct in the fully nonlinear regime when using a viscosity approach.
Here, we utilize the distance function to the boundary, which necessitates $C^{1,1}$-smoothness of the boundary within a small tubular neighborhood. A similar assumption is required for the parabolic case. 

For the p-Laplacian and p-parabolic cases, however, our approach is variational, where Lipschitz boundary and data suffice to construct the necessary barrier arguments.

These assumptions will be specified in each case, when we state the results.

 \end{remark}

\subsection{Structure of the paper}

The  paper is organized as follows. Section~\ref{sec:ell} is dedicated to the elliptic case, beginning with the formulation and proof of a general elliptic meta-theorem in Section~\ref{sec:el-meta}, as well as a result on the relation between the supports of the components of the system. 
This framework is subsequently applied to fully nonlinear operators in Section~\ref{sec:FN-elliptic} and to the $p$-Laplacian case in Section~\ref{sec:p-lap-elliptic}. In Section~\ref{sec:par}, we extend our analysis to the parabolic setting, establishing a parabolic meta-theorem in Section~\ref{sec:par-meta} and exploring its applications to fully nonlinear and $p$-parabolic operators in Sections \ref{sec:FN-par} and \ref{sec:p-par}, respectively. In Section~\ref{sec:mix}, we discuss mixing operators and present several open questions.
Finally, Appendix~\ref{appen:par-p-Lapl-reg} provides detailed regularity results for solutions in the $p$-parabolic case, covering both singular and degenerate regimes.

\subsection{Notation}\label{sec:notation1}

For the reader's convenience, we gather here a list of notation to be used throughout this paper.

\begin{longtable}[l]{l l}

\,\,\,\,\,$\Omega$ &\qquad A bounded  domain  in $\mathbb{R}^n$ ($n\geq 1$) \\[2mm]

\,\,\,\,\,$\mathbb{R}_+$ &\qquad $\{t\in\mathbb{R}:\ t\ge0\}$ \\[2mm]
\,\,\,\,$W^{1,p}(\Omega)$ &\qquad Sobolev space of functions with weak derivatives in $L^p$ \\[2mm]
\,\,\,\,\,$B_r(x_0)$ &\qquad Euclidean ball of radius $r$ centered at $x_0$ \\[2mm]
\,\,\,\,\,$S(n)$ &\qquad The space of $n\times n$ symmetric matrices \\[2mm]
\,\,\,\,\,$e_i$&\qquad Eigenvalues of $M\in S(n)$\\[2mm]

\,\,\,\,\,$\mathcal{M}^+_{\lambda,\Lambda}(M)$ &\qquad Pucci maximal operator, $\displaystyle \Lambda \sum_{e_i > 0} e_i + \lambda \sum_{e_i < 0} e_i$,\quad($\Lambda\geq\lambda> 0$ are constants.)  \\[2mm]

\,\,\,\,\,$\mathcal{M}^-_{\lambda,\Lambda}(M)$ &\qquad Pucci minimal operator,  $\displaystyle \lambda \sum_{e_i > 0} e_i + \Lambda \sum_{e_i < 0} e_i$,\quad ($\Lambda\geq\lambda> 0$ are constants.) \\[2mm]

\,\,\,\,\,$\nabla u$ &\qquad Spatial gradient of $u$ (also in the parabolic setting) \\[2mm]
\,\,\,\,\,$\Delta_p u$ &\qquad  $p$-Laplacian, $\operatorname{div}(|\nabla u|^{p-2}\nabla u)$, $p\in(1,\infty)$ \\[2mm]

\,\,\,\,\,$K_{g,p,\Omega}$ &\qquad $\{w\in W^{1,p}(\Omega):\ w=g \text{ on }\partial\Omega\}$ \\[2mm]

\,\,\,\,\,$u\vee v$ &\qquad $\max\{u,v\}$ \\[2mm]

\,\,\,\,\,$u\wedge v$ &\qquad $\min\{u,v\}$ \\[2mm]

\,\,\,\,\,$X $ &\qquad $(x ,t )\in\mathbb{R}^{n+1}$ \\[2mm]

\,\,\,\,\,$\Omega_T$ &\qquad $\Omega\times(-T,0]$, a space-time cylinder \\[2mm]

\,\,\,\,\,$\partial_p\Omega_T$ &\qquad $(\partial\Omega\times[-T,0])\cup (\Omega\times\{-T\})$  (parabolic boundary) \\[2mm]

\,\,\,\,\,$Q_r^\theta(X_0)$ &\qquad $B_r(x_0)\times(t_0-r^\theta,t_0]$, for $r>0$, $\theta>0$ \\[2mm]

\,\,\,\,\,$Q_r^2(X_0)$ &\qquad $B_r(x_0)\times(t_0-r^2,t_0]$ \\[2mm]

\,\,\,\,\,$d(X,Y)$
&\qquad
$\max\{|x-y|,\ |t-s|^{1/2}\}$; \quad  $X=(x,t)$ and $Y=(y,s)$ (parabolic distance)\\[2mm]
\,\,\,\,\,$A_t$ &\qquad $\{x\in\Omega:\ (x,t)\in A\}$, time slice of $A\subset\Omega_T$ \\[2mm]

\,\,\,\,\,$\partial_t u$ &\qquad Time derivative of $u$ \\[2mm]

\,\,\,\,\,$K_{g,p,\Omega_T}$ &\qquad $\{w\in W^{1,p}(\Omega_T):\ w=g \text{ on }\partial_p\Omega_T\}$ \\[2mm]

\end{longtable}


\section{The Elliptic Case}\label{sec:ell}
To state our main result in the elliptic case, we first introduce the notation and definitions that will be used throughout the paper. The meta-theorem in Section~\ref{sec:el-meta} provides the conceptual framework for the proof strategy. While the logic of the proof is simple and relies on the assumed structural properties of the operators, applying the meta-theorem to specific cases requires substantial work, namely verifying that these properties hold. The verification of these properties constitutes one of the main contributions of this paper. In Section~\ref{sec:FN-elliptic}, we apply the meta-theorem to fully nonlinear elliptic operators, while in Section~\ref{sec:p-lap-elliptic}, we treat the $p$-Laplacian case.

\subsection{An elliptic meta-theorem}\label{sec:el-meta}

Let $\mathcal H$ denote the class of functions
$h:\mathbb R\to\mathbb R_+$ satisfying, for some constants $-1<a<1$, $C^0>0$, 
\begin{align}
\label{eq:rhs-1}
\begin{cases}
h=0 & \text{on } (-\infty,0],\\
h \text{ is positive and continuous} & \text{on } (0,\infty),\\
h(t)\le C^0 t^a & \text{for } 0<t<M_0,\\
h \text{ is bounded}&\text{on }[M_0,\infty).
\end{cases}
\end{align}
Here $M_0$ is the supremum bound for the boundary values, as defined earlier.

\medskip

Let $\mathcal K$ denote the class of functions
$k:\mathbb R\to\mathbb R_+$ satisfying
\begin{align}
\label{eq:rhs-2}
\begin{cases}
k=0 & \text{on } (-\infty,0],\\
k \text{ is non-decreasing and continuous} & \text{on } \mathbb R.  
\end{cases}
\end{align}

\medskip

Since the nonlinearity $h$ may be singular near the origin, we introduce a class
of admissible regularization functions that will be used in the approximation scheme. We define $\widetilde{\mathcal H}$ to be the class of functions
$\tilde h:\mathbb R\to\mathbb R_+$ satisfying
\begin{align}
\label{eq:rhs-1-1}
\begin{cases}
\tilde h=0 & \text{on } (-\infty,\tilde c],\\
\tilde h \text{ is bounded and Lipschitz} & \text{on } \mathbb R,\\
0 < \tilde h(t)\le C^0t^a+2 & \text{for } t\in
[\tilde c ,\infty),
\end{cases}
\end{align}
for some constant $\tilde c>0$.

\begin{remark}
Since $h$ satisfies $h(t)\le C^0t^a$ for $t<M_0$, the condition $\tilde h(t)\le C^0t^a+2$
should be viewed as a compatibility condition between the singular nonlinearity $h$
and its regularization $\tilde h$.
It  is used in the verification of the solvability property \ref{cond-sol} introduced below, for the
concrete operators considered later in the paper. In particular, this domination
condition is employed in the construction of barriers and in the derivation of uniform
scalar estimates for the regularized problems.
\end{remark}

\medskip

We denote by $\mathcal F(\Omega)$ a class of non-negative functions in $\Omega$, which will be specified depending on the operator under consideration. 
Given an operator $\cL:\cF(\Omega)\to \mathcal{A}(\Omega;\R_+)$, where $\cA(\Omega;\R_+)$ is the space of nonnegative measurable functions defined in $\Omega$,  we define the following structural properties on the operator $\cL$.

\begin{enumerate}[label={\bf (PE\arabic*)}]
    \item\label{cond-sol} {\bf Solvability:}  For any $ g\in C(\partial\Omega;\R_+)$ , $f\in L^\infty(\Omega;\R_+)$ and $\tilde h\in\widetilde{\mathcal H}$, there exists a solution $w\in \cF(\Omega)$ of the Dirichlet problem 
    \begin{align}\label{eq:sol-Dir}
        \begin{cases}
            \cL w=f\tilde h(w)&\text{in }\Omega,\\
            w=g&\text{on }\partial\Omega,
        \end{cases}
    \end{align}
    which satisfies, for some  constant $0<\al<1$,  
    the following regularity estimates: for any $\Omega'\Subset\Omega$,
    \begin{align}\label{eq:mod-conti}
        |w(x_1)-w(x_2)|\le\rho(|x_1-x_2|)\quad\text{for any }x_1,x_2\in\Omega'
    \end{align}
    and\footnote{The one sided estimate will mainly be used for points $x$ on the free boundary, i.e., the boundary of the support of $w$. 
    Since $w\geq 0$ in our paper, we expect  that $w$ is reasonably smooth at the boundary of its support. It is also noteworthy that the inequality may seems ad-hoc if $|\nabla w (x) | > 0$. However, in this case one uses  the fact that $x \in \Omega'$, and  
    $\sup_{B_r(x)\cap\Omega'}w\le Cw(x)$, where $C$ is large enough, depending on the distance to the boundary and the maximum value of $w $ on $\Omega$. }
    \begin{align}
        \label{eq:ftn-space}
        \sup_{B_r(x)\cap\Omega'}w\le C(r^{1+\al}+w(x))\quad\text{for any }x\in \Omega'\text{ and }r>0,
    \end{align}  
        where  $\rho$ is a modulus of continuity and $C>0$ is a constant, both depending only on $\Omega,\Omega',\cL, C^0,a,\|f\|_\infty$ and $\|g\|_\infty$. Moreover, there exists a function $w_*\in \text{Lip}(\Omega)\cap C(\overline{\Omega})$, depending only on $h,\|f\|_\infty, g$ and $\Omega$, such that
        \begin{align}\label{eq:lower-bound}
            \begin{cases}
                w_*\le w&\text{in }\Omega,\\
                w_*=g&\text{on }\partial\Omega.
            \end{cases}
        \end{align}

    \medskip

    \item\label{cond-comp} {\bf Comparison:} Let $g\in C(\partial\Omega;\R_+)$ and $f_1,f_2\in L^\infty(\Omega;\R_+)$, and assume that $\tilde h$ satisfies \eqref{eq:rhs-1-1}. Suppose $w_1$ and $w_2$ are solutions of
    \begin{align*}
        \begin{cases}
            \cL w_1=f_1\tilde h(w_1)&\text{in }\Omega,\\
            w_1=g&\text{on }\partial\Omega,
        \end{cases}\qquad
        \begin{cases}
            \cL w_2=f_2\tilde h(w_2)&\text{in }\Omega,\\
            w_2=g&\text{on }\partial\Omega,
        \end{cases}
    \end{align*}
    respectively. 
   If $f_1<f_2$ in $\Omega$, then $w_1\ge w_2$ in $\Omega$.

   \medskip

    \item\label{cond-conv} {\bf Convergence:} Let $g\in C(\partial\Omega;\R_+)$ and for each $j\in \N$, let    $\tilde h^j\in \widetilde{\mathcal H}, \,f^j\in L^\infty(\Omega;\mathbb R_+), $ and $w_j\in\mathcal F(\Omega),$
satisfy
    $$
    \begin{cases}
        \cL w_j=f^j\tilde h^j(w_j)&\text{in }\Omega\cap\{w_j>0\},\\
        w_j=g&\text{on }\partial\Omega.
    \end{cases}
    $$ 
    Suppose that
$\sup_{j\in\mathbb N}\|f^j\|_{L^\infty(\Omega)}<\infty,$  $f^j\to f$ and $w_j\to w$ locally uniformly in $\Omega$ for some $f\in L^\infty(\Omega)$ and $w\in C(\Omega;\R_+)$ and that $\tilde h^j\to  \hat  h $ locally uniformly on $(0,\infty)$ for some $\hat h\in  \cH\cup\widetilde\cH$.
Then
    \begin{align*}
            w\in \cF(\Omega)\quad\text{and}\quad  \cL w=f\hat h(w)\,\,\text{ in }\Omega\cap\{w>0\}.
    \end{align*}
\end{enumerate}

While we generally expect the above comparison \ref{cond-comp} to hold for $f_1 \leq f_2$, the necessity of the strict inequality  is primarily dictated by the current proof structure of \ref{cond-comp}, which, in the context of the $p$-Laplacian, appears to rely essentially on this condition.

Our meta-theorem for the elliptic case is the following.

\begin{theorem}\label{thm:meta}[Elliptic Meta-theorem]
    Let $\Omega$ be a bounded domain in $\R^n$. For $i=1,2$ and $C^1>c^1>0$, let $g_i\in C(\partial\Omega;\R_+)$ with $c^1<g_i<C^1$ and $f_i\in L^\infty(\Omega;\R_+)$ and assume that  $h_i \in \mathcal{H}$,  $k_i \in \mathcal{K},$ and $\cL_i$ satisfies \ref{cond-sol}-\ref{cond-conv} in a function space $\cF^i(\Omega)$. Then there exists a solution pair $(u,v)\in \cF^1(\Omega)\times \cF^2(\Omega)$ of the system \eqref{eq:sol-2}.
\end{theorem}

Next, we state a result on the relation between the coincidence sets, namely the region where the solutions vanish. 
For this purpose, we impose the following additional conditions on $f,k, h \text{ and } \tilde h$:
\begin{align}
    \label{eq:rhs-add}
    \begin{cases}
        \text{$\inf_\Omega f>0$  },\\
        \text{$k>0$ on $(0,\infty)$, }\\
        \text{For some $0<b<1$ and $c^0>0$, we have $h(t), \tilde h(t)\ge c^0t^b$ for $0<t<M_0$.   }
    \end{cases}
\end{align}
In addition, we need to  define the following further properties on the operator $\cL$: 
\begin{enumerate}[label={\bf (PE\arabic*')}]
    \item\label{cond-Hopf} {\bf Hopf's boundary principle:} For any ball $B_r(x^0)\Subset \Omega$, if $w\in C^1(\overline{B_r(x^0)})$ satisfies that $\cL w=0$ and $w>0$ in $B_r(x^0)$ and that $w(z^0)=0$ for some $z^0\in \partial B_r(x^0)$, then $|\D w(z^0)|\neq0$.
    \item\label{cond-nondeg}{\bf Nondegeneracy:} Let $w$ be a solution of \eqref{eq:sol-Dir} in $\Omega$. If $x^0\in \overline{\{w>0\}}$ and $B_r(x^0)\subset\Omega$, then
    $$
    \sup_{B_r(x^0)}w\ge cr^{b_0} \quad \hbox{for some } b_0 > 0, 
    $$
    where $c>0$ is a constant depending only on $b,c^0,n,\cL,\inf_\Omega f$. Here, $c\to\infty$ as $\inf_\Omega f\to\infty$.

\item\label{cond-conv-2}{\bf Regularity:} For $g\in C(\partial\Omega;\R_+)$, let $w^*$ be a solution of 
\begin{align*}
    \begin{cases}
        \cL w^*=0&\text{in }\Omega,\\
        w^*=g&\text{on }\partial\Omega.
    \end{cases}
\end{align*}
    Then $w^*\in C(\overline{\Omega})$ with a modulus of continuity depending only on $\Omega,\cL$ and $\|g\|_\infty$. Moreover,  $\|w^*\|_{L^\infty(\Omega)}\to0$ as $\|g\|_{L^\infty(\partial\Omega)}\to0$.
\end{enumerate}

\begin{theorem}\label{thm:support}Let $(u,v)$ be a solution of the system \eqref{eq:sol-2} given by Theorem \ref{thm:meta}.
Assume in addition that for  $i=1,2$, $f_i, k_i,h_i$ satisfy \eqref{eq:rhs-add} and $\cL_i$ satisfies the further conditions \ref{cond-Hopf}-\ref{cond-conv-2}. 
Let $C_u$ and $C_v$ be connected components of $\{u=0\}^\circ$ and $\{v=0\}^\circ$, respectively. Then
\[
\text{either } C_u = C_v \quad \text{or} \quad C_u \cap C_v = \emptyset.
\]
Moreover, for any $k_i,h_i$ and $\cL_i$, there exist $(f_i,g_i)$ such that for a solution $(u,v)$, one nonempty connected component of $\{u=0\}^{\mathrm{o}}$ and one of 
$\{v=0\}^{\mathrm{o}}$ are disjoint.
\end{theorem}

It should be remarked that unlike in \cite{ELSH25}, where the coincidence sets of solutions coincide, in our setting the sets $\{u=0\}^{\mathrm{o}}$ and $\{v=0\}^{\mathrm{o}}$ may  occur as nonempty and disjoint sets. This difference is mainly due to the behavior of the right-hand side of the first equation in \eqref{eq:sol-2}. In \cite{ELSH25}, this term blows up as $v\to0$; see \cite[Assumption 1 (A2)]{ELSH25}. In contrast, in our case the right-hand side remains bounded near $v=0$.

For the proof of Theorem \ref{thm:meta}, we need an existence theory in the case where the right hand sides are reasonably smooth.
This is shown in the following proposition.

\begin{proposition}\label{prop:meta}
    For $i=1,2$, let $\Omega$,  $g_i$, $f_i$, $k_i$, $h_i$,  $\mathcal{F}^i(\Omega)$, and
$\mathcal{L}_i$ be as in Theorem~\ref{thm:meta}, and assume that 
$\tilde h_i\in \widetilde{\mathcal{H}}$.
    Then there exists a solution $(u,v)\in\cF^1(\Omega)\times\cF^2(\Omega)$ of the system
    \begin{align*}
        \begin{cases}
            \cL_1u=f_1k_1(v)\tilde h_1(u)&\text{in }\Omega\cap\{u>0\},\\
            \cL_2v=f_2k_2(u)\tilde h_2(v)&\text{in }\Omega\cap\{v>0\},\\
            u=g_1,\,\,\, v=g_2&\text{on }\partial\Omega.
        \end{cases}
    \end{align*}
    Moreover,  $(u,v)$ satisfies the estimates \eqref{eq:mod-conti} and \eqref{eq:ftn-space}.
\end{proposition}

\begin{proof}
\medskip\noindent\emph{Step 1.}
We claim that for any $0<\e<1$, there is  a solution $(u^\e,v^\e)\in \cF^1(\Omega)\times \cF^2(\Omega)$ of the $\e-$
system
\begin{align}
    \label{eq:sol-eps}
    \begin{cases}
        \cL_1 u^\e=(f_1k_1(v^\e)+\e)\tilde h_1(u^\e)&\text{in }\Omega\cap\{u^\e>0\},\\
            \cL_2 v^\e=(f_2k_2(u^\e)+\e)\tilde h_2(v^\e)&\text{in }\Omega\cap\{v^\e>0\},\\
            u^\e=g_1, v^\e=g_2&\text{on }\partial\Omega.
    \end{cases}
\end{align}
To prove this, let $u^*\in \cF^1(\Omega)$ and $v^*\in \cF^2(\Omega)$ be solutions of
\begin{align}\label{eq:sol-hom}
    \begin{cases}
        \cL_1u^*=0&\text{in }\Omega,\\
        u^*=g_1&\text{on }\partial\Omega,
    \end{cases}
    \qquad \begin{cases}
        \cL_2v^*=0&\text{in }\Omega,\\
        v^*=g_2&\text{on }\partial\Omega,
    \end{cases}
\end{align}
respectively. 
We let $\{\sigma_j\}_{j=0}^\infty$ be a strictly decreasing sequence of real numbers in the open interval $(1,2)$ converging to $1$ and $\{\tau_j\}_{j=0}^\infty$ be a strictly increasing sequence in the interval $(1/2,1)$ converging to $1$. We first use $v^*$ to find  $u_0$, which in turn is used to find $v_0$.
 Indeed, we let  $(u_0,v_0)\in \cF^1(\Omega)\times \cF^2(\Omega)$ solve  two separate  scalar problems
\begin{align*}
    \begin{cases}
        \cL_1 u_0=(f_1k_1(v^*)+\e \sigma_0)\tilde h_1(u_0)&\text{in }\Omega,\\
        \cL_2 v_0=(f_2k_2(u_0)+\e \tau_0)\tilde h_2(v_0)&\text{in }\Omega,\\
        u_0=g_1,v_0=g_2&\text{on }\partial\Omega.
    \end{cases}
\end{align*}
For the existence, we first find $u_0\in \cF^1(\Omega)$ using \ref{cond-sol}, and then find $v_0\in \cF^2(\Omega)$ in the same way.

In a similar way, we find $(u_1,v_1)\in \cF^1(\Omega)\times \cF^2(\Omega)$ which solves
\begin{align*}
    \begin{cases}
        \cL_1 u_1=(f_1k_1(v_0)+\e \sigma_1)\tilde h_1(u_1) &\text{in }\Omega,\\
        \cL_2 v_1=(f_2k_2(u_1)+\e \tau_1)\tilde h_2(v_1) &\text{in }\Omega,\\
        u_1=g_1,v_1=g_2&\text{on }\partial\Omega.
    \end{cases}
\end{align*}

By repetition, we get a sequence $\{(u_j,v_j)\}_{j=1}^\infty\subset \cF^1(\Omega)\times \cF^2(\Omega)$ solving 
\begin{align}\label{eq:sol-seq}
\begin{cases}
        \cL_1 u_j=(f_1k_1(v_{j-1})+\e \sigma_j)\tilde h_1(u_j) &\text{in }\Omega,\\
        \cL_2 v_j=(f_2k_2(u_j)+\e \tau_j)\tilde h_2(v_j) &\text{in }\Omega,\\
        u_j=g_1,v_j=g_2&\text{on }\partial\Omega.
    \end{cases}
\end{align}

We claim that
\begin{align}
    \label{eq:mon}
    0\le u_0\le u_1\le\ldots\le u^*\quad\text{and}\quad v^*\ge v_0\ge v_1\ge\ldots\ge0\quad\text{in }\Omega.
\end{align} 
Indeed, 
by \ref{cond-comp} we have $u_j\le u^*$ and $v_j\le v^*$ in $\Omega$ for any $j\ge0$. 
In particular, we have $v^*\ge v_0$, thus $f_1k_1(v^*)+\e \sigma_0>f_1k_1(v_0)+\e \sigma_1$ in $\Omega$ due to the monotonicity of $k_1$ and $\sigma_j$ and the nonnegativity of $f_1$, and therefore $u_0\le u_1$ in $\Omega$ by \ref{cond-comp}. In turn, this gives $f_2k_2(u_0)+\e \tau_0<f_2k_2(u_1)+\e \tau_1$ in $\Omega$, thus $v_0\ge v_1$ in $\Omega$ by \ref{cond-comp}. Using this inequality, we can argue as above to get $u_2\ge u_1$ in $\Omega$, which will imply by a similar reason $v_1\ge v_2$ in $\Omega$. Repeating this process, we obtain \eqref{eq:mon}.

By \eqref{eq:mon}, we have that $u_j\to u^\e$ and $v_j\to v^\e$ pointwise in $\Omega$ for some $u^\e,v^\e:\Omega\to \R_+$. Moreover,  $f_1k_1(v_{j-1})+\e \sigma_j\le f_1k_1(v^*)+2\e$ and $f_2k_2(u_j)+\e \tau_j\le f_2k_2(u^*)+\e$ in $\Omega$ for any $j\in \N$. These, along with \eqref{eq:mod-conti}, imply by Arzelà–Ascoli theorem that over a subsequence, $u_j$ and $v_j$ converge locally uniformly to $u^\e$ and $v^\e$, respectively. 
Furthermore, since $k_1$ and $k_2$ are continuous on $\R$, $f_1k_1(v_{j-1})+\e \sigma_j\to f_1k_1(v^\e)+\e$ and $f_2k_2(u_j)+\e \tau_j\to f_2k_2(u^\e)+\e$ locally uniformly in $\Omega$. Thus, by \ref{cond-conv}, $(u^\e,v^\e)$ belongs to $\cF^1(\Omega)\times \cF^2(\Omega)$ and solves the system
\begin{align*}
    \begin{cases}
        \cL_1 u^\e=(f_1k_1(v^\e)+\e)\tilde h_1(u^\e) &\text{in }\Omega\cap\{u^\e>0\},\\
        \cL_2 v^\e=(f_2k_2(u^\e)+\e)\tilde h_2(v^\e) &\text{in }\Omega\cap\{v^\e>0\}.
    \end{cases}
\end{align*}
To prove the boundary conditions $u^\e=g_1$ and $v^\e=g_2$ on $\partial\Omega$, we note that for any $j\in \mathbb{N}$,
$$
f_1k_1(v_{j-1})+\e\sigma_j\le \|f_1\|_\infty k_1(\|v^*\|_\infty)+2\quad\text{in }\Omega.
$$
By \eqref{eq:lower-bound}, there exists a function $u_*$ in $\Omega$ which satisfies that for any $j\in \mathbb{N}$
\begin{align*}
    \begin{cases}
        u_*\le u_j&\text{in }\Omega,\\
        u_*=g_1&\text{on }\partial\Omega.
        \end{cases}
\end{align*}
Thus $u_*\le u^\e$ in $\Omega$. From $u_*\le u^\e\le u^*$ in $\Omega$ and $u_*=u^*=g_1$ on $\partial\Omega$, we infer $u^\e=g_1$ on $\partial\Omega$. Similarly, we can find $v_*$ that satisfies $v_*\le v^\e$ in $\Omega$ and obtain $v^\e=g_2$ in $\partial\Omega$.

\medskip\noindent\emph{Step 2.} Let $u^*$ and $v^*$ be as defined in \eqref{eq:sol-hom}. For each $0<\e<1$, let $(u^\e,v^\e)$ be as in \eqref{eq:sol-eps}. Note that $u^\e\le u^*$ and $v^\e\le v^*$ in $\Omega$. These inequalities, combined with the monotonicity of $k_1$ and $k_2$, imply that we  have  uniform bounds on the right hand side: $f_1k_1(v^\e)+\e\le f_1k_1(\|v^*\|_\infty)+1$ and $f_2k_2(u^\e)+\e\le f_2k_2(\|u^*\|_\infty)+1$ in $\Omega$. Due to \eqref{eq:mod-conti}, we have by Arzelà–Ascoli theorem that over a subsequence, $u^\e\to u$ and $v^\e\to v$ locally uniformly in $\Omega$ for some $u,v\in C(\Omega)$.
From the continuity of $k_i$ and the convergence $v^\e \to v$ and $u^\e\to u$, we see that
\[
f_1 k_1(v^\e) + \e \to f_1 k_1(v)\quad\text{and}\quad f_2k_2(u^\e)+\e\to f_2k_2(u)
\]
locally uniformly in $\Omega$. 
Therefore, by \ref{cond-conv}, we get
\begin{align*}
    \begin{cases}
        (u,v)\in \cF^1(\Omega)\times \cF^2(\Omega),\\
        \cL_1u=f_1k_1(v)\tilde h_1(u)&\text{in }\Omega\cap\{u>0\},\\
        \cL_2v=f_2k_2(u)\tilde h_2(v)&\text{in }\Omega\cap\{v>0\}.
    \end{cases}
\end{align*}
Regarding the boundary conditions, we observe that for any $0<\e<1$
$$
\text{$u_*\le u^\e\le u^*$ and $v_*\le v^\e\le v^*$ in $\Omega$},
$$
thus
$$
\text{$u_*\le u\le u^*$ and $v_*\le v\le v^*$ in $\Omega$},
$$
hence $u=g_1$ and $v=g_2$ on $\partial\Omega$.
\end{proof}

\begin{proof}[Proof of Theorem~\ref{thm:meta}]
Let $\{\tilde h_1^i\}_{i\in\mathbb N}\subset\widetilde\cH$ and $\{\tilde h_2^i\}_{i\in\mathbb N}\subset \widetilde\cH$
be sequences of admissible regularizations such that
\[
\tilde h_1^i \to h_1,
\qquad
\tilde h_2^i \to h_2
\quad\text{locally uniformly on }(0,\infty).
\]
In view of Proposition~\ref{prop:meta}, there exists a solution $(u^i,v^i)\in \cF^1(\Omega)\times \cF^2(\Omega)$ of
\begin{align*}
    \begin{cases}
        \cL_1u^i=f_1k_1(v^i)\tilde h_1^i(u^i)&\text{in }\Omega\cap\{u^i>0\},\\
        \cL_2v^i=f_2k_2(u^i)\tilde h_2^i(v^i)&\text{in }\Omega\cap\{v^i>0\},\\
        u^i=g_1, v^i=g_2&\text{on }\partial\Omega.
    \end{cases}
\end{align*} 
It is easily seen that $(u^i,v^i)$ satisfies the estimates \eqref{eq:mod-conti} and \eqref{eq:ftn-space} with constants $C>0$ and $0<\al<1$ and a modulus of continuity $\rho$, which are independent of $i\in\N$. 
By these uniform estimates and \ref{cond-conv}, we have that over a subsequence, $(u^i,v^i)\to (u,v)\in \cF^1(\Omega)\times\cF^2(\Omega)$ and
\begin{align*}
    \begin{cases}
        \cL_1u=f_1k_1(v)h_1(u)&\text{in }\Omega\cap\{u>0\},\\
        \cL_2v=f_2k_2(u)h_2(v)&\text{in }\Omega\cap\{v>0\},\\
        u=|\D u|=0&\text{on }\Omega\cap\partial\{u>0\},\\
        v=|\D v|=0&\text{on }\Omega\cap\partial\{v>0\}.
    \end{cases}
\end{align*}

Note that the free-boundary gradient condition follows from the growth estimate \eqref{eq:ftn-space}.
It remains to prove the boundary conditions $u=g_1$ and $v=g_2$ on $\partial\Omega$. To prove this we  let $u^*,v^*,u_*$ and $v_*$ be as in the proof of Proposition~\ref{prop:meta}. For every $i\in\mathbb{N}$,
$$
\text{$u_*\le u^i\le u^*$ and $v_*\le v^i\le v^*$ in $\Omega$,}
$$
thus
$$
\text{$u_*\le u\le u^*$ and $v_*\le v\le v^*$ in $\Omega$},
$$
hence $u=g_1$ and $v=g_2$ on $\partial\Omega$. This completes the proof.
\end{proof}

For the proof of Theorem \ref{thm:support}, we need the following  elementary topological lemma.

\begin{lemma}
    \label{lem:top}
For closed sets $F$ and $G$ in $\R^n$, let $A$ be a connected component of $F^{\mathrm{o}}$ and $B$ be that of $G^{\mathrm{o}}$. If $A\subset\overline B$ and $B\subset\overline A$, then $A=B$.    
\end{lemma}

\begin{proof}
Since $A$ is open, we have $A=A^{\mathrm{o}}\subset\overline A^{\mathrm{o}}$. Moreover, from the inclusion of $A\subset F$ and the closedness of $F$, we get $\overline A^{\mathrm{o}}\subset\overline A\subset \overline F=F$. Thus, $\overline A^{\mathrm{o}}$ is a connected open set containing $A$ contained in $F$, hence $\overline A^{\mathrm{o}}=A$ by the definition of $A$. Similarly, we have $\overline B^{\mathrm{o}}=B$. In addition, from $A\subset \overline B$ and $B\subset\overline A$, we have $\overline A\subset \overline B$ and $\overline B\subset\overline A$, hence $\overline A=\overline B$. Therefore, $A=\overline A^{\mathrm{o}}=\overline B^{\mathrm{o}}=B$.    
\end{proof}

\begin{proof}[Proof of Theorem~\ref{thm:support}]
We divide the proof into two steps.

\medskip\noindent\emph{Step 1.}
We first prove a dichotomy, namely  either $C_u=C_v$ or $C_u\cap C_v=\emptyset$. Indeed, towards a contradiction, suppose $C_u \cap C_v\neq \emptyset$ but $C_u\neq C_v$. From Lemma~\ref{lem:top}, we see that at least one of $C_v\setminus\overline{C_u}\neq\emptyset$ and $C_u\setminus\overline{C_v}\neq\emptyset$ occurs. Without loss of generality, we may assume that the open set $C_v\setminus\overline{C_u}$ is nonempty. By this and $C_u\cap C_v\neq\emptyset$, we can take a ball $B_r(x^0)\subset C_v\setminus \overline{C_u}$ such that $\overline{B_r(x^0)}\cap\overline{C_u}$ contains a point, say $z^0$. Note that $z^0\in \partial\overline{C_u}$. We further have
$$
z^0\in \partial\overline{C_u}\subset \partial C_u\subset\partial\{u=0\}^{\mathrm{o}}\subset\partial\{u=0\}=\partial\{u>0\}.
$$

Moreover, since $u=g_1>0$ on $\partial\Omega$, the free boundary $\partial\{u>0\}$ is contained in $\Omega$, and hence $z^0\in \Omega\cap\partial\{u>0\}$.
In particular, by the free boundary condition, $|\nabla u(z^0)|=0$. On the other hand, since $B_r(x^0)\subset C_v\subset\{v=0\}$, we have $v=0$ in $B_r(x^0)$,
and therefore $k_1(v)=0$ in $B_r(x^0)$. It follows that $
\mathcal L_1 u = f_1 k_1(v) h_1(u)=0 \quad \text{in } B_r(x^0).$
Since $u>0$ in $B_r(x^0)$ and $u(z^0)=0$, Hopf's lemma (condition \ref{cond-Hopf}) yields $|\nabla u(z^0)|\neq 0$, which is a contradiction.

\medskip\noindent\emph{Step 2.} In this step, we prove the second statement of Theorem~\ref{thm:support}. 
Using the result of step 1, it suffices to construct $(f_i,g_i)$ such that $\{u=0\}^\circ$ and $\{v=0\}^\circ$ are nonempty and distinct. Let $g_i$ and $f_i$, $i=1,2$, be given that will be determined later. We focus on the relations between them. In view of the proofs of Proposition~\ref{prop:meta} and Theorem~\ref{thm:meta},   we can find upper and lower barriers $u^*,v^*,u_*,v_*$. To be more specific, $u^*$ and $v^*$ are  solutions of \eqref{eq:sol-hom}
satisfying $u\le u^*$ and $v\le v^*$ in $\Omega$.  Moreover, $u_*,v_*  \in \text{Lip}(\Omega)\cap C(\overline{\Omega})$ are obtained from \eqref{eq:lower-bound},  with $u_*$ depending  only on $h_1,\|f_1\|_\infty,g_1,k_1(\|v^*\|_\infty)$, and $ v_* $ depending only on  on $h_2,\|f_2\|_\infty, k_2(\|u^*\|_\infty),g_2$, and satisfying 
\begin{align*}
    \begin{cases}
        u_*\le u&\text{in }\Omega,\\
        u_*=g_1&\text{on }\partial\Omega,
    \end{cases}\qquad
    \begin{cases}
        v_*\le v&\text{in }\Omega,\\
        v_*=g_2&\text{on }\partial\Omega.
    \end{cases}
\end{align*}

For a fixed point $x^0\in \partial\Omega$, we have $v\ge v_*\ge \frac{g_2(x^0)}2$ in $B_{\e_0}(x^0)\cap \Omega$ for a small $\e_0>0$, depending only on $h_2,\|f_2\|_\infty,  k_2(\|u^*\|_\infty),g_2$. It follows that 
$$
\cL_1u\ge \inf_{\Omega}f_1\,k_1\left(\frac{g_2(x^0)}2\right)h_1(u)\quad\text{in }\Omega\cap B_{\e_0}(x^0)\cap\{u>0\}.
$$
Since $u\le u^*$ in $\Omega$ and $\e_0$ and $u^*$ are independent of $f_1$, if $\inf_{\Omega}f_1\ge C(f_2,g_1,g_2)$, then we have by \ref{cond-nondeg} that $u=0$ in an open set inside $B_{\e_0}(x^0)\cap \Omega$. This proves that $\{u=0\}^{\mathrm{o}}\neq \emptyset$ with $\{u=0\}^{\mathrm{o}}\neq \{v=0\}^{\mathrm{o}}$ since $v>0$ in $B_{\e_0}(x^0)\cap \Omega$.

Next, we show that $\{v=0\}^{\mathrm{o}}\neq\emptyset$. We assume to the contrary $v>0$ in $\Omega$. Then, by its continuity, $\inf_\Omega v>0$. We take a point $y^0\in \partial\Omega\setminus B_{\e_0}(x^0)$. From $g_1(y^0)>c^1$, we see that $u\ge u_*\ge c^1/2$ in $B_{\rho_0}(y^0)\cap \Omega$ for small $\rho_0>0$, depending on $\|g_1\|_\infty$, by \ref{cond-sol}. Thus
$$
\cL_2v\ge\inf_{\Omega}f_2\,k_2\left(c^1/2 \right)h_2(v)\quad\text{in }B_{\rho_0}(y^0)\cap\Omega\cap \{v>0\}.
$$
From this and $\inf_{\Omega}v>0$, we have by \ref{cond-nondeg} that $\sup_\Omega v\ge \hat c$. Here, $\hat c>0$ depends on $\|g_1\|_\infty$ and $\inf_\Omega f_2$, but independent of $g_2$. Then $\sup_\Omega v^*\ge \sup_\Omega v\ge \hat c$. In view of \ref{cond-conv-2}, this is a contradiction if $\|g_2\|_\infty$ is small.

It remains to choose $f_i$ and $g_i$ appropriately. Recall that we required two conditions
$\inf_\Omega f_1\ge C(f_2,g_1,g_2)$ and $\|g_2\|_\infty\le c(g_1,f_2)$.
We first fix any $g_1$ and $f_2$ satisfying the the assumption stated in the theorem. We then select $g_2$ to satisfy the second condition, and finally choose $f_1$ so that the first condition holds.
\end{proof}


\subsection{The fully nonlinear elliptic case}
\label{sec:FN-elliptic}

We assume that a fully nonlinear operator $F:S(n)\to \R$ satisfies
\begin{align}
    \label{eq:assump-fully-nonlinear}
    \begin{cases}
    \text{$F$ is uniformly elliptic, i.e., there are constants $\Lambda\ge\la>0$ such that}\\
    \qquad \cM^-_{\la,\Lambda}(M-N)\le F(M)-F(N)\le \cM^+_{\la,\Lambda}(M-N)\text{ for every } M,N\in S(n),\\
    F(0)=0,\\
    \text{$F$ is concave},\\
    \text{$F$ is homogeneous of degree $1$, i.e., }    \  F(rM)=rF(M) \ \ \text{for all }   r>0, \ M\in S(n).
    \end{cases}
\end{align}
For $0<\al<1$, let $\cF_\al(\Omega;\R_+)$ be the space of functions $w\in C(\Omega;\R_+)$ satisfying \eqref{eq:mod-conti} and \eqref{eq:ftn-space}.

The following existence result is a slight modification of \cite[Theorem 2.1]{RicTei11}. We omit the proof as the generalization to our setting is straightforward.

\begin{theorem}(Compare \cite[Theorem 2.1]{RicTei11})
    \label{thm:exist-FN}
    Suppose $\Omega$ is a bounded domain in $\R^n$ and $\tilde h:\R\to\R$ be a bounded Lipschitz function. Assume that $F$ satisfies \eqref{eq:assump-fully-nonlinear} and that for $f\in L^\infty(\Omega;\R_+)$ and $g\in C(\partial\Omega)$, the equation 
    \begin{align}
        \label{eq:vis-sol}
        \begin{cases}
            F(D^2w)=f\tilde h(w)&\text{in }\Omega,\\
            w=g&\text{on }\partial\Omega
        \end{cases}
    \end{align}
    admits a viscosity subsolution $w_\sharp$ and a supersolution $w^\sharp$. Define
    \begin{align*}
        \mathcal{S}:=\{v\in C(\overline\Omega)\,:\, w_\sharp\le v\le w^\sharp\text{ and }v \text{ is a subsolution of }\eqref{eq:vis-sol}\}.
    \end{align*}
    Then 
    $$
    w(x):=\sup_{v\in \cS} v(x)
    $$
    is a viscosity solution of \eqref{eq:vis-sol}.
\end{theorem}

\begin{theorem}
    \label{thm:FN}
    Let $\Omega$ be a bounded $C^{1,1}$
      domain in $\R^n$, $n\ge1$. For $i=1,2$ and the constants $C^1>c^1>0$, suppose that $g_i\in \text{Lip}(\partial\Omega)$ with $c^1<g_i<C^1$, $f_i\in L^\infty(\Omega;\R_+)$,  $k_i\in \cK$ and $h_i\in \cH$. Suppose $F_1$ and $F_2$ satisfy \eqref{eq:assump-fully-nonlinear}. Then there exists a solution $(u,v)\in \cF_{\al_1}(\Omega;\R_+)\times  \cF_{\al_2}(\Omega;\R_+)$ of the system
    \begin{align}\label{system:fullynonlinear}
        \begin{cases}
            F_1(D^2u)=f_1k_1(v)h_1(u)&\text{in }\Omega\cap\{u>0\},\\
            F_2(D^2v)=f_2k_2(u)h_2(v)&\text{in }\Omega\cap\{v>0\},\\
            u=g_1,v=g_2&\text{on }\partial\Omega,\\
            u=|\D u|=0&\text{on }\Omega\cap\partial\{u>0\},\\
            v=|\D v|=0&\text{on }\Omega\cap\partial\{v>0\},
        \end{cases}
    \end{align}
    where $\al_i\in(0,1)$ is a constant depending only on $a_i$.
\end{theorem}

\begin{proof}
It suffices to show that the properties \ref{cond-sol}-\ref{cond-conv} hold with the operator $\cL w=F(D^2w)$.

\medskip

\noindent
\underline{Proof of Property  \ref{cond-sol}:}
We split the proof of \ref{cond-sol} into two steps.

\medskip\noindent\emph{Step 1.} The aim of this step is to construct $w_*$ for \eqref{eq:lower-bound} which serves as subsolution  of \eqref{eq:sol-Dir}.  For this purpose let $\tilde g$ be the solution of the homogeneous Pucci problem
\begin{align}\label{eq:g-tilde}
    \begin{cases}
        \cM^-_{\la,\Lambda}(D^2\tilde g)=0&\text{in }\Omega,\\
        \tilde g=g&\text{on }\partial\Omega.
    \end{cases}
\end{align}
Since $g$ is Lipschitz and $\Omega$ is $C^{1,1}$, the standard elliptic theory gives that $\tilde g\in \text{Lip}(\overline{\Omega})$ with
$$
\|\D\tilde g\|_{L^\infty(\Omega)}\le C(n,\la,\Lambda,\Omega).
$$
Note that $c^1<\tilde g<C^1$ in $\Omega$ by maximum and minimum principles.

For small $r\in(0,1)$ to be chosen later, we define
\begin{align*}
    &d_{\partial\Omega}(x):=\dist(x,\partial\Omega),\quad x\in \Omega,\\
    &\Omega_r:=\{x\in\Omega\,:\,0<d_{\partial\Omega}(x)<r\}.
\end{align*}
We consider
\begin{align*}
    w_*(x):=\begin{cases}
        \tilde g(x)\left(\frac{r-d_{\partial\Omega}(x)}{r}\right)^\be,& x\in \Omega_r,\\
        0,&x\in \Omega\setminus\Omega_r,
    \end{cases}
\end{align*} 
where $\be:=\frac2{1-a}\in(1,\infty)$. It is easily seen that $w_*=g$ on $\partial\Omega$ and $w_*>0$ in $\Omega_r$.
Moreover, a direct computation gives that in $\Omega_r$,
\begin{align}\label{eq:subsol-comp-1}\begin{split}
    &r^\be D^2w_*\\
    &= (r-d_{\partial\Omega})^\be D^2\tilde g-2\be(r-d_{\partial\Omega})^{\be-1}(\D \tilde g\otimes\D d_{\partial\Omega})+\be(\be-1)\tilde g(r-d_{\partial\Omega})^{\be-2}\D (d_{\partial\Omega}\otimes\D d_{\partial\Omega})\\
    &\quad-\be\tilde g(r-d_{\partial\Omega})^{\be-1}D^2d_{\partial\Omega}.
\end{split}\end{align}
By using the equation $\cM_{\la,\Lambda}^-(D^2\tilde g)=0$, we further have in $\Omega_r$
\begin{align}
    \label{eq:subsol-comp-2}\begin{split}
     r^\be F(D^2w_*)\ge& r^\be \cM_{\la,\Lambda}^-(D^2w_*)\\
     \ge&\be(\be-1)\tilde g(r-d_{\partial\Omega})^{\be-2}\cM_{\la,\Lambda}^-(\D d_{\partial\Omega}\otimes\D d_{\partial\Omega})-2\be(r-d_{\partial\Omega})^{\be-1}\cM_{\la,\Lambda}^+(\D \tilde g\otimes \D d_{\partial\Omega})\\
    &-\be\tilde g(r-d_{\partial\Omega})^{\be-1}\cM_{\la,\Lambda}^+(D^2d_{\partial\Omega}).
\end{split}\end{align}
Since the eigenvalues of $\D d_{\partial\Omega}\otimes\D d_{\partial\Omega}$ are $|\D d_{\partial\Omega}|^2,0,\cdots,0$ and $|\D d_{\partial\Omega}|\ge1$, we have $\cM_{\la,\Lambda}^-(\D d_{\partial\Omega}\otimes\D d_{\partial\Omega})\ge\la$. Moreover, $|\D d_{\partial\Omega}|\le C(n,\Omega)$ and $|D^2d_{\partial\Omega}|\le C(n,\Omega)$ in $\Omega_r$ if $r=r(n,\Omega)>0$ is small enough. This, along with the Lipschitz regularity of $\tilde g$, gives $\cM_{\la,\Lambda}^+(D^2d_{\partial\Omega})\le C(n,\la,\Lambda,\Omega)$ and $\cM_{\la,\Lambda}^+(\D g\otimes\D d_{\partial\Omega})\le C(n,\la,\Lambda,\Omega)$ in $\Omega_r$. Therefore, we have in $\Omega_r$
\begin{align*}
    r^\be F(D^2w_*)&\ge c(c^1,a,n,\la,\Lambda,\Omega)(r-d_{\partial\Omega})^{\be-2}-C(C^1,a,n,\la,\Lambda,\Omega)(r-d_{\partial\Omega})^{\be-1}\\
    &\ge c(c^1,C^1,a,n,\la,\Lambda,\Omega)(r-d_{\partial\Omega})^{\be-2}
\end{align*}
if $r=r(c^1,C^1,a,n,\la,\Lambda,\Omega)$ is small enough. Thus, we have in $\Omega_r$
\begin{align*}
    F(D^2w_*)\ge c\frac{(r-d_{\partial\Omega})^{\be-2}}{r^\be}=\frac{c}{r^2}\left(\frac{r-d_{\partial\Omega}}{r}\right)^{\be-2}\ge\frac{c}{r^2}
\end{align*}
and 
\begin{align*}
    F(D^2w_*)\ge c\frac{(r-d_{\partial\Omega})^{\be-2}}{r^\be}=cr^{\be(a-1)}\frac{(r-d_{\partial\Omega})^{\be-2}}{r^{\be a}}=cr^{\be(a-1)}\left(\frac{r-d_{\partial\Omega}}r\right)^{\be a}\ge cr^{\be(a-1)}w_*^a.
\end{align*}
Thus, 
choosing $r>0$ sufficiently small so that
$\frac{c}{r^2}\ge 4\|f\|_\infty,$
and using that $\beta(a-1)=-2$, we obtain
\begin{align*}
    F(D^2w_*)\ge f(C^0w_*^a+2)\ge f\tilde h(w_*)\quad\text{in }\Omega_r,
\end{align*} 
and hence  $w_*$ is a subsolution of \eqref{eq:sol-Dir} in $\Omega$.

\medskip\noindent\emph{Step 2.} In this step, we prove \ref{cond-sol} by utilizing the subsolution property of $w_*$. Taking the solution of the homogeneous Dirichlet problem as a supersolution for \eqref{eq:sol-Dir}, we obtain a solution $w$ of \eqref{eq:sol-Dir} by applying Theorem~\ref{thm:exist-FN}. Moreover, since Theorem~\ref{thm:exist-FN} tells us that $w$ is the largest subsolution of \eqref{eq:sol-Dir}, we infer $w\ge  w_*$ in $\Omega$.

It remains to prove the regularity properties \eqref{eq:mod-conti} and \eqref{eq:ftn-space}. Recall that $\tilde h(w)=0$ if $w\le\tilde c$ and let
$$
\hat w:=w^{\frac{1-a_-}2},\quad \text{where }a_-:=\min\{a,0\}\le0.
$$
By following computations (11) - (16) in \cite{AraTei13}, we get
$$
F\left(D^2\hat w+\frac{1+a_-}{1-a_-}\hat w^{-1}\D\hat w\otimes\D\hat w \right)=\hat f\hat w^{-1}\quad\text{in }\Omega,\quad\text{where } \hat f:=\left(\frac{1-a_-}2\right)fw^{-a_-}\tilde h(w).
$$
We claim that
$$
\hat f\le C(C^0,a,\|g\|_\infty)\|f\|_\infty.
$$
Indeed, in view of the definition of $\hat f$ and that $\hat f=0$ when $\hat w\le \tilde c^{\frac2{1-a_-} }$, it suffices to show that $\hat f\le C(C^0,a,\|g\|_\infty)$ in $\Omega\cap\{w>0\}$. From \eqref{eq:rhs-1-1}, we infer
$$
w^{-a_-}\tilde h(w)\le w^{-a_-}(C^0w^a+2) \quad\text{in }\Omega\cap\{w>0\}.
$$
If $a\le 0$, then we use $w\le \|g\|_\infty$ in $\Omega$, which follows from the maximum principle, to get
$$
w^{-a_-}\tilde h(w)\le C^0+2w^{-a}\le C^0+2\|g\|_\infty^{-a}.
$$
On the other hand, if $a>0$, then 
$$
w^{-a_{-}}\tilde h(w)\le C^0w^a+2\le C^0\|g\|_\infty^a+2.
$$

Now, with the above equation for $\hat w$ and the uniform boundedness of $\hat f$, we can proceed as in \cite{AraTei13}, in particular \cite[Proposition~1 and Theorem~3]{AraTei13}, to obtain \eqref{eq:mod-conti} and \eqref{eq:ftn-space}. In \eqref{eq:ftn-space}, we can take $\al=\frac{1+a}{1-a}\in(0,1)$ when $a\in(-1,0)$. When $a\in[0,1)$, $\al$ can be chosen as any positive number less than $1$.

\medskip

\noindent
\underline{Proof of Property \ref{cond-comp}:} 
 Suppose $w_1$ and $w_2$ satisfy
    \begin{align*}
        \begin{cases}
             F(D^2 w_1)=f_1\tilde h(w_1)&\text{in }\Omega,\\
            w_1=g&\text{on }\partial\Omega,
        \end{cases}\qquad
        \begin{cases}
              F(D^2 w_2)=f_2\tilde h(w_2)&\text{in }\Omega,\\
            w_2=g&\text{on }\partial\Omega.
        \end{cases}
    \end{align*}
   If $f_1<f_2$ in $\Omega$ then 
$$F(D^2w_2)=f_2\tilde h(w_2) \geq   f_1\tilde h(w_2) \qquad \hbox{in } \Omega ,$$ and hence  $w_2$ is a subsolution of
$$\begin{cases}
    F(D^2u)=f_1\tilde h(u)&\text{in }\Omega,\\
    u=g&\text{on }\partial\Omega.
\end{cases}$$
But since  $w_1$ is the largest subsolution of this problem, then $w_2 \leq w_1.$

\medskip

\noindent
\underline{Proof of Property \ref{cond-conv}:} The limiting function $w$ satisfies the equation $F(D^2w)=f\tilde h(w)$ in $\Omega\cap\{w>0\}$ due to the stability of viscosity solutions under uniform limits.  
\end{proof}

We conclude this subsection with a coincidence-set result for fully nonlinear systems, obtained as a consequence of Theorem~\ref{thm:support}.

\begin{corollary}\label{thm:FN-coin}
Assume the hypotheses of Theorem~ \ref{thm:FN}, together with \eqref{eq:rhs-add} and suppose that $(u,v)$ is a  solution pair of the system \eqref{system:fullynonlinear}.
If $C_u$ is a connected component of $\{u=0\}^\circ$ and
$C_v$ is a connected component of $\{v=0\}^\circ$, then either
\[
C_u=C_v
\qquad\text{or}\qquad
C_u\cap C_v=\emptyset.
\]
Moreover, for any $k_i$ and $h_i$, there exist $(f_i,g_i)$ such that for a solution $(u,v)$ of \eqref{system:fullynonlinear} , one nonempty connected component of $\{u=0\}^\circ$ and one of $\{v=0\}^\circ$ are disjoint.
\end{corollary}

\begin{proof}
In view of Theorem \ref{thm:support} it suffices to verify that the operator $\mathcal L(w)=F(D^2w)$
satisfies the properties \ref{cond-Hopf}-\ref{cond-conv-2}.

\medskip

\noindent
\underline{Proof of Property \ref{cond-Hopf}-\ref{cond-nondeg}}: 
The Hopf boundary point lemma for uniformly elliptic fully nonlinear equations  can be found in \cite[P. 27 and Lemma 3.4]{GT2001}
 and \cite[Lemma 1.1 and Lemma 1.2]{Safonov2010}.
For the non-degeneracy estimate, we note that $F(D^2w)\ge (C^0\inf_\Omega f)w^b$ in $\Omega$. Then, from \cite[pages 463-464]{JeoSha24}, we infer $\sup_{B_r(x_0)}w\ge \hat c r^{\frac2{1-b}}$ holds, where $\hat c>0$ is any constant satisfying $(C^0\inf_{\Omega}f)(1-b)-2\hat c\Lambda(n+\frac{2b}{1-b})\ge0$. Thus, by taking $\hat c:=\frac{c^0}{2\Lambda(n+\frac{2b}{1-b})}\inf_\Omega f$, we obtain that $\hat c\to\infty$ as $\inf_\Omega f\to\infty$.

\medskip

\noindent
\underline{Proof of Property  \ref{cond-conv-2}}: Let $w^*$ solve
\[
F(D^2w^*)=0 \quad\text{in }\Omega,\qquad w^*=g \quad\text{on }\partial\Omega,
\]
with $g\in C(\partial\Omega)$. By the maximum  principle,
\[
\|w^*\|_{L^\infty(\Omega)}
\le
\|g\|_{L^\infty(\partial\Omega)}.
\]
Moreover, since $\Omega$ is $C^{1,1}$ and $g\in C(\partial\Omega)$,
standard boundary regularity theory for uniformly elliptic fully nonlinear
equations implies that
$
w^*\in C(\overline\Omega);$ (see, for example, \cite{CC95}).
The conclusion now follows from Theorem~\ref{thm:support}.
\end{proof}


\subsection{The \texorpdfstring{$p$}{p}-Laplacian   case}\label{sec:p-lap-elliptic}
Throughout  this section, we shall assume that $1<p<\infty$ and the boundary data $g_1,g_2$ as well as  $\partial \Omega$ are Lipschitz. 
For $h\in L^1(\R)$, we define
$$
H(t):=\int_{-\infty}^th(s)\,ds,\quad t\in\R.
$$
For this $H$, along with $p\in(1,\infty)$, $f\in L^\infty(\Omega;\R_+)$ and $g\in L^\infty(\partial\Omega;\R_+)$, we define
\begin{align*}
    &J(w,f,p,H,\Omega):=\int_{\Omega}|\D w|^p+pfH(w),\\
    &\cK_{g,p,\Omega}:=\{w\in W^{1,p}(\Omega)\,:\, w=g\,\,\,\text{on }\partial\Omega\}.
\end{align*}

The following regularity result follows from small modifications of \cite[Theorems~5 and ~6]{JeoSha24}:

\begin{theorem}
    \label{thm:p-Lapl-reg}
Assume that $h:\R\to\R$ satisfies $h=0$ on $(-\infty,0]$ and $0\le h(t)\le C^0t^a+2$ on $(0,\infty)$ for some $C^0>0$ and $-1<a<1$. For $f,g\in L^\infty(\Omega;\R_+)$ and $p\in(1,\infty)$, suppose $u$ is a minimizer of $J(\cdot,f,p,H,\Omega)$ over $\cK_{g,p,\Omega}$. Then $u\in C^{1,\al}_{\loc}(\Omega)$ for some constant $\al\in(0,1)$, depending only on $n,p$ and $a$. Moreover, for any $\Omega'\Subset\Omega$, there exists a constant $C>0$, depending only on $\Omega,\Omega',n,p,a,\|f\|_\infty,\|g\|_\infty$, such that
$$
\|u\|_{C^{1,\al}(\Omega')}\le C.
$$
\end{theorem}

\begin{proof}
In \cite[Theorems~5 and ~6]{JeoSha24}, the counterpart of Theorem~\ref{thm:p-Lapl-reg} was proved in the simpler case when $f=1$. In the proof of \cite[Theorems~5 and ~6]{JeoSha24} (see also \cite{LeiDeQTei15}), the property of $H$ is used only to obtain the following inequalities: for any $B_r(x_0)\Subset\Omega$ and $j\in\N$,
\begin{align}
    \label{eq:H-property}
    \begin{split}
        &\int_{B_r(x_0)\cap\{u>j\}}(H(u_j)-H(u))\le0,\\
        &\int_{B_r(x_0)}(H(u^*)-H(u))\le C\int_{B_r(x_0)}|u^*-u|^{1+a},
    \end{split}
\end{align}
where $u_j:=\min\{u,j\}$ and $u^*$ is the $p$-harmonic replacement of $u$ in $B_r(x_0)$, i.e., $u^*$ is the solution of the Dirichlet problem
\begin{align*}\begin{cases}
    \Delta_pu^*=0&\text{in }B_r(x_0),\\
    u^*=u&\text{on }\partial B_r(x_0).
\end{cases}\end{align*}
Thus, Theorem~\ref{thm:p-Lapl-reg} with our $f\in L^\infty(\Omega;\R_+)$ follows once we show that \eqref{eq:H-property} still holds when $H$ is replaced with $fH$. Indeed, from the inequality $u_j\le u$ and the monotonicity of $H$, we have $H(u_j)-H(u)\le 0$. Since $f\ge0$, we further have $fH(u_j)-fH(u)=f(H(u_j)-H(u))\le0$, thus
$$
\int_{B_r(x_0)\cap \{u>j\}}(fH(u_j)-fH(u))\le0.
$$
For the second inequality, we observe that $u\le\|g\|_\infty$ and $u^*\le\|g\|_\infty$ by the maximum principle. If $-1<a<0$, then arguing exactly as in \cite{JeoSha24}, we obtain
\begin{align*}
    \int_{B_r(x_0)}(fH(u^*)-fH(u))&=\int_{B_r(x_0)}f(x)\int_{u(x)}^{u^*(x)}h(s)\,dsdx\le \|f\|_\infty \int_{B_r(x_0)\cap\{u^*>u\}}\int_{u(x)}^{u^*(x)}(C^0s^{a}+2)\,dsdx\\
    &=\frac{C\|f\|_\infty}{1+a}\int_{B_r(x_0)\cap\{u^*>u\}}\left[((u^*(x))^{1+a}-((u(x))^{1+a})+(u^*(x)-u(x))\right]\,dx\\
    &\le \frac{C\|f\|_\infty}{1+a}\int_{B_r(x_0)}|u^*-u|^{1+a},
\end{align*}
where we used \cite[Lemma~2.5]{LeiDeQTei15} in the last step. \\
If $0<a<1$, then $h$ is bounded on $[0,\|g\|_\infty],$
and therefore $H$ is Lipschitz on this interval. 
Consequently,
\begin{align*}
    \int_{B_r(x_0)}(fH(u^*)-fH(u))
    &\le
    C(C^0,a,\|g\|_\infty)\,
    \|f\|_\infty
    \int_{B_r(x_0)}
    |u^*-u|.
\end{align*}
and the same comparison argument yields the desired local $C^{1,\alpha}$ estimate. This completes the proof.    
\end{proof}

\begin{theorem}
    \label{thm:p-Lapl}
Let $\Omega$ be a domain in $\R^n$, $n\ge1$. For $i=1,2$ and $C^1>c^1>0$, suppose $p_i\in (1,\infty)$, $g_i\in Lip(\partial\Omega;\R_+)$ with $c^1<g_i<C^1$, $f_i\in L^\infty(\Omega;\R_+)$, $k_i\in \cK$ and $h_i\in \cH$. Then there exists a solution $(u,v)\in C^{1,\al_1}_{\loc}(\Omega;\R_+)\times C^{1,\al_2}_{\loc}(\Omega;\R_+)$ of the system
\begin{align}\label{system:plaplacian}
    \begin{cases}
        \Delta_{p_1}u=f_1k_1(v)h_1(u)&\text{in }\Omega\cap\{u>0\},\\
        \Delta_{p_2}v=f_2k_2(u)h_2(v)&\text{in }\Omega\cap\{v>0\},\\
        u=g_1,v=g_2&\text{on }\partial\Omega,\\
        u=|\D u|=0&\text{on }\Omega\cap\partial\{u>0\},\\
        v=|\D v|=0&\text{on }\Omega\cap\partial\{v>0\},
    \end{cases}
\end{align}
where $\al_i\in (0,1)$ is a constant depending only on $n$, $p_i$ and $a_i$, $i=1,2$.
\end{theorem}

\begin{proof}
Due to the meta-theorem, Theorem~\ref{thm:meta}, it is enough to show that \ref{cond-sol}-\ref{cond-conv} hold with $\cL=\Delta_p$ and $\cF(\Omega)=C^{1,\al}_{\loc}(\Omega)$, where $\al\in(0,1)$ is as in Theorem~\ref{thm:p-Lapl-reg}. 

\medskip

\noindent 
\underline{Proof of Property \ref{cond-sol}:} To find a solution of 
\begin{align*}
    \begin{cases}
        \Delta_pw=f\tilde h(w)&\text{in }\Omega,\\
        w=g&\text{on }\partial\Omega,
    \end{cases}
\end{align*}
let $w$ be a minimizer of $J(\cdot,f,p,\tilde H,\Omega)$ over $\cK_{g,p,\Omega}$, where $\tilde H(t)=\int_{-\infty}^t\tilde h(s)\,ds$. Clearly, $\max\{w,\tilde c\}\in \cK_{g,p,\Omega}$. From $\tilde h\ge0$, we see that $t\longmapsto \tilde H(t)$ is non-decreasing, hence $J(\max\{w,\tilde c\},f,p,\tilde H,\Omega)\le J(w,f,p,\tilde H,\Omega)$ with the equality if and only if $\max\{w,\tilde c\}=w$ in $\Omega$. Thus, $w\ge \tilde c>0$ in $\Omega$. The estimates \eqref{eq:mod-conti} and \eqref{eq:ftn-space} follow from Theorem~\ref{thm:p-Lapl-reg}. We leave the proof of the existence of $w_*$ after \ref{cond-comp} as we need \ref{cond-comp} to prove the existence of $w_*$.

\medskip

\noindent
\underline{Proof of Property \ref{cond-comp}:} To prove it, we follow the idea in \cite[1.7. Comparison Lemma]{AltPhi86}. For $\tilde H(t)=\int_{-\infty}^t\tilde h(s)\,ds$ as above, we write $E(w,f_i,\tilde H):=|\D w|^p+pf\tilde H(w)$ so that $J(w,f_i,\tilde H):=J(w,f_i,p,\tilde H,\Omega)=\int_\Omega E(w,f_i,\tilde H)$. Since $w_1\vee w_2, w_1\wedge w_2\in \cK_{g,p,\Omega}$, 
\begin{align}
    \label{eq:energy-min-p}
    J(w_1,f_1,\tilde H)\le J(w_1\vee w_2,f_1,\tilde H),\qquad J(w_2,f_2,\tilde H)\le J(w_1\wedge w_2,f_2,\tilde H).
\end{align}
Moreover, a direct computation yields
\begin{align*}
    &((E(w_1\vee w_2,f_1,\tilde H)+E(w_1\wedge w_2,f_2,\tilde H))-(E(w_1,f_1,\tilde H)+E(w_2,f_2,\tilde H))\\
    &=\begin{cases}
        p(f_1-f_2)\int_{w_1}^{w_2}\tilde h(s)\,ds  &\text{in }\{w_1<w_2\},\\
        0&\text{in }\{w_1\ge w_2\}.
    \end{cases}
\end{align*}
Thus, if $\{w_1<w_2\}\neq\emptyset$, then 
\begin{align*}
    ((J(w_1\vee w_2,f_1,\tilde H)+J(w_1\wedge w_2,f_2,\tilde H))-(J(w_1,f_1,\tilde H)+J(w_2,f_2,\tilde H))<0,
\end{align*}
which contradicts \eqref{eq:energy-min-p}. Therefore, $w_1\ge w_2$ in $\Omega.$

Existence of $w_*$ in \ref{cond-sol}: Let $w_*$ be a minimizer of $J(\cdot,\|f\|_\infty,p,H_*,\Omega)$ over $\cK_{g,p,\Omega}$, where $H_*(t)=\int_{0}^t(C^0s^a+2) \,ds$. Then we can compute
\begin{align*}
     &((E(w\vee w_*,f,\tilde H)+E(w\wedge w_*,f,H_*))-(E(w,f,\tilde H)+E(w_*,f_2,H_*))\\
    &=\begin{cases}
        p\left(f\int_{w}^{w_*}\tilde h(s)\,ds-\|f\|_\infty\int_w^{w_*}(C^0s^a+2)\,ds  \right)  &\text{in }\{w<w_*\},\\
        0&\text{in }\{w\ge w_*\}.
    \end{cases}
\end{align*}
Since $w\ge \tilde c$ in $\Omega$ and $\tilde h\le (C^0s^a+2)\chi_{(0,\infty)}$ on $\R$, we can use this computation and argue as in the proof of \ref{cond-comp} to obtain $w_*\le w$ in $\Omega$.

\medskip

\noindent 
\underline{Proof of Property \ref{cond-conv}:} Let $\vp\in C_0^\infty(\Omega\cap\{v>0\})$. Due to the uniform convergence $w_j\to w$, we have that $w_j\ge c>0$ in $\supp\vp$ for large $j$. For such $j$, from the equation $\Delta_pw_j=f^j\tilde h^j(w_j)$ in $\Omega\cap\{w_j>0\}$, we have
$$
\int_{\supp\vp}|\D w_j|^{p-2}\D w_j\cdot\D\vp+f^j\tilde h^j(w_j)\vp=0.
$$
Since $\|w_j\|_{C^{1,\al}_{\loc}(\Omega)}\le C$, we can pass to the limit to get
$$
\int_{\supp\vp}|\D w|^{p-2}\D w\cdot\D\vp+f\tilde h(w)\vp=0.
$$
Thus
$$
\Delta_pw=f\tilde h(w)\quad\text{in }\Omega\cap\{w>0\}.
$$
This completes the proof of Theorem \ref{thm:p-Lapl}.
\end{proof}

We conclude this subsection with a coincidence-set result for the
$p$-Laplacian system, obtained as a consequence of
Theorem~\ref{thm:support}.

\begin{corollary}
\label{cor:p-coin}
Assume the hypotheses of Theorem~\ref{thm:p-Lapl}, together with
\eqref{eq:rhs-add}.
Let $(u,v)$ be a solution pair of \eqref{system:plaplacian}. If $C_u$ is a connected component of $\{u=0\}^\circ$ and
$C_v$ is a connected component of $\{v=0\}^\circ$, then either $
C_u=C_v $ or $C_u\cap C_v=\emptyset.$

Moreover, for any $h_i,k_i$ and $p_i$, there exist $(f_i,g_i)$ such that the corresponding
solution $(u,v)$ of \eqref{system:plaplacian} has nonempty connected
components of $\{u=0\}^\circ$ and $\{v=0\}^\circ$ which are disjoint.
\end{corollary}
\begin{proof}
We verify the three conditions needed to apply Theorem~\ref{thm:support} with
$
\mathcal L(w)=\Delta_{p}w.
$

\medskip
\noindent
\underline{Proof of Properties \ref{cond-Hopf}-\ref{cond-nondeg}:}

The Hopf boundary point lemma for the $p$-Laplacian can be found in \cite[Proposition 3.3.1]{Tol84}.\footnote{The proof of the Hopf lemma for general operators follows in the same way as that in \cite{Hopf1952}, provided that the interior sphere condition is satisfied. We also refer to \cite{Mi-Sh} for $C^{1,\text{Dini}}$-domains and the so-called $H$-harmonic functions (which include the $p$-Laplace operator).}

Non-degeneracy is standard in the literature; see, e.g., \cite[(2.14)]{JeoSha24}, or  \cite[Lemma 3.1]{KKPS}.

\medskip
\noindent
\underline{Proof of Property \ref{cond-conv-2}:}
Let $g_i\in C(\partial\Omega)$ and let $w^*$ solve
\[
\Delta_{p_i} w^* = 0 \quad\text{in }\Omega,\qquad
w^* = g_i \quad\text{on }\partial\Omega.
\]
By the comparison principle,
\[
\|w^*\|_{L^\infty(\Omega)}\le \|g_i\|_{L^\infty(\partial\Omega)}.
\]
Moreover, standard regularity theory for the $p$-Laplacian implies that $w^*\in C(\overline{\Omega})$.
Hence
$$\|w^*\|_{L^\infty(\Omega)}\to 0
\qquad\text{as}\qquad
\|g_i\|_{L^\infty(\partial\Omega)}\to 0,$$
see \cite{Lin19,Tol84}.

Having verified \ref{cond-Hopf}--\ref{cond-conv-2}, the conclusion follows from
Theorem~\ref{thm:support}.
\end{proof}

\section{The Parabolic Case}\label{sec:par}

As in  the elliptic case, we first fix all  necessary notations and definitions. The proof strategy remains the same, following the meta-theorem in Section \ref{sec:par-meta} and hinging on the operator properties established earlier. Since the overall procedure is nearly identical to the elliptic setting, we omit repetitive details where the arguments follow the same lines of reasoning. The main task is to verify the properties required for specific operators in Sections \ref{sec:FN-par} and \ref{sec:p-par}, although the latter presents significantly more technical resistance.
In particular, the $p$-parabolic case involves additional difficulties arising from the
degeneracy of the equation and the corresponding regularity theory. For this reason,
some of the more technical arguments are deferred to the Appendix, and we shall
briefly indicate the relevant points in the $p$-parabolic subsection.

\subsection{A parabolic meta-theorem}\label{sec:par-meta}
For a bounded domain $\Omega \subset \mathbb{R}^n$ ($n \ge 1$) and $T > 0$, we define $\Omega_T := \Omega \times (-T, 0]$, with $\partial_p \Omega_T$ denoting the parabolic boundary.
Space-time points are denoted by capital letters $X= (x,t)$, and the parabolic cylinder of width $r$, height $r^2 $ and base center $X^0 =(x_0,t_0)$ 
by   
$Q^2_r (X^0):= B_r(x_0)\times (t_0-r^2,t_0]$.
We define a "parabolic" distance as 
$$d(X,Y):= \max(|x-y|,|t-s|^{1/2}),\qquad 
\hbox{where } X=(x,t),\ Y=(y,s).$$

We also continue to use the notation  $\nabla  u $ as the space gradient of $u$. 
We denote by $\cF(\Omega_T)$ a class of non-negative continuous functions in $\Omega_T$, which will be specified depending on the operator under consideration.
Given  an operator $\cL:\cF(\Omega_T)\to \cA(\Omega_T;\R_+)$, where $\cA(\Omega_T;\R_+)$ is the space of nonnegative measurable functions in $\Omega_T$, we define the following structural properties on the operator $\cL $:

\medskip


\begin{enumerate}[label={\bf (PP\arabic*)}]
    \item\label{cond-sol-par} {\bf Solvability:}  For any $g\in C(\partial_p\Omega_T;\R_+)$, $f\in L^\infty(\Omega_T;\R_+)$ and $\tilde h\in\widetilde{\mathcal H}$, there exists a solution $w\in \cF(\Omega_T)$ of the Dirichlet problem 
    \begin{align}\label{eq:sol-Dir-par}
        \begin{cases}
            \cL w-\partial_tw=f\tilde h(w)&\text{in }\Omega_T,\\
            w=g&\text{on }\partial_p\Omega_T,
        \end{cases}
    \end{align}
    which satisfies the following regularity estimates: for any $K\Subset\Omega_T$,

    \begin{align}\label{eq:mod-conti-par}
|w(X^1)-w(X^2)|
\le \rho\big(d_2(X^1,X^2)\big)
\quad\text{for all }X^1,X^2\in K. 
\end{align}
    We further require 
    \begin{align}
        \label{eq:ftn-space-par}
        \sup_{Q_r^2(Z)\cap K}w\le C(r^{1+\al}+w(Z))\quad\text{for any } Z \in K\text{ and }r>0,
    \end{align} 
        where $C>0$ and $0<\al<1$ are constants and $\rho$ is a modulus of continuity which depend only on $\Omega_T,K,\cL,a,\|f\|_\infty$ and $\|g\|_\infty$.
        Moreover, there exists a  function $w_*$  in $\Omega_T$, depending only on $h,\|f\|_\infty, g$ and $\Omega_T$, such that
        \begin{align}\label{eq:lower-bound-par}
            \begin{cases}
        w_*\le w&\text{in }\Omega_T,\\
        w_*=g&\text{on }\partial_p\Omega_T.
    \end{cases}
        \end{align}

    \medskip

    \item\label{cond-comp-par} {\bf Comparison:} Let $g\in C(\partial_p\Omega_T;\R_+)$, $f_1,f_2\in L^\infty(\Omega_T;\R_+)$, and $\tilde h\in\widetilde{\mathcal H}$. Suppose that $w_1$ and $w_2$ are solutions of
    \begin{align*}
        \begin{cases}
            \cL w_1-\partial_tw_1=f_1\tilde h(w_1)&\text{in }\Omega_T,\\
            w_1=g&\text{on }\partial_p\Omega_T,
        \end{cases}\qquad
        \begin{cases}
            \cL w_2-\partial_tw_2=f_2\tilde h(w_2)&\text{in }\Omega_T,\\
            w_2=g&\text{on }\partial_p\Omega_T,
        \end{cases}
    \end{align*}
    respectively. 
   If $f_1<f_2$ in $\Omega_T$, then $w_1\ge w_2$ in $\Omega_T$.

   \medskip

    \item\label{cond-conv-par} {\bf Convergence:}
    Let $g\in C(\partial_p\Omega_T;\R_+)$. For each $j\in \N$, let $w_j\in \cF(\Omega_T)$, $\tilde h^j\in\widetilde{\mathcal H},$ and $f^j\in L^\infty(\Omega_T;\R_+)$ with
    $$
    \begin{cases}
        \cL w_j-\partial_tw_j=f^j\tilde h^j(w_j)&\text{in }\Omega_T\cap\{w_j>0\},\\
        w_j=g&\text{on }\partial_p\Omega_T.
    \end{cases}
    $$ 
  and suppose that $
\sup_{j\in\mathbb N}\|f^j\|_{L^\infty(\Omega_T)}<\infty.$ Assume that $f^j\to f$ and $w_j\to w$ locally uniformly in $\Omega_T$ for some $f\in L^\infty(\Omega_T)$ and $w\in C(\Omega_T;\R_+)$ and that $\tilde h^j\to \tilde h\in \cH\cup\widetilde\cH$ locally uniformly on $(0,\infty)$. Then
    \begin{align*}
            w\in \cF(\Omega_T)\quad\text{and}\quad  \cL w-\partial_t w=f\tilde h(w)\,\,\text{ in }\Omega_T\cap\{w>0\}.
    \end{align*}
\end{enumerate}

With these properties at hand one can now derive a similar meta-theorem in the parabolic setting, whose proof follows the same lines of argument as those of the elliptic case, which we do not repeat.
However, the difficulty  in the parabolic setting is the application of the meta-theorem to particular PDEs, which requires the verification of properties \ref{cond-sol-par}-\ref{cond-conv-par} for each PDE. 
These technical hurdles will be  clear in the proof of $p$-parabolic case. Since some of the technical results needed are not standard, we have presented them and their proofs in Appendix \ref{appen:par-p-Lapl-reg}.

\begin{theorem}\label{thm:meta-par}[Parabolic Meta-theorem]
    Let $\Omega_T$ be a bounded cylinder in $\R^{n+1}$. For $i=1,2$ and $C^1>c^1>0$, let $g_i\in C(\partial_p\Omega_T;\R_+)$ with $c^1<g_i<C^1$ and $f_i\in L^\infty(\Omega_T;\R_+)$ and assume that
$h_i\in\mathcal H,$ and $k_i\in\mathcal K$. For $i=1,2$, let $\mathcal F^i(\Omega_T)$ be function classes such that
the operator $\mathcal L_i$ satisfies
\ref{cond-sol-par}--\ref{cond-conv-par} on $\mathcal F^i(\Omega_T)$.
Then there exists a solution $(u,v)\in \cF^1(\Omega_T)\times \cF^2(\Omega_T)$ of the system
    \begin{align}\label{eq:sol-2-par}
        \begin{cases}
            \cL_1 u-\partial_tu=f_1k_1(v)h_1(u)&\text{in }\Omega_T\cap\{u>0\},\\
            \cL_2 v-\partial_tv=f_2k_2(u)h_2(v)&\text{in }\Omega_T\cap\{v>0\},\\
            u=g_1, v=g_2&\text{on }\partial_p\Omega_T,\\
            u=|\D u|=0&\text{on }\Omega_T\cap\partial\{u>0\},\\
            v=|\D v|=0&\text{on }\Omega_T\cap\partial\{v>0\}.
        \end{cases}
    \end{align}
\end{theorem}

 Observe that $g_i, f_i$ above may also depend on $t$.

\begin{proof} To prove this theorem and the next one we would need  a parabolic counterpart of Proposition \ref{prop:meta}, which actually works the same way as in the elliptic case with $\mathcal{L}_i$ replaced by $\mathcal{L}_i - \partial_t$. The only difference is that we would need the initial data to be satisfied in the parabolic case which is taken care of by the assumption \ref{cond-sol-par}, especially \eqref{eq:lower-bound-par}. 
In addition, the Arzel\`{a}--Ascoli theorem is applied with respect to the parabolic distance, using the equicontinuity estimate~\eqref{eq:mod-conti-par}.
From this, the proof of the theorem  follows by repeating the argument in its elliptic counterpart, Theorem~\ref{thm:meta}. 
\end{proof}


As in the elliptic case, see \eqref{eq:rhs-add},
to state our results on the relation between the coincidence sets, additional conditions on $f,k,h,\tilde h$ and $\cL_i$ are necessary:
\begin{align}
    \label{eq:rhs-add-par}
    \begin{cases}
        \text{$\inf_{\Omega_T} f>0$  },\\
        \text{$k>0$ on $(0,\infty)$, }\\
        \text{For some $0<b<1$ and $c^0>0$, we have $h(t),\tilde h(t)\ge c^0t^b$ for $0<t<\infty$.   }
    \end{cases}
\end{align}
\begin{enumerate}[label={\bf (PP\arabic*')}]
    \item\label{cond-Hopf-par} {\bf Hopf's boundary principle:} For any cylinder $Q_r(Z^0)\Subset \Omega_T$, if $w\in C^1(\overline{Q_r(Z^0)})$ satisfies that $\cL w-\partial_tw=0$ and $w>0$ in $Q_r(Z^0)$ and that $w(z^1)=0$ for some $z^1\in \partial B_r(x^0)\times\{t^0\}$, then $|\D w(z^1)|\neq0$.
    \item\label{cond-nondeg-par}{\bf Nondegeneracy:} Let $w$ be a solution of \eqref{eq:sol-Dir-par} in $\Omega_T$. If $Z^0\in \overline{\{w>0\}}$ and $Q_r(Z^0)\subset\Omega_T$, then
    $$
    \sup_{Q_r(Z^0)}w\ge cr^{b_0} \quad\text{for some }b_0>0,
    $$
    where $c>0$ is a constant depending only on $b,c^0,n,\cL,\inf_{\Omega_T} f$. Here, $c\to\infty$ as $\inf_{\Omega_T} f\to\infty$.
\item\label{cond-conv-2-par}{\bf Regularity:} For $g\in C(\partial_p\Omega_T;\R_+)$, let $w^*$ be a solution of 
\begin{align*}
    \begin{cases}
        \cL w^*-\partial_tw^*=0&\text{in }\Omega_T,\\
        w^*=g&\text{on }\partial_p\Omega_T.
    \end{cases}
\end{align*}
    Then $w^*\in C(\overline{\Omega_T})$ with a modulus of continuity depending only on $\Omega,\cL$ and $\|g\|_\infty$. Moreover,  $\|w^*\|_{L^\infty(\Omega)}\to0$ as $\|g\|_{L^\infty(\partial\Omega)}\to0$.
\end{enumerate}

For any set $A\subset \Omega_T$, we write
$$
A^t:=\{x\in\Omega\,:\, (x,t)\in A\}.
$$

\begin{theorem}\label{thm:support-par}
Let $(u,v)$ be a solution of the system \eqref{eq:sol-2-par} given by Theorem \ref{thm:meta-par}.
Assume in addition that for  $i=1,2$, $f_i, k_i,h_i$ satisfy \eqref{eq:rhs-add-par} and $\cL_i$ satisfy the further conditions \ref{cond-Hopf-par}-\ref{cond-conv-2-par}.  Let $C_u$ and $C_v$ be connected components of $\{u=0\}^\circ$ and $\{v=0\}^\circ$, respectively. 
    Then, for each $t\in(-T,0]$,
    \begin{align}\label{eq:dichotomy-par} 
        \text{either\,\,\, $C_u^t=C_v^t$ \,\,\,or \,\,\,$C_u^t\cap C_v^t=\emptyset$.   }
    \end{align}
    Moreover, for any $k_i$, $h_i$ and $\cL_i$, there exist $(f_i,g_i)$ such that for any solution pair  $(u,v)$, there exist a connected component of $\{u=0\}^{\mathrm{o}}$ and one of $\{v=0\}^{\mathrm{o}}$, say $C_u$ and $C_v$, such that
    \begin{equation}\label{eq:disjoint-par} 
        C_u^t \text{ and } C_v^t \text{ are nonempty and disjoint for some } t\in(-T,0].
    \end{equation}
\end{theorem}

\begin{proof} 
Since the proof of Theorem~\ref{thm:support-par} is rather a straightforward modification of the proof of its elliptic counterpart, Theorem~\ref{thm:support}, we omit the details. We just remark that as mentioned earlier in the proof of Theorem \ref{thm:meta-par} we would need a parabolic version of 
 Proposition \ref{prop:meta}, which again works similarly as in the elliptic case.
\end{proof}


\subsection{The fully nonlinear parabolic case}
\label{sec:FN-par}

Using the parabolic meta-theorem we can derive a parabolic counterpart of Theorem \ref{thm:FN}. 
A key component in the proof is the  scalar parabolic counterpart of the elliptic existence theory, Theorem \ref{thm:exist-FN}. 
As in the elliptic case, we work with viscosity solutions.

\begin{theorem}
    \label{thm:FNpar}
    Let $\Omega_T = \Omega \times (-T,0]$, where 
    $\Omega$ is a bounded $C^{1,1}$ domain in $\R^n$, $n\ge1 $. For $i=1,2$ and the  constants $C^1>c^1>0$, suppose that $g_i\in\text{Lip}(\partial_p\Omega_T)$ with $c^1<g_i<C^1$ and $f_i\in L^\infty(\Omega_T;\R_+)$. Assume that $k_i \in \cK,$ and $h_i\in \cH $. Suppose $F_1$ and $F_2$ satisfy \eqref{eq:assump-fully-nonlinear}. Then there exists a solution $(u,v)\in \cF(\Omega_T;\R_+)$ of the system
    \begin{align}\label{sys:FNpar}
        \begin{cases}
            F_1(D^2u) - \partial_t u= f_1k_1(v)h_1(u)&\text{in }\Omega_T\cap\{u>0\},\\
            F_2(D^2v) - \partial_t v =f_2k_2(u)h_2(v)&\text{in }\Omega_T\cap\{v>0\},\\
            u=g_1,v=g_2&\text{on }\partial_p\Omega_T,\\
            u=|\D u|=0&\text{on }\Omega_T\cap\partial\{u>0\},\\
            v=|\D v|=0&\text{on }\Omega_T\cap\partial\{v>0\}.
        \end{cases}
    \end{align}
\end{theorem}

\begin{proof}  
The proof of this theorem follows the same lines as its elliptic counterpart, Theorem~\ref{thm:FN}. We thus need to show that properties~\ref{cond-sol-par}--\ref{cond-conv-par} hold for the parabolic operators $F(D^2w) - \partial_t w$. 

Since the starting point of our approach is to work with regularized scalar  problem, we can use standard classical references such as \cite{CrandallIshiiLions1992, Ishii1987}.
The existence of a viscosity solution to
$F(D^2w)-\partial_t w = f\tilde h(w)$ in $\Omega_T$ with
$w=g$ on $\partial_p\Omega_T$, given a sub- and supersolution,
follows in a standard way;  for a recent and partially related problem see \cite[Section 3]{AraSaUrb24} or  \cite[Theorem~3.1]{RTU}.
Property~\ref{cond-comp-par} follows by the same argument as its
elliptic counterpart in the proof of Theorem~\ref{thm:FN}.
Property~\ref{cond-conv-par} follows from \cite[Theorem~5.1]{RTU}.

We shall now explain how property~\ref{cond-sol-par} is established in this case. We need to have the parabolic counterpart of the function $w_*$ (see \eqref{eq:lower-bound}) which plays the subsolution role in the proof. 
Although the parabolic distance would be a natural choice, one may encounter difficulties close to initial state, since the distance function is not smooth when we are switching from the space boundary to time zero. 
To rectify this issue, we modify the proof of the elliptic counterpart, Step 1 in Property \ref{cond-sol} in the proof of Theorem~\ref{thm:FN}. Let $\tilde g$ be the solution of
\begin{align}\label{eq:Pucci-par}
    \begin{cases}
        \cM_{\la,\Lambda}^-(D^2\tilde g)-\partial_t\tilde g=0&\text{in }\Omega_T,\\
        \tilde g=g&\text{on }\partial_p\Omega_T.
    \end{cases}
\end{align}
For $d_{\partial\Omega}$ and $\Omega_r$ as in the elliptic case, we consider
\begin{align*}
    w_*(x,t):=\begin{cases}
        \tilde g(x,t)\left(\frac{r-d_{\partial\Omega}(x)}{r}\right)^\be,& x\in \Omega_r\times(-T,0],\\
        0,&x\in (\Omega\setminus\Omega_r)\times(-T,0].
    \end{cases}
\end{align*} 
By combining \eqref{eq:subsol-comp-1} and the first equation in \eqref{eq:Pucci-par}, we get the parabolic analogue of \eqref{eq:subsol-comp-2}:
\begin{align}
     r^\be \left(F(D^2w_*)-\partial_t w_*\right)\ge&\be(\be-1)\tilde g(r-d_{\partial\Omega})^{\be-2}\cM_{\la,\Lambda}^-(\D d_{\partial\Omega}\otimes\D d_{\partial\Omega})-2\be(r-d_{\partial\Omega})^{\be-1}\cM_{\la,\Lambda}^+(\D \tilde g\otimes \D d_{\partial\Omega})\\
    &-\be\tilde g(r-d_{\partial\Omega})^{\be-1}\cM_{\la,\Lambda}^+(D^2d_{\partial\Omega}). 
\end{align} 
By following the remaining computation in the elliptic case and, if necessary, taking $r$ even smaller so that $w_*\le g$ on $\Omega\times\{-T\}$, we infer that $w_*$ is the desired parabolic barrier, with the equality in the initial state replaced by an inequality.

To address the above issue with the initial state, let $\hat g$ be the solution of $F(D^2\hat g)-\partial_t\hat g=0$
in $\Omega_T$ with initial value $g(x,-T)$ and lateral boundary values 
smaller than $g$. Define $ v_*(x,t) := \hat g(x,t)\left(1-\frac{t+T}{t_0}\right)_+$
for $t_0>0$ small, chosen so that $v_*\le g(x,t)$ on
$\partial\Omega\times(-T,-T+t_0)$. Define $\Omega^{t_0} = \Omega \times [-T,-T+t_0/2] $. Since $\hat g$ solves the
homogeneous equation, a direct calculation gives
$$F(D^2v_*)-\partial_t v_* =  \frac{\hat g}{t_0}
 \geq \frac{\inf_{\Omega^{t_0}}\hat g}{t_0}\ge c/t_0 \quad\text{in }\Omega^{t_0}$$ for
$t_0$ small. For $a<0$,
\[
  \tilde h(v_*)\le Cv_*^a+2\le C\left(\inf_{\Omega^{t_0}} \hat g\!\left(1-\frac{t+T}{t_0}\right)\right)^{\!a}+2
  \le \frac{c}{t_0}
\]
for $t_0$ small enough, so $F(D^2v_*)-\partial_t v_*\ge
f\tilde h(v_*)$. For $a\geq 0$ a similar argument works, where $\inf_{\Omega^{t_0}} \hat g\left(1-\frac{t+T}{t_0}\right)$ is replaced by $\sup_{\Omega^{t_0}} \hat g\left(1-\frac{t+T}{t_0}\right)$ and we obtain a similar conclusion. 
Now, it is easily seen that $\max\{w_*,v_*\}$ satisfies \eqref{eq:lower-bound-par}, thus Property~\ref{cond-sol-par} is proved.

To finalize, we need to show   \eqref{eq:mod-conti-par}-\eqref{eq:ftn-space-par}.  The modulus-of-continuity estimate~\eqref{eq:mod-conti-par} follows from \cite[Section 4.1--4.2]{AraSaUrb24}. The one-sided growth
estimate~\eqref{eq:ftn-space-par} follows by 
\cite[Corollary 1]{AraSaUrb24}.
\end{proof}

We next state  the counterpart of Corollary \ref{thm:FN-coin}, whose proof follows the same lines as that in the proof of Corollary \ref{thm:FN-coin}.  Indeed, one needs boundary Hopf lemma for the parabolic version, which can be found in 
\cite[Section 4]{CaffarelliLiNirenberg2013}.
Regarding the non-degeneracy, suppose that $w$ satisfies $F(D^2w)-\partial_tw\ge (c^0\inf_\Omega f)w^b$ in $\Omega_T$. Then, for $\bar w:=w^{1-b}$, a direct calculation gives that
\begin{align*}
    F(D^2\bar w)-\partial_t\bar w&\ge F\left((1-b)w^{-b}D^2w\right)+\cM_{\la,\Lambda}^-\left(\frac{b}{1-b}\cdot\frac{\D \bar w\otimes\D \bar w}{\bar w}\right)-(1-b)w^{-b}\partial_tw\\
    &\ge (c^0\inf_\Omega f)(1-b)-\frac{b\Lambda}{1-b}\frac{|\D\bar w|^2}{\bar w}.
\end{align*}
This is the parabolic analogue of the argument in lines 12–13 of \cite[page 463]{JeoSha24}. Therefore, by following the computations on \cite[pages 463-464]{JeoSha24}, we obtain the desired non-degeneracy estimate.
Moreover, for a solution $w^*$ of
$$
F(D^2w^*)-\partial_tw^*=0\quad\text{in }\Omega_T,\qquad w^*=g\quad\text{on }\partial_p\Omega_T,
$$
the global continuity $w^*\in C(\overline \Omega)$ follows from the standard existence and regularity for uniformly parabolic fully nonlinear equations. Finally, the comparison principle can be found in \cite[Theorem~14.1]{Lie96}, which implies that $\|w^*\|_{L^\infty(\Omega_T)}\to0$ as $\|g\|_{L^\infty(\partial_p\Omega_T)}\to0$.

    \begin{corollary}\label{cor:FNpar-coin}
  Let $(u,v)$ be the solution of the system \eqref{sys:FNpar} in Theorem~\ref{thm:FNpar}. Assume in addition that for  $i=1,2$, $f_i, k_i,h_i$ satisfy \eqref{eq:rhs-add-par}.  Then the connected components  $C_u$ and $C_v$ of  $\{u=0\}^\circ$ and $\{v=0\}^\circ$ satisfy
  \eqref{eq:dichotomy-par}. Moreover, for any $k_i$ and $h_i$, there exist $(f_i,g_i)$ such that for any solution pair  $(u,v)$, there exist connected components $C_u$ and $C_v$  satisfying \eqref{eq:disjoint-par}.
\end{corollary}

\subsection{The \texorpdfstring{$p$-parabolic}{p-parabolic}  case}\label{sec:p-par}

In this section, 
we only consider the case when $p \neq 2$, as the case $p=2$ is already included in the Fully nonlinear theory, when the operator is standard heat operator.

We now consider the $p$-Laplacian case in the parabolic setting. 
We first define the notion of a solution.  By the solution to
\begin{align}\label{eq:p-sol-par}
    \begin{cases}
        \Delta_pw-\partial_tw=f\tilde h(w)&\text{in }\Omega_T,\\
        w=g&\text{on }\partial_p\Omega_T,
    \end{cases}
\end{align} 
we mean a function $w\in \cK_{g,p,\Omega_T}=\{w\in L^p(-T,0;W^{1,p}(\Omega)) \,:\, w=g\,\,\,\text{on }\partial_p\Omega_T\}$ such that  for any $v\in \cK_{g,p,\Omega_T}$:
\begin{align}
    \label{eq:par-var-ineq}
    \int_{\Omega_T}|\D w|^p+pf\tilde H(w)+p(w-v)\partial_tv\le \int_{\Omega_T}|\D v|^p+pf\tilde H(v);
\end{align} 
 here $\tilde H(t)=\int_{-\infty}^t\tilde h(s)\,ds$.  Similarly, by a solution of 
\begin{align*}
    \begin{cases}
        \Delta_pw-\partial_tw=f\tilde h(w)&\text{in }\Omega_T\cap\{w>0\},\\
        w=g&\text{on }\partial_p\Omega_T,
    \end{cases}
\end{align*}
we mean a function $w\in \cK_{g,p,\Omega_T}$ satisfying \eqref{eq:par-var-ineq} for any $v\in \cK_{g,p,\Omega_T}$ with $\supp(w-v)\Subset\{w>0\}$. Note that by a standard variation argument, these notion of solutions is stronger than that of weak solutions.

\begin{theorem}
    \label{thm:p-par}
    Let $\Omega_T = \Omega \times (-T,0]$, where 
    $\Omega$ is a bounded Lipschitz  domain in $\R^n$, $n\ge1 $. For $i=1,2$ and constants $C^1>c^1>0$, suppose $p_i\in\left(\max\left\{1,\frac{2n}{n+2}\right\},\infty \right)\setminus\{2\}$,
    $g_i\in \text{Lip}(\partial_p\Omega_T)$ with $c^1<g_i<C^1$ and $f_i\in L^\infty(\Omega_T;\R_+)$. Assume that $k_i\in \cK$ , and $h_i\in \cH$ with $-1<a_i<1$. Then there exists a solution $(u,v)\in C^{1+\al_1,\frac{1+\al_1}2}_{\loc}(\Omega_T;\R_+)\times C^{1+\al_2,\frac{1+\al_2}2}_{\loc}(\Omega_T;\R_+)$ of the system
    \begin{align}\label{sys:p-par}
        \begin{cases}
            \Delta_{p_1} u - \partial_t u= f_1k_1(v)h_1(u)&\text{in }\Omega_T\cap\{u>0\},\\
           \Delta_{p_2} v - \partial_t v =f_2k_2(u)h_2(v)&\text{in }\Omega_T\cap\{v>0\},\\
            u=g_1,v=g_2&\text{on }\partial_p\Omega_T,\\
            u=|\D u|=0&\text{on }\Omega_T\cap\partial\{u>0\},\\
            v=|\D v|=0&\text{on }\Omega_T\cap\partial\{v>0\},
        \end{cases}
    \end{align}
    where $\al_i\in(0,1)$ is a constant depending only on $n,p_i$ and $a_i$, $i=1,2$.
\end{theorem}
\begin{proof}
\underline{Proof of Property \ref{cond-sol-par}:} We split the proof of \ref{cond-sol-par} into two steps.

\medskip\noindent\emph{Step 1.} In this step, we prove that there exists a function $w\in \cK_{g,p,\Omega_T}$ satisfying
\begin{align}\label{eq:par-var-ineq-1}
\int_{\Omega_T}|\D w|^p+pf\tilde H(w)+p(w-v)\partial_tw\le \int_{\Omega_T}|\D v|^p+pf\tilde H(v)
\end{align}
for any $v\in \cK_{g,p,\Omega_T}$. We prove it by using the discrete minimizing-movement approximation. 
By a standard approximation argument, we may assume that $f$ is Lipschitz in time. Let $\tau\in(0,1/2)$ with $K:=T/\tau\in \N$. For $k\in \{0,1,2,\ldots,K\}$, we set $t_k:=\tau k-T$, $k\in \{0,1,2,\ldots, K\}$, and define $w_k^\tau:\Omega\to\R$ as follows: When $k=0$, we let $w_0^\tau(x):=g(x,-T)$, $x\in \Omega$. Given $w_{k-1}^\tau$, we let $w_k^\tau\in g(\cdot,t_{k-1})+W^{1,p}_0(\Omega)$ be a minimizer of
$$
J_k^\tau(u):=\int_{\Omega}\left(|\D u|^p+pf(x,t_k)\tilde H(u)+\frac{p}{2\tau}|u-w_{k-1}^\tau|^2 \right)\,dx
$$
over $u\in g(\cdot,t_{k-1})+W^{1,p}_0(\Omega)$. For any $v\in g(\cdot,t_{k-1})+W^{1,p}_0(\Omega)$, we have by the minimality of $w_k^\tau$
\begin{equation*}
    \int_\Omega\left(|\D w_{k}^\tau|^p
    +pf_k\tilde H(w_k^\tau)
    +\frac{p}{2\tau}|w_k^\tau-w_{k-1}^\tau|^2 \right)dx
    \le \int_\Omega\left(|\D v|^p
    +pf_k\tilde H(v)
    +\frac{p}{2\tau}|v-w_{k-1}^\tau|^2 \right)dx. 
\end{equation*}

Due to the identity $|w_k^\tau-w_{k-1}^\tau|^2-2(w_k^\tau-w_{k-1}^\tau)(w_k^\tau-v)=|v-w_{k-1}^\tau|^2-|v-w_k^\tau|^2$, this is equivalent to
\begin{flalign}
    \label{eq:min-move}
  &\int_\Omega\left(|\D w_k^\tau|^p+pf(x,t_k)\tilde H(w_k^\tau)+\frac{p}{\tau}(w_k^\tau-w_{k-1}^\tau)(w_k^\tau-v)  \right)\,dx \\
&\le \int_\Omega\left(|\D v|^p+pf(x,t_k)\tilde H(v)+\frac{p}{2\tau}|v-w_k^\tau|^2  \right)\,dx.
\end{flalign}
By plugging in $v=w_{k-1}^\tau\in g(\cdot,t_{k-1})+W_0^{1,p}(\Omega)$ and using that $f$ is Lipschitz in time and $\tilde H$ is bounded, we obtain
\begin{align*}
    &\int_\Omega \left(|\D w_k^\tau|^p+pf(x,t_k)\tilde H(w_k^\tau)+\frac{p}{2\tau}|w_k^\tau-w_{k-1}^\tau|^2 \right)\,dx\le\int_\Omega \left(|\D w_{k-1}^\tau|^p+pf(x,t_k)\tilde H(w_{k-1}^\tau) \right)\,dx\\
    &\qquad= \int_\Omega \left(|\D w_{k-1}^\tau|^p+pf(x,t_{k-1})\tilde H(w_{k-1}^\tau)+p(f(x,t_k)-f(x,t_{k-1}))\tilde H(w_{k-1}^\tau) \right)\,dx\\
    &\qquad\le\int_\Omega \left(|\D w^\tau_{k-1}|^p+pf(x,t_{k-1})\tilde H(w_{k-1}^\tau)  \right)\,dx+C\tau.
\end{align*}
Then, for 
$$
E_k(u):=\int_\Omega\left(|\D u|^p+pf(x,t_k)\tilde H(u)\right)dx,\quad 0\le k\le K,
$$
we have
\begin{align}
    \label{eq:par-var-ineq-comp}
    E_k(w_k^\tau)+\frac p{2\tau}\int_\Omega|w_k^\tau-w_{k-1}^\tau|^2\le E_{k-1}(w_{k-1}^\tau)+C\tau,\quad 1\le k\le K.
\end{align}
Since $K\tau=T$, \eqref{eq:par-var-ineq-comp} gives
$$
E_k(w_k^\tau)\le C.
$$
This, along with \eqref{eq:par-var-ineq-comp}, yields
\begin{align}
    \label{eq:min-move-est}
    \sup_{1\le k\le K}\int_\Omega|\D w_k^\tau|^p\le C,\qquad \sum_{k=1}^K\int_{\Omega}|w_k^\tau-w_{k-1}^\tau|^2\le C\tau.
\end{align}
Let $w^\tau, \tilde w^\tau:\Omega_T\to\R$ be defined by
$$
w^\tau(x,t):=w^\tau_k(x),\quad t\in (t_{k-1},t_k],\,\,\, x\in\Omega 
$$ 

and
$$
\tilde w^\tau(x,t):=\frac{t-t_{k-1}}\tau w_k^\tau(x)+\frac{t_k-t}\tau w_{k-1}^\tau(x),\quad t\in(t_{k-1},t_k],\,\,\, x\in \Omega.
$$
Note that $\partial_t\tilde w^\tau=\frac{w_k^\tau-w_{k-1}^\tau}\tau$ for $t\in(t_{k-1},t_k)$. From \eqref{eq:min-move-est}, we infer
$$
\int_{\Omega_T}|\D w^\tau|^p\le C,\qquad \int_{\Omega_T}|\partial_t\tilde w^\tau|^2\le C.
$$
From these estimates and Poincaré inequality, we deduce that for some $w\in \cK_{g,p,\Omega_T}$ and $q>1$,
\begin{align*}
    &w^\tau \rightharpoonup w\quad\text{weakly in }L^p(\Omega_T),\\
    &w^\tau\rightarrow w\quad\text{strongly in } L^q(\Omega_T),\\
    &\partial_t\tilde w^\tau\rightharpoonup \partial_tw \quad\text{weakly in }L^2(\Omega_T).
\end{align*}
This, combined with \eqref{eq:min-move} and the trace theorem, implies that $w$ satisfies \eqref{eq:par-var-ineq-1}.

\medskip\noindent\emph{Step 2.} In this step, we verify \ref{cond-sol-par} by using the result in Step 1.
 Since $w=g=v$ on $\partial_p\Omega_T$,
$$
\int_{\Omega_T}(w-v)\partial_tw-\int_{\Omega_T}(w-v)\partial_tv=\int_{\Omega_T}\partial_t\left(\frac12(w-v)^2\right)\ge0.
$$
By combining this and \eqref{eq:par-var-ineq-1}, we obtain that $w$ satisfies \eqref{eq:par-var-ineq}.

To show that $w\ge0$ in $\Omega_T$, let $w_+:=\max\{w,0\}\ge0$. In $\{w\ge0\}$, we have $w-w_+=0$. On the other hand, in $\{w<0\}$, we have $w_+=0$, thus $\partial_tw_+=0$ almost everywhere. Thus, $(w-w_+)\partial_tw_+=0$ in $\Omega_T$ almost everywhere. Then, we have by 
\eqref{eq:par-var-ineq}
$$
\int_{\Omega_T}\left(|\D w|^p+pf\tilde H(w)\right)\le\int_{\Omega_T}\left(|\D w_+|^p+pf\tilde H(w_+)\right).
$$
Clearly, $|\D w|^p\ge |\D w_+|^p$ and $pf\tilde H(w)\ge pf\tilde H(w_+)$ in $\Omega_T$. Moreover, $|\D w|^p=|\D w_+|^p$ if and only if $w=w_+$. Therefore, $w=w_+$ in $\Omega_T$, hence $w\ge0$ in $\Omega_T$.

The regularity estimates \eqref{eq:mod-conti-par} and \eqref{eq:ftn-space-par} follow from the locally uniform parabolic Hölder estimates of the solution and its gradient in Appendix~\ref{appen:par-p-Lapl-reg}; see Theorem~\ref{thm:p-reg-par}.

\medskip

\noindent 
\underline{Proof of Property \ref{cond-comp-par}:}
 As in the elliptic counterpart, we write $E(w,f_i,\tilde H)=|\D w|^p+pf\tilde H(w)$. Since $w_1\vee w_2, w_1\wedge w_2\in \cK_{g,p,\Omega_T}$,
\begin{align*}
    &\int_{\Omega_T}E(w_1,f_1,\tilde H)+p(w_1\vee w_2-w_1)\partial_t(w_1\vee w_2)\le \int_{\Omega_T}E(w_1\vee w_2,f_1,\tilde H),\\
    &\int_{\Omega_T}E(w_2,f_2,\tilde H)+p(w_1\wedge w_2-w_2)\partial_t(w_1\wedge w_2)\le \int_{\Omega_T}E(w_1\wedge w_2,f_2,\tilde H).
\end{align*}
For $\hat w:=w_2-w_1$, since $\hat w=0$ on $\partial_p\Omega_T$, we have
\begin{align*}
    &\int_{\Omega_T}(w_1\vee w_2-w_1)\partial_t(w_1\vee w_2)+(w_1\wedge w_2-w_2)\partial_t(w_1\wedge w_2)\\
    &=\int_{\{w_1<w_2\}}(w_2-w_1)\partial_tw_2+(w_1-w_2)\partial_tw_1=\int_{\{\hat w>0\}}\hat w\partial_t\hat w=\int_{\Omega_T}\partial_t\left(\frac12(\hat w_+)^2\right)\ge0.
\end{align*}
Thus, we get
\begin{align}\label{eq:par-energy-ineq}
    \int_{\Omega_T}E(w_1,f_1,\tilde H)+E(w_2,f_2,\tilde H)\le \int_{\Omega_T}E(w_1\vee w_2,f_1,\tilde H)+E(w_1\wedge w_2,f_2,\tilde H).
\end{align}
In addition, a direct calculation gives
\begin{align*}
    &((E(w_1\vee w_2,f_1,\tilde H)+E(w_1\wedge w_2,f_2,\tilde H))-(E(w_1,f_1,\tilde H)+E(w_2,f_2,\tilde H))\\
    &=\begin{cases}
        p(f_1-f_2)\int_{w_1}^{w_2}\tilde h(s)\,ds  &\text{in }\{w_1<w_2\},\\
        0&\text{in }\{w_1\ge w_2\}.
    \end{cases}
\end{align*}
Thus, if $\{w_1<w_2\}\neq\emptyset$, then \eqref{eq:par-energy-ineq} is violated. Therefore, $w_1\ge w_2$ in $\Omega_T$.

\medskip

\noindent 
\underline{Proof of Property \ref{cond-conv-par}:}
Let $v\in\cK_{g,p,\Omega_T}$ with $\supp(w-v)\Subset\{w>0\}$.
Since $w_j\to w$ locally uniformly in $\Omega_T$, for all
sufficiently large $j$ we have $w_j>0$ on $\supp(w-v)$, so
$\supp(w_j-v)\Subset\{w_j>0\}$. Hence the variational
inequality \eqref{eq:par-var-ineq} holds for $w_j$ with
this $v$:
\begin{align*}
  \int_{\Omega_T}|\nabla w_j|^p+pf^j\tilde H_j(w_j)
  +p(w_j-v)\partial_tv
  \;\leq\;
  \int_{\Omega_T}|\nabla v|^p+pf^j\tilde H_j(v),
\end{align*}
where $\tilde H_j(t)=\int_{-\infty}^t\tilde h^j(s)\,ds$.
Since $\|w_j\|_{C^{1,\al}(\supp(w-v))}\le C$ by
Theorem~\ref{thm:p-reg-par} in the Appendix, with $C$ independent of $j$,
we have $\nabla w_j\to\nabla w$ uniformly on $\supp(w-v)$.
Since $f^j\to f$ and $\tilde h^j\to\tilde h$ locally
uniformly, and $w_j\ge c>0$ on $\supp(w-v)$ for large $j$,
we also have $f^j\tilde H_j(w_j)\to f\tilde H(w)$ and
$f^j\tilde H_j(v)\to f\tilde H(v)$ uniformly on $\supp(w-v)$.
Passing to the limit $j\to\infty$:
\begin{align*}
  \int_{\Omega_T}|\nabla w|^p+pf\tilde H(w)+p(w-v)\partial_tv
  \;\leq\;
  \int_{\Omega_T}|\nabla v|^p+pf\tilde H(v).
\end{align*}
Since this holds for all $v\in\cK_{g,p,\Omega_T}$ with
$\supp(w-v)\Subset\{w>0\}$, $w$ satisfies
\eqref{eq:par-var-ineq} in $\Omega_T\cap\{w>0\}$.
\end{proof}

The proof of the counterpart to Corollary \ref{cor:p-coin} relies on the establishment of a Hopf-type lemma for the solution of the homogeneous equation and the non-degeneracy of solutions for \eqref{eq:sol-Dir-par}.  
Hopf-type lemmas for parabolic $p$-Laplace problems are restricted by the operator's sensitivity to $p$. Indeed, for $p > 2$, the diffusion becomes degenerate, allowing solutions to vanish at the boundary with arbitrary speed, which invalidates the standard Hopf Lemma. However, in the singular range $\frac{2n}{n+1} < p \leq 2$, the boundary behavior is governed by the existence of suitable barriers, as proven in \cite{AvelinGianazzaSalsa2016} and \cite{KuusiMingioneNystrom2014}, which ensure linear growth near the contact set. 

Regarding non-degeneracy, as in the fully nonlinear case, the proof of the non-degeneracy estimate for the $p$-Laplace case in \cite[(2.14)]{JeoSha24} extends to the parabolic setting in a straightforward manner.
In addition, we refer to \cite[Theorem~1.2 in Chapter {\rm III}]{DiB93} for the boundary continuity of the $p$-caloric function and to \cite[Theorem~3.1 in Chapter {\rm VI}]{DiB93} for the comparison principle for the parabolic $p$-Laplace equation.
Therefore, by confining $p$ to the interval $(\frac{2n}{n+1}, 2]$, we balance the singular and degenerate characteristics of the operator, ensuring that the solutions remain non-degenerate and satisfy the necessary boundary growth conditions. This allows the argument to proceed via a standard geometric analysis of the connected components $C_u$ and $C_v$.

    \begin{corollary}\label{cor:p-par-coin}
  Let $(u,v)$ be the solution of the system \eqref{sys:p-par}, with $ 2n/(n+1)<p\leq 2 $.
 Assume also that for  $i=1,2$, $f_i, k_i,h_i$ satisfy \eqref{eq:rhs-add-par}.  Then the connected components  $C_u$ and $C_v$ of  $\{u=0\}^\circ$ and $\{v=0\}^\circ$ satisfy
  \eqref{eq:dichotomy-par}. Moreover, for any $k_i$ and $h_i$, there exist $(f_i,g_i)$ such that for any solution pair  $(u,v)$, there exist connected components $C_u$ and $C_v$  satisfying \eqref{eq:disjoint-par}.
\end{corollary}

\section{Mixing operators and some open questions}\label{sec:mix}

Recall that, in the proof of Theorem~\ref{thm:FN} for the existence of solutions to the system with fully nonlinear elliptic operators, we verified that fully nonlinear elliptic operators satisfy \ref{cond-sol}-\ref{cond-conv}, so that Theorem~\ref{thm:FN} follows from the meta-theorem, Theorem~\ref{thm:meta}, by taking $\cL=F_1$ and $\cL=F_2$. The same argument applies to the $p$-Laplacian case. Thus, by taking $\cL_1$ to be a fully nonlinear operator and $\cL_2$ to be the $p$-Laplacian, we obtain the corresponding existence result for these mixed operator. The same argument also extends to the coincidence-set result and to the parabolic setting.

\begin{theorem}\label{eq:mixed}[Mixed Operators: Elliptic and Parabolic Cases]
    The conclusions of Theorems \ref{thm:FN} and \ref{thm:FNpar} and corollaries \ref{thm:FN-coin} and \ref{cor:FNpar-coin} hold for mixed operators, where $\cL_1$ is a fully nonlinear operator and $\cL_2$ is the $p$-Laplacian.
\end{theorem}

Regarding the limitations of this work, several important directions remain open. 

First, the case of systems with more than two components remains 
unresolved. Our approach relies on the monotonicity of the iteration 
sequence (see \eqref{eq:mon}), which does not readily generalize to systems beyond 
two components. 

Second, a complete optimality analysis—covering both optimal regularity and 
optimal non-degeneracy—remains an open problem. These properties are highly 
sensitive to the specific nature of the differential operators involved, and 
the behavior of solutions under mixed operators is currently not well 
understood. Finally, we do not address the regularity of the free boundary 
(e.g., the $C^{1,\alpha}$ regularity of $\partial\{u>0\}$). Since proving such 
regularity typically requires both optimal regularity and optimal 
non-degeneracy, we leave this as a significant challenge for future 
research.

Finally, we can try to mix operators in each equation, namely we can allow $\cL_j$ itself be combination of operators, e.g. 
$\cL_j =  \cL_{1,j} + \cL_{2,j}$. A result in this direction for spherical case, and in the spirit of \cite{ELSH25} is under investigation by \cite{Khademloo}.


\appendix

\section{Regularity of solutions for \texorpdfstring{$p$-parabolic}{p-parabolic}  case}\label{appen:par-p-Lapl-reg}
In this section, we establish the following local Hölder estimates for solutions of \eqref{eq:p-sol-par} and for their spatial gradients:

\begin{theorem}
    \label{thm:p-reg-par}
    Suppose $p\in\left(\max\left\{1, \frac{2n}{n+2}\right\},\infty\right)\setminus\{2\}$ and $w$ is a solution of \eqref{eq:p-sol-par}. Then there exists $\be=\be(n,p,a)\in(0,1)$ such that $w,\D w\in C_{\loc}^{\be,\be/2}(\Omega_T)$. Moreover, for any $K\Subset\Omega_T$,
    $$
    \|w\|_{C^{\be,\be/2}(K)}\le C\quad\text{and}\quad \|\D w\|_{C^{\be,\be/2}(K)}\le C,
    $$
    where $C>0$ is a constant depending only on $n,p,C^0,a,\|g\|_\infty,\|f\|_\infty, \Omega_T$ and $K$.
\end{theorem}

The elliptic counterpart of Theorem~\ref{thm:p-reg-par} was proved in \cite{LeiDeQTei15} for $p\ge2$ and \cite{JeoSha24} for $1<p<2$. These results were obtained by employing Campanato's method and comparing with $p$-harmonic replacements. In our parabolic setting, we also use Campanato's method and comparison with $p$-caloric replacements. However, our case is substantially more technical, mainly because $p$-caloric functions do not satisfy estimates as nice as those available for $p$-harmonic functions; see Lemmas~\ref{lem:p-caloric-reg} and \ref{lem:p-caloric-reg-deg}.

We establish the regularity for solutions of $\Delta_pw -\partial_t w=f\tilde h(w)$ for the purpose of applications. However, it will be clear from the proofs below that the same argument works for solutions of $\Delta_pw -\partial_t w = fh(w)$, where $h \in \cH$ which need not be either bounded or Lipschitz.

Although we do not provide a formal proof of Theorem~\ref{thm:p-reg-par}, it follows from the combination of Theorem~\ref{thm:par-holder-sing} - \ref{thm:par-grad-holder}. In Theorem~\ref{thm:par-holder-deg} and \ref{thm:par-grad-holder}, we prove $C^{\be,\be/p}$-type regularity rather than $C^{\be,\be/2}$. This distinction is inessential, since the exponent $\be$ is only a constant depending on $n,p,a$. In fact, the optimal regularity is not known even for $p$-caloric functions.

Due to algebraic differences, we divide the proof of Theorem~\ref{thm:p-reg-par} into the singular case $p<2$, treated in Section~\ref{subsec:sing}, and the degenerate case $p>2$, treated in Section~\ref{subsec:deg}. Since the case $p=2$ is covered by the parabolic fully nonlinear case, we only provide regularity estimates when $p\neq2$. We expect that the same results hold for $p=2$ with a simpler argument. In what follows we define  the parabolic cylinder of width $r$, height $r^\theta $ and base center $X^0 =(x_0,t_0)$ 
by   
$Q^\theta_r (X^0):= B_r(x_0)\times (t_0-r^\theta,t_0]$.

\subsection{Singular case}\label{subsec:sing}

In this section, we consider the singular case $p<2$. The main regularity estimates rely on the bound 
\begin{align}
    \label{eq:key-grad}
    \left(\dashint_{Q_r^\theta(X^0)}|\D w|^p\right)^{1/p}\lesssim r^{\sigma-1}\quad\text{for any $\sigma<1$},
\end{align}
which is the key estimate in this section. We establish \eqref{eq:key-grad} in Lemma~\ref{lem:par-u-v-est}, which is the technical core of the argument. It turns out that to achieve it, we need to first consider the case $\theta=p$ and a small $\sigma>0$, and then carefully increase $\sigma$ and $\theta$ simultaneously.

We begin by recalling the known regularity of $p$-caloric functions, which has been studied widely, e.g., \cite{CheDib88,DiB85,DiB93,DiBFri84,DiBFri85,Wie86}. In particular, we state the results from \cite{DiB93}.

\begin{lemma}
    \label{lem:p-caloric-reg}
    For $\max\{1,\frac{2n}{n+2}\}<p<2$, let $v$ be a nonnegative and bounded $p$-caloric function in $\Omega_T$. Then there are constants $C>0$ and $\al\in(0,1)$, depending only on $n$ and $p$, such that the following hold:
   
    \begin{enumerate}
    \item (Theorem~1.1 in Chapter {\rm IV}  in \cite{DiB93}) $v\in C^{\al,\al/p}_{\loc}(\Omega_T)$, and for any $K\Subset\Omega_T$
    $$
    |v(X^1)-v(X^2)|\le C\|v\|_{L^\infty(\Omega_T)}\left(\frac{\|v\|_{L^\infty(\Omega_T)}^{\frac{2-p}p}|x_1-x_2|+|t_1-t_2|^{1/p} }{p-\dist(K,\partial_p\Omega_T) }  \right)^{\al}\quad\text{for any }X^1,X^2\in K,
    $$
    where $p-\dist(K,\partial_p\Omega_T):=\inf\left\{\|v\|_{L^\infty(\Omega_T)}^{\frac{2-p}p}|x-y|+|t-s|^{1/p}\,:\, (x,t)\in K,\, (y,s)\in \partial_p\Omega_T \right\}$ is the intrinsic parabolic distance from $K$ to $\partial_p\Omega_T$.

\item (Theorem~5.2' and Remark 5.4 in Chapter {\rm VIII} in \cite{DiB93}) Whenever $B_\rho(x_0)\times(t_0-\tau,t_0]\Subset \Omega_T$,
    \begin{align*}
        \sup_{B_{\rho/2}(x_0)\times(t_0-\tau/2,t_0]}|\D v|\le C(\rho^2/\tau)^{\frac{n}{n(p-2)+2p}}\left(\dashint_{B_{\rho}(x_0)\times(t_0-\tau,t_0]}|\D v|^pdz\right)^{\frac{2}{n(p-2)+2p}}+C(\tau/\rho^2)^{\frac{1}{2-p}}.
    \end{align*}
    
    \item (Theorem~1.1'' in Chapter {\rm IX} in \cite{DiB93}) If $\D v$ is bounded in $\Omega_T$, then $\D v\in C^{\al,\al/2}_{\loc}(\Omega_T)$, and for any $K\Subset\Omega_T$
    \begin{align*}
    |\D v(X^1)-\D v(X^2)|\le C\|\D v\|_{L^\infty(\Omega_T)}\left(\frac{\max\left\{1,\|\D v\|_{L^\infty(\Omega_T)}^{\frac{2-p}2}\right\}|x_1-x_2|+|t_1-t_2|^{1/2} }{\dist(K,\partial_p\Omega_T)} \right)^\al
    \end{align*}    
   for any $X^1,X^2\in K.$
    \end{enumerate}
\end{lemma}

The following is the Caccioppoli inequality for $p$-subcaloric functions.

\begin{lemma}
    \label{lem:Caccio}
    For $Q^\theta_r(X^0)\Subset \Omega_T$, let $v$ be a nonnegative weak subsolution of $\Delta_pv-\partial_tv=0$ in $Q^\theta_r(X^0)$, i.e., for any nonnegative smooth test function $\vp$ in $Q^\theta_r(X^0)$ with $\vp=0$ on $\partial_pQ^\theta_r(X^0)$,
    $$
    \int_{Q_r^\theta(X^0)}\left(|\D v|^{p-2}\D v\cdot\D\vp-v\partial_t\vp\right)\le0.
    $$
    Then
    $$
    \left(\dashint_{Q^\theta_{r/2}(X^0)}|\D v|^{p}\right)^{1/p}\le \frac{C\|v\|_{L^\infty(Q^\theta_r(X^0))} }{r}.
    $$
\end{lemma}

\begin{proof}
Without loss of generality, we may assume $X^0=0$. Then the lemma follows from a standard argument by taking a test function $\vp(x,t)=v(x,t)\eta^p(x)\xi(t)$, where $\eta\in C^\infty_0(B_r)$ with $\eta\equiv1$ in $B_{r/2}$ and $\eta\ge0$, $|\D \eta|\le C/r$ in $B_r$, and $\xi\in C^\infty(-r^\theta,0)$ with $\xi(-r^\theta)=0$, $\xi\equiv 1$ in $(-\frac34r^\theta,0)$ and $\xi\ge0$, $|\xi'|\le \frac{C}{r^\theta}$ in $(-r^\theta,0)$. 
\end{proof}

In the remainder of this paper, we assume without loss of generality $a\in(-1,0)$ so that
$$
\gamma:=1+a\in(0,1).
$$
In the rest of this section, we assume that 
$$
\max\left\{1,\frac{2n}{n+2}\right\}<p<2
$$ 
and that $w$ is a solution of \eqref{eq:p-sol-par}, i.e., $w$ belongs to $\cK_{g,p,\Omega_T}$ and satisfies \eqref{eq:par-var-ineq}. We also fix a subset $K\Subset \Omega_T$. A universal constant depends only on $n,p,C^0,a,\|g\|_\infty,\|f\|_\infty,\Omega_T$ and $K$.

Unlike $p$-harmonic functions, $p$-caloric functions do not minimize the $p$-energy functional $\int|\D v|^p$. Nevertheless, in the following lemma, we show that their energy is still controlled by the energy of solutions to \eqref{eq:p-sol-par}.
We first recall that the $p$-parabolic replacement $v$, of a function $w$ in $Q_r^\theta(X^0)$, is a solution to the homogeneous  $p$-parabolic equation with boundary values $w$.

\begin{lemma}
    \label{lem:u-v-p-est-comp}
    If $Q_r^\theta(X^0)\Subset\Omega_T$ and $v$ is the $p$-caloric replacement of $w$ in $Q_r^\theta(X^0)$, then
    $$
    \dashint_{Q_r^\theta(X^0)}|\D v|^p\le C\left(1+\dashint_{Q_r^\theta(X^0)}|\D w|^p\right),
    $$
    where $C>0$ is a constant depending only on $n$ and $p$.
\end{lemma}

\begin{proof}
Without loss of generality, we may assume $X^0=0$. Recall $w\ge0$ in $\Omega_T$. Since $\Delta_pw-\partial_tw\ge0$ in $\Omega_T$, $w\le\sup g$ in $\Omega_T$ by the maximum principle. Thus $0\le v\le \sup g$ in $Q_r$ by the maximum/minimum principle. Then 
$$
\tilde H(v)\le \int_0^v(C^0t^a+2)\,dt\le \frac{C^0}{1+a}v^{a+1}+2v\le \frac{C^0}{1+a}\|g\|_\infty^{a+1}+2\|g\|_\infty,
$$
which along with $H(w)\ge0$, gives
\begin{align}\label{eq:u-v-est-0}
    \int_{Q_r^\theta}\left(|\D w|^p-|\D v|^p+p(w-v)\partial_tv\right)\le \int_{Q_r^\theta}pf(H(v)-H(w))\le C\int_{Q_r^\theta}H(v)\le Cr^{n+\theta}.
\end{align}
Moreover, since $v$ is $p$-caloric in $Q_r^\theta$,
$$
\int_{Q_r^\theta}\left(|\D v|^{p-2}\D v\cdot\D(w-v)+(w-v)\partial_tv  \right)=0. 
$$
Thus, for $w^s:=sw+(1-s)v$, $0\le s\le 1$, we have
\begin{align*}
    \int_{Q_r^\theta}\left(|\D w|^p-|\D v|^p\right)
    &=p\int_0^1\int_{Q_r^\theta}|\D w^s|^{p-2}\D w^s\cdot\D(w-v)\,dzds\\
    &=p\int_0^1\int_{Q_r^\theta}\left(|\D w^s|^{p-2}\D w^s\cdot\D(w-v)-|\D v|^{p-2}\D v\cdot\D(w-v)-(w-v)\partial_tv\right)dzds\\
    &=p\int_0^1\int_{Q_r^\theta}\left(|\D w^s|^{p-2}\D w^s-|\D v|^{p-2}\D v\right)\cdot \D(w-v)\,dzds-p\int_{Q_r^\theta}(w-v)\partial_tv\,dz,
\end{align*}
and hence
\begin{align}
    \label{eq:u-v-est-diff-sing}
    \begin{split}
        \int_{Q_r^\theta}\left(|\D w|^p-|\D v|^p+p(w-v)\partial_tv\right)\,dz&=p\int_0^1\int_{Q_r^\theta}\left(|\D w^s|^{p-2}\D w^s-|\D v|^{p-2}\D v\right)\cdot\D(w-v)\,dzds\\
    &=p\int_0^1\int_{Q_r^\theta}\left(|\D w^s|^{p-2}\D w^s-|\D v|^{p-2}\D v\right)\cdot\D(w^s-v)\,dz\frac{ds}s.
    \end{split}
\end{align}
By using the well-known inequality
\begin{align}\label{eq:ineq-sing}
\left(|\xi|^{p-2}\xi-|\eta|^{p-2}\eta\right)\cdot(\xi-\eta)\ge c_0|\xi-\eta|^2(|\xi|+|\eta|)^{p-2}
\end{align}
for any nonzero $\xi,\eta\in\R^n$ and a constant $c_0=c_0(n,p)>0$, we further have
\begin{align*}
    &\int_{Q_r^\theta}\left(|\D w|^p-|\D v|^p+p(w-v)\partial_tv\right)\,dz\ge c\int_0^1\int_{Q_r^\theta}|\D(w^s-v)|^2\left(|\D w^s|+|\D v|\right)^{p-2}\,dz\frac{ds}s\\
    &\ge c\int_0^1\int_{Q_r^\theta}|\D(w-v)|^2(|\D w|+|\D v|)^{p-2}\,dz\,s\,ds\ge c\int_{Q_r^\theta}|\D(w-v)|^2(|\D w|+|\D v|)^{p-2}\,dz.
\end{align*}
On the other hand, by using the Hölder's inequality, we have for any $0<\rho\le r$
\begin{align*}
    \int_{Q_\rho^\theta}|\D(w-v)|^p\le \left(\int_{Q_r^\theta}|\D(w-v)|^2(|\D w|+|\D v|)^{p-2}\right)^{p/2}\left(\int_{Q_\rho^\theta}(|\D w|+|\D v|)^p\right)^{1-p/2}.
\end{align*}
By combining the preceding two estimates and \eqref{eq:u-v-est-0}, we obtain
\begin{align}
    \label{eq:par-u-v-diff-est}
    \begin{split}
        \int_{Q_\rho^\theta}|\D(w-v)|^p&\le C\left(\int_{Q_r^\theta}(|\D w|^p-|\D v|^p+p(w-v)\partial_tv)\right)^{p/2}\left(\int_{Q_\rho^\theta}(|\D w|^p+|\D v|^p)\right)^{1-p/2}\\
        &\le Cr^{(n+\theta)p/2}\left(\int_{Q_\rho^\theta}(|\D w|^p+|\D v|^p)\right)^{1-p/2}.
    \end{split}
\end{align}
By letting $\rho=r$, this gives
\begin{align*}
    \dashint_{Q_{r}^\theta}|\D(w-v)|^p\le C\left(\dashint_{Q_{r}^\theta}(|\D w|^p+|\D v|^p)\right)^{1-p/2}.
\end{align*}
Then
\begin{align*}
    \dashint_{Q_r^\theta}|\D v|^p\le C\dashint_{Q_r^\theta}|\D w|^p+C\dashint_{Q_r^\theta}|\D(w-v)|^p\le C\dashint_{Q_r^\theta}|\D w|^p+C\left(\dashint_{Q_r^\theta}|\D w|^p\right)^{1-p/2}+C_0\left(\dashint_{Q_r^\theta}|\D v|^p\right)^{1-p/2}.
\end{align*}
For simplicity, we write $A_v:=\dashint_{Q_r^\theta}|\D v|^p$ and $A_w:=\dashint_{Q_r^\theta}|\D w|^p$. Then, the above estimate can be rewritten as
$$
A_v\le C_0A_v^{1-p/2}+C\left(A_w+A_w^{1-p/2}\right)\le C_0A_v^{1-p/2}+C(1+A_w).
$$
If $A_v-C_0A_v^{1-p/2}>\frac{A_v}2$. then we have by the previous inequality $A_v\le C(1+A_w)$. On the other hand, if $A_v-C_0A_v^{1-p/2}\le \frac{A_v}{2}$, then $A_v\le 2C_0A_v^{1-p/2}$, thus $A_v\le (2C_0)^{2/p}\le (2C_0)^{2/p}(1+A_w)$. This completes the proof.    
\end{proof}

\begin{lemma}\label{lem:iter-sing-1}
    For $X^0\in K$, suppose that there are universal constants $C_0>0$, $r_0>0$ and $\theta\in[p,2)$ such that
    $$
    \dashint_{Q_r^\theta(X^0)}|\D w|^p\le C_0r^{p(\de-1)}\quad\text{whenever $X^0\in K$ and $0<r<r_0$}.
    $$
    Then, there are universal constants $C_1>0$ and $r_1>0$ such that for any $0<\e<1$, $0<r<r_1$ and $\theta'\in[\theta,2)$ with $Q_r^{\theta'}(X^0)\Subset \Omega_T$,
    $$
    \dashint_{Q_r^{\theta'}(X^0)}|\D w|^p\le Cr^{p(\tilde\de-1)}.
    $$
    Here,
    \begin{align}\label{eq:degree}
        \tilde\de:=\min\{\tilde\de_1,\tilde\de_2,\tilde\de_3\}, 
    \end{align}
    where
    \begin{align*}
        &\tilde\de_1:=\frac{2\de+p(1-\de)-\frac{(2-p)(\theta'-\theta)}{p} }{2-\gamma}-\frac{\e}{1+\e}\left(n+\theta'+\frac{\gamma-(2-p)(1-\de+\frac{\theta'-\theta}{p}) }{2-\gamma}  \right),\\
        &\tilde\de_2:=\frac{2p\de-2(\theta'-\theta)-n(\theta'-p)}{n(p-2)+2p}-\frac\e{1+\e}\cdot\frac{2[n+\theta+p(\de-1)]-(n+2)\theta'}{n(p-2)+2p},\\
        &\tilde\de_3:=\frac{\theta'-p}{2-p}+\frac{\e}{1+\e}\cdot\frac{2-\theta'}{2-p}.
    \end{align*}
\end{lemma}

\begin{proof}
Without loss of generality, we may assume $X^0=0$. Let $v$ be the $p$-caloric replacement of $w$ in $Q_r^{\theta'}$. We use the first inequality in \eqref{eq:par-u-v-diff-est} with $\rho=r$ to get
\begin{align}
    \label{eq:par-u-v-diff-est-2}
    \int_{Q_r^{\theta'}}|\D(w-v)|^p\le C\left(\int_{Q_r^{\theta'}}\left(|\D w|^p-|\D v|^p+p(w-v)\partial_tv\right)\right)^{p/2}\left(\int_{Q_r^{\theta'}}\left(|\D w|^p+|\D v|^p\right)\right)^{1-p/2}.
\end{align}
Concerning the right-hand side in this inequality, we have by Lemma~\ref{lem:u-v-p-est-comp}
\begin{align}
    \label{eq:par-u-v-est-2}
    \begin{split}
        \int_{Q_r^{\theta'}}(|\D w|^p+|\D v|^p)&\le Cr^{n+\theta'}+C\int_{Q_r^{\theta'}}|\D w|^p\le Cr^{n+\theta'}+C\int_{Q_r^\theta}|\D w|^p\\
        &\le Cr^{n+\theta'}+Cr^{n+\theta+p(\de-1)}\le Cr^{n+\theta+p(\de-1)}. 
    \end{split}
\end{align}
Moreover, by recalling $\gamma=1+a\in(0,1)$ and using \cite[Lemma~2.5]{LeiDeQTei15}, we also have

\begin{align*}
    \int_{Q_r^{\theta'}}\left(|\D w|^p-|\D v|^p+p(w-v)\partial_tv\right)&\le \int_{Q_r^{\theta'}}pf(\tilde H(v)-\tilde H(w))=\int_{Q_r^{\theta'}}pf\int_{w(z)}^{v(z)}\tilde h(s)\,ds\,dz\\
    &\le \int_{Q_r^{\theta'}\cap\{v>w\}}pf\int_{w(z)}^{v(z)}\left(C_0s^a+2\right)\,ds\,dz\\
    &\le C\int_{Q_r^{\theta'}\cap\{v>w\}}(v^\gamma-w^{\gamma})+(v-w)\\
    &\le C\int_{Q_r^{\theta'}\cap\{v>w\}}|w-v|^\gamma+|v-w|\le C\int_{Q_r^{\theta'}}|w-v|^\gamma.
\end{align*}
Note that by Hölder's inequality and Poincaré inequality,
$$
\left(\dashint_{Q_r^{\theta'}}|w-v|^\gamma\right)^{1/\gamma}\le Cr\left(\dashint_{Q_r^{\theta'}}|\D(w-v)|^p\right)^{1/p},
$$
thus we have
\begin{align}\label{eq:u-v-min-ineq}
    \int_{Q_r^{\theta'}}\left(|\D w|^p-|\D v|^p+p(w-v)\partial_tv\right)\le Cr^{(n+\theta')(1-\gamma/p)+\gamma}\left(\int_{Q_r^{\theta'}}|\D(w-v)|^p\right)^{\gamma/p}. 
\end{align}
Combining this inequality with \eqref{eq:par-u-v-diff-est-2} and \eqref{eq:par-u-v-est-2} gives that for $A:=\int_{Q_r^{\theta'}}|\D(w-v)|^p$,
$$
A\le C\left(r^{(n+\theta')(1-\gamma/p)+\gamma}A^{\gamma/p}\right)^{p/2}r^{[n+\theta+p(\de-1)](1-p/2)}.
$$
Thus
\begin{align}
    \label{eq:par-u-v-diff-est-3}
    \int_{Q_r^{\theta'}}|\D(w-v)|^p=A\le Cr^{n+\theta'+p\frac{\gamma-(2-p)(1-\de+\frac{\theta'-\theta}{p})}{2-\gamma}}.
\end{align}

Next, we choose $r_1\in(0,1)$ small so that $r_1^\e\le 1/4$. If $0<r<r_1$, then we have by applying Lemma~\ref{lem:p-caloric-reg} with $\rho=r$ and $\tau=r^{\theta'}$ and \eqref{eq:par-u-v-est-2} that
\begin{align*}
    \int_{Q_{r^{1+\e}}^{\theta'}}|\D v|^p&\le Cr^{(n+\theta')(1+\e)}\|\D v\|_{L^\infty(Q_{r/2}^{\theta'})}^p\\
    &\le Cr^{(n+\theta')(1+\e)}\left[\left(r^{2-\theta'}\right)^{\frac{n}{n(p-2)+2p}}\left(\dashint_{Q_r^{\theta'}}|\D v|^p\right)^{\frac{2}{n(p-2)+2p}}+r^{\frac{\theta'-2}{2-p}}\right]^p\\
    &\le Cr^{(n+\theta')(1+\e)}\left[r^{\frac{n(2-\theta')-2(n+\theta')+2[n+\theta+p(\de-1)]}{n(p-2)+2p}}+r^{\frac{\theta'-2}{2-p}} \right]^p\\
    &\le Cr^{(n+\theta')(1+\e)+p\frac{2[n+\theta+p(\de-1)]-(n+2)\theta' }{n(p-2)+2p  }  }+Cr^{(n+\theta')(1+\e)+p\frac{\theta'-2}{2-p} }.
\end{align*}
By combining this with \eqref{eq:par-u-v-diff-est-3}, we infer
\begin{align*}
    \int_{Q_{r^{1+\e}}^{\theta'}}|\D w|^p&\le C\int_{Q^{\theta'}_{r^{1+\e}}}|\D(w-v)|^p+C\int_{Q^{\theta'}_{r^{1+\e}}}|\D v|^p\\
    &\le Cr^{n+\theta'+p\frac{\gamma-(2-p)(1-\delta+\frac{\theta'-\theta}{p}) }{2-\gamma} }+Cr^{(n+\theta')(1+\e)+p\frac{2[n+\theta+p(\de-1)]-(n+2)\theta'}{n(p-2)+2p}}+Cr^{(n+\theta')(1+\e)+p\frac{\theta'-2}{2-p} }.
\end{align*}
Taking $\rho=r^{1+\e}$ yields
$$
\dashint_{Q_{\rho}^{\theta'}}|\D w|^p\le C\rho^{p(\tilde\de-1)},
$$
where $\tilde\de$ is as in \eqref{eq:degree}.
\end{proof}

We use Lemma~\ref{lem:iter-sing-1} to show that the exponent $\sigma$ in \eqref{eq:key-grad} can be improved from $\de$ to $\de'$; see \eqref{eq:del}.

\begin{lemma}
    \label{lem:iter-sing-2}
    For $X^0\in K$, suppose $p\le\theta<\theta'<2$ and $0<\e<1$. We let
    \begin{align}\label{eq:del}
    \de:=\frac{\theta-p}{2-p}+\frac{\e}{1+\e}\cdot\frac{2-\theta'}{2-p},\qquad \de':=\de+\frac{\theta'-\theta}{(2-p)(1+\e)}.
    \end{align}
    Assume that $\theta'-\theta$ and $\e$ are small with $\theta'-\theta\ll \e$ so that
    \begin{align}
        \label{eq:iter-cond}
            \frac{4}{2-p}(\theta'-\theta)+7n\e\le \gamma,\qquad
            \de'<1,\qquad
            \theta'-\theta\le\frac{(2-\theta')^2}{14}\e.
    \end{align}
If 
$$
\dashint_{Q_r^\theta(X^0)}|\D w|^p\le C_0r^{p(\de-1)}
$$
for any $0<r<r_0$, where $C_0>0$ and $r_0>0$ are universal constants, then there exist universal constants $C_1>0$ and $r_1>0$ such that
$$
\dashint_{Q_r^{\theta'}(X^0)}|\D w|^p\le C_1r^{p(\de'-1)}
$$
whenever $0<r< r_1$.
\end{lemma}

\begin{proof}
In view of Lemma~\ref{lem:iter-sing-1}, it is sufficient to prove
$$
\de'\le\tilde\de,
$$
where $\tilde\de=\min\{\tilde\de_1,\tilde\de_2,\tilde\de_3\}$ is as in \eqref{eq:degree}. We divide its proof into three steps.

\medskip\noindent\emph{Step 1.} In this step, we show that $\de'\le\tilde\de_1$, i.e.,
$$
\de+\frac{\theta'-\theta}{(2-p)(1+\e)}\le \frac{2\de+p(1-\de)-\frac{(2-p)(\theta'-\theta)}p }{2-\gamma }-\frac{\e}{1+\e}\left[n+\theta'+\frac{\gamma-(2-p)(1-\de+\frac{\theta'-\theta}p) }{2-\gamma} \right].
$$
Indeed, a direct computation gives that this is equivalent to
\begin{align*}
&(p-\gamma)\de\\
&\le p-\frac{(\theta'-\theta)(2-\gamma)}{(2-p)(1+\e)}-\frac{(2-p)(\theta'-\theta)}p-\frac{\e}{1+\e}\left[(n+\theta')(2-\gamma)+\gamma-(2-p)(1-\de+\frac{\theta'-\theta}p) \right].\\
\end{align*}
This holds since
\begin{align*}
    &\frac{(\theta'-\theta)(2-\gamma)}{(2-p)(1+\e)}+\frac{(2-p)(\theta'-\theta)}p+\frac{\e}{1+\e}\left[(n+\theta')(2-\gamma)+\gamma-(2-p)(1-\de+\frac{\theta'-\theta}p)\right]\\
    &\le \frac{2}{2-p}(\theta'-\theta)+(2-p)(\theta'-\theta)+\e[(n+2)2+1]\le\frac4{2-p}(\theta'-\theta)+7n\e\le \gamma.
\end{align*}

\medskip\noindent\emph{Step 2.} In this step, we prove $\de'\le \tilde\de_2$.  We first  observe that
$$
\tilde\de_2=1+\frac{\mu}{1+\e}=1+\mu-\frac{\e}{1+\e}\mu,\quad\text{where }\mu=\frac{n(2-\theta')-2[\theta'-\theta+p(1-\de)] }{n(p-2)+2p }.
$$
If $\mu\ge0$, then $\tilde\de_2\ge1$, which along with the condition $\de'<1$ gives $\de'<\tilde\de_2$. 
Thus we may assume $\mu<0$, which implies $\tilde\de_2>1+\mu$. Then it suffices to show $\de'\le 1+\mu$, which is equivalent to
$$
\de+\frac{\theta'-\theta }{(2-p)(1+\e) }\le\frac{2p-n(\theta'-p)-2[\theta'-\theta+p(1-\de)] }{n(p-2)+2p }.
$$
A direct calculation gives that it is equivalent to
$$
\de\ge\frac{\theta-p}{2-p}+\left[\frac{n(p-2)+2p}{n(2-p)(1+\e)}+\frac2n+1 \right]\frac{\theta'-\theta}{2-p}.
$$
Since $\de=\frac{\theta-p}{2-p}+\frac{\e}{1+\e}\cdot\frac{2-\theta'}{2-p}$, it also equals to
$$
\frac{\e}{1+\e}(2-\theta')\ge\left[\frac{n(p-2)+2p}{n(2-p)(1+\e)}+\frac2n+1 \right](\theta'-\theta).
$$
This inequality holds thanks to the last condition on \eqref{eq:iter-cond} since
\begin{align*}
    \left[\frac{n(p-2)+2p}{n(2-p)(1+\e)}+\frac2n+1 \right](\theta'-\theta)&\le \left[\frac{2p}{n(2-p)(1+\e)}+3 \right](\theta'-\theta)\\
    &\le \frac7{2-p}(\theta'-\theta)\le\frac{7}{2-\theta'}(\theta'-\theta)
\end{align*}
and $
\frac{2-\theta'}{1+\e}\ge\frac{2-\theta'}2.
$

\medskip\noindent\emph{Step 3.} In this last step, we prove $\de'\le \tilde\de_3$. Indeed, we simply have
\begin{align*}
    \de'&=\frac{\theta-p}{2-p}+\frac{\e}{1+\e}\cdot\frac{2-\theta'}{2-p}+\frac{\theta'-\theta}{(2-p)(1+\e)}\\
    &\le\frac{\theta-p}{2-p}+\frac{\e}{1+\e}\cdot\frac{2-\theta'}{2-p}+\frac{\theta'-\theta}{2-p}=\frac{\theta'-p}{2-p}+\frac{\e}{1+\e}\cdot\frac{2-\theta'}{2-p}=\tilde\de_3.
\end{align*}
This completes the proof.   
\end{proof}

Next, we prove that \eqref{eq:key-grad} holds for small $\sigma>0$ and $\theta=p$.

\begin{lemma}
    \label{lem:par-u-est-initial}
    There exist universal constants $C>0$ and $r_0>0$ such that
    $$
    \dashint_{Q_r^p(X^0)}|\D w|^p\le Cr^{p\left(\frac{p}{n+2+p}-1 \right)}
    $$
    whenever $X^0\in K$ and $0<r< r_0$.
\end{lemma}

\begin{proof}Without loss of generality, we may assume $X^0=0$. We take a small constant $\e\in(0,1)$ to be chosen later, and then choose $r_0\in(0,1)$ small so that $r_0^\e\le 1/4$. For $0<r<r_0$, since $\Delta_pw-\partial_tw=f\tilde h(w)\ge 0$ and $0\le w\le \sup g$, we have by Lemma~\ref{lem:Caccio},
\begin{align}
    \label{eq:par-grad-u-p-est}
    \int_{Q^p_{r^{1+\e}}}|\D w|^p\le Cr^{n(1+\e)}.
\end{align}
To improve it, we apply Lemma~\ref{lem:Caccio} and use $0\le v\le \sup g$ in $Q_r^p$ to get
\begin{align}
    \label{eq:par-grad-v-sup-est}
    \left(\dashint_{Q_{r/2}^p}|\D v|^p\right)^{1/p}\le \frac{C\|v\|_{L^\infty(Q_{r}^p)}}{r}\le  \frac Cr,
\end{align}
which, along with Lemma~\ref{lem:p-caloric-reg}, yields

\begin{align}\label{eq:grad-v-p-est}
        \int_{Q^p_{r^{1+\e}}}|\D v|^p &\le Cr^{(n+p)(1+\e)}\|\D v\|_{L^\infty(Q^p_{r/4})}^p\\&\le Cr^{(n+p)(1+\e)+\frac{(2-p)np}{n(p-2)+2p}}\left(\dashint_{Q^p_{r/2}}|\D v|^p\right)^{\frac{2p}{n(p-2)+2p}}+Cr^{(n+p)(1+\e)-p}\\
        &\le Cr^{(n+p)(1+\e)+\frac{(2-p)np}{n(p-2)+2p}-\frac{2p^2}{n(p-2)+2p}}+Cr^{(n+p)(1+\e)-p}\le Cr^{(n+p)(1+\e)-p}.
\end{align}

By using \eqref{eq:par-u-v-diff-est}, \eqref{eq:par-grad-u-p-est} and \eqref{eq:grad-v-p-est}, we get
\begin{align}
    \label{eq:par-grad-u-ep-est}\begin{split}
    \int_{Q^p_{r^{1+\e}}}|\D w|^p&\le C\int_{Q^p_{r^{1+\e}}}|\D(w-v)|^p+C\int_{Q^p_{r^{1+\e}}}|\D v|^p\\
    &\le Cr^{(n+p)p/2}\left(\int_{Q^p_{r^{1+\e}}}(|\D w|^p+|\D v|^p)\right)^{1-p/2}+Cr^{(1+\e)(n+p)-p }\\
    &\le Cr^{(n+p)p/2}\left(r^{n(1+\e)}+r^{(n+p)(1+\e)-p}\right)^{1-p/2}+Cr^{(1+\e)(n+p)-p}.
    \end{split}
\end{align}
By letting $\rho=r^{1+\e}$, we obtain
$$
\int_{Q^p_\rho}|\D w|^p\le C\left(\rho^{n+\frac{p(p-\e n)}{2(1+\e)}}+\rho^{n+\frac{\e p}{1+\e}}\right).
$$
so that $\frac{p(p-\e n)}{2(1+\e)}=\frac{\e p}{1+\e}$ to obtain that for $\de:=\frac{p}{n+2+p}\in(0,1)$
\begin{align}
    \label{eq:par-grad-u-est-1}
    \int_{Q^p_\rho}|\D w|^p\le C\rho^{n+p\de}.
\end{align}
This completes the proof.
\end{proof}

We now derive \eqref{eq:key-grad} by using Lemma~\ref{lem:par-u-est-initial} and repeatedly applying Lemma~\ref{lem:iter-sing-2}.

\begin{lemma}\label{lem:par-u-v-est}
    For any $\sigma\in (0,1)$, there are $\theta\in(p,2)$, $C>0$ and $r_0>0$, depending only on $n,p,C^0,a,\|g\|_\infty,\|f\|_\infty, \Omega_T,K, \sigma$, such that if $X^0\in K$ and $0<r<r_0$, then
    \begin{align*}
        \dashint_{Q_r^\theta(X^0)}|\D w|^p\le Cr^{p(\sigma-1)}.
    \end{align*}
    Here, for each $\sigma\in(0,1)$, $\theta\in(p,2)$ can be chosen so that $\theta\to2$ as $\sigma\to1$.
\end{lemma}

\begin{proof}
We let
\begin{align*}
&\e:=\min\left\{\frac{p}{n+2+p}, \frac{\gamma}{14n},\sigma, 1-\sigma\right\}\in(0,1/14),\\
&\theta^\sharp:=\frac{(2-p)\sigma+p}2+1=1+\sigma+\frac{p}2(1-\sigma)\in(p,2),\\
&\nu:=\min\left\{\frac{2-\theta^\sharp}2, \frac{\left(2-\frac{\theta^\sharp+2}{2} \right)^2}{14}\e, \frac{2-p}8\gamma \right\}\in(0,1).
\end{align*}
We then consider an affine function
$$
\psi(\theta):=\frac{\theta-p+\e(2-p-\nu) }{(1+\e)(2-p) }=\frac{\theta-p}{2-p}+\frac{\e}{1+\e}\cdot\frac{2-\theta-\nu}{2-p}.
$$
A direct computation gives
$$
\psi^{-1}(\sigma)=p+\sigma(2-p)-\e[(2-p)(1-\sigma)-\nu].
$$
Since $\nu<\frac{2-\theta^\sharp}2=\frac{(2-p)(1-\sigma)}4<(2-p)(1-\sigma)$ and $(2-p)\sigma+p<(2-p)+p=2$,
$$
\psi^{-1}(\sigma)\le p+\sigma(2-p)<\frac{p+\sigma(2-p)}2+1=\theta^\sharp.
$$
Moreover, since $\psi(p)=\frac{\e(2-p-\nu)}{(1+\e)(2-p)}<\frac{\e}{1+\e}<\sigma$ and $\psi$ is increasing, $p<\psi^{-1}(\sigma)<\theta^\sharp<2$. We consider sequences $\theta_j:=p+\nu j$ and $\de_j:=\psi(p+\nu j)$, $j\in \mathbb{N}\cup\{0\}$. Since $\nu\le\frac{2-\theta^\sharp}2$, there exists $J\in \mathbb{N}$ such that $\theta^\sharp\le\theta_J<\frac{\theta^\sharp+2}2$. 

We claim that there exist universal constants $C>0$ and $r_0>0$ such that
\begin{align}
    \label{eq:par-grad-u-est-seq}
    \dashint_{Q_r^{\theta_j}(X^0)}|\D w|^p\le Cr^{p(\de_j-1)}
\end{align}
whenever $X^0\in K$, $0<r< r_0$ and $0\le j\le J$. We prove it by induction on $j$. Since $\theta_0=p$ and $\de_0=\psi(p)=\frac{\e(2-p-\nu)}{(1+\e)(2-p)}<\e\le \frac{p}{n+2+p}$, \eqref{eq:par-grad-u-est-seq} holds for $j=0$ by Lemma~\ref{lem:par-u-est-initial}. Now we assume that \eqref{eq:par-grad-u-est-seq} holds for $j\in \{0,1,2,\ldots,J-1\}$ and prove it for $j+1$. Now we apply Lemma~\ref{lem:iter-sing-2} with $\de=\de_j$, $\de'=\de_{j+1}$, $\theta=\theta_j$ and $\theta'=\theta_{j+1}$. To prove that it is applicable, we need to verify \eqref{eq:del} and \eqref{eq:iter-cond}. For \eqref{eq:del}, we use the definition of $\psi$ and $\theta'=\theta_{j+1}=\theta_j+\nu$ to have
\begin{align*}
    &\de_j=\psi(\theta_j)=\frac{\theta_j-p}{2-p}+\frac{\e}{1+\e}\cdot\frac{2-\theta_j-\nu}{2-p}=\frac{\theta-p}{2-p}+\frac{\e}{1+\e}\cdot\frac{2-\theta'}{2-p}=\de,\\
    &\de_{j+1}=\psi(\theta_{j+1})=\psi(\theta_j)+\frac{\theta_{j+1}-\theta_j}{(1+\e)(2-p)}=\de+\frac{\theta'-\theta}{(1+\e)(2-p)}=\de'.
\end{align*}
Concerning \eqref{eq:iter-cond}, the first condition follows from $\e\le \frac{\gamma}{14n}$  and $\nu\le \frac{2-p}8\gamma$. The second one follows from 
$$
\de'=\de_{j+1}\le \de_J=\psi(\theta_J)<\psi(2)=1-\frac{\e}{1+\e}\cdot\frac{\nu}{2-p}<1.
$$
Regarding the last one, we use $\nu\le\frac{\left(2-\frac{\theta^\sharp+2}2 \right)^2}{14}\e$ and $\theta'=\theta_{j+1}\le \theta_J<\frac{\theta^\sharp+2}2$ to get
$$
\theta'-\theta=\nu\le \frac{\left(2-\frac{\theta^\sharp+2}2 \right)^2}{14}\e<\frac{(2-\theta')^2}{14}\e.
$$
This completes the proof of the claim \eqref{eq:par-grad-u-est-seq}.

Now, by \eqref{eq:par-grad-u-est-seq}, we have
$$
\dashint_{Q_r^{\theta_J}(X^0)}|\D w|^p\le Cr^{p(\de_J-1)}.
$$
Since $\de_J=\psi(\theta_J)\ge \psi(\theta^\sharp)\ge \psi(\psi^{-1}(\sigma))=\sigma$, we conclude
$$
\dashint_{Q_r^{\theta_J}(X^0)}|\D w|^p\le Cr^{p(\sigma-1)}.
$$

 It remains to show that $\theta\to 2$ as $\sigma\to1$. For this aim, we recall $\psi^{-1}(\sigma)=p+\sigma(2-p)-\e[(2-p)(1-\sigma)-\nu]$. From $0<\e\le 1-\sigma$, we see that $\e\to0$ as $\sigma\to1$. Therefore, $\psi^{-1}(\sigma)\to2$ as $\sigma\to1$. 
\end{proof}

We are now ready to prove the Hölder estimates for $w$ and $\D w$.

\begin{theorem}\label{thm:par-holder-sing}
$w\in C^{\be,\be/\theta}(K)$ for some $\theta\in[p,2)$, where $\be:=\frac{\al p}{n+(3+\al)p}$. Moreover, 
 $$
 \|w\|_{C^{\be,\be/\theta}(K)}\le C $$
 for some universal constant $C>0$.
\end{theorem}

\begin{proof}
It suffices to show that for a fixed $X^0\in K$,
$$
\left(\dashint_{Q^\theta_\rho(X^0)}|w-\mean{w}_{Q^\theta_\rho(X^0)}|^p\right)^{1/p}\le C\rho^\al \quad\text{for any }0<\rho<\rho_0,
$$
where $C>0$ and $\rho_0>0$ are universal constants.

Without loss of generality, we may assume $X^0=0$. For $0<r<1/2$, let $v$ be the $p$-caloric replacement of $w$ in $Q_r^\theta$. For $\e\in(0,1)$ to be chosen later, we apply Poincaré inequality and Lemmas~\ref{lem:u-v-p-est-comp} and \ref{lem:par-u-v-est} to obtain that for some $\theta\in[p,2)$, if $r\le 2^{-1/\e}$, then
\begin{align*}
    \int_{Q_r^\theta}|w-v|^p&=\int_{-r^\theta}^0\int_{B_r}|w(x,t)-v(x,t)|^p\,dxdt\le Cr^p\int_{-r^\theta}^0\int_{B_r}|\D (w-v)|^p\,dxdt\\
    &=Cr^p\int_{Q_r^\theta}|\D(w-v)|^p\,dz\le Cr^p\int_{Q_r^\theta}(|\D w|^p+|\D v|^p) \le Cr^{n+\theta+p(1-\e)}.
\end{align*}
By using this estimate, together with Jensen's inequality and Lemma~\ref{lem:p-caloric-reg}, we have
\begin{align}\label{eq:par-holder-ineq-sing}\begin{split}
    \int_{Q^\theta_{r^{1+\e}}}|w-\mean{w}_{Q^\theta_{r^{1+\e}}}|^p
    &\le C\int_{Q^\theta_{r^{1+\e}}}|v-\mean{v}_{Q^\theta_{r^{1+\e}}}|^p+C\int_{Q^\theta_{r^{1+\e}}}|w-v|^p\\
    &\le Cr^{(1+\e)(n+\theta)+p\e\al}+Cr^{n+\theta+p(1-\e)}.\end{split}
\end{align}
Taking $\rho=r^{1+\e}$, we further have
$$
\dashint_{Q_\rho^\theta}|w-\mean{w}_{Q^\theta_\rho}|^p\le C\rho^{\frac{p\e\al }{1+\e }}+C\rho^{\frac{(1-\e) p-\e(n+p)}{1+\e}}.
$$
We take $\e=\frac{p}{n+(2+\al)p}$ so that $p\e\al=(1-\e)p-\e(n+p)$, and get
$$
\left(\dashint_{Q_\rho^\theta}|w-\mean{w}_{Q^\theta_\rho}|^p\right)^{1/p}\le C\rho^{\frac{\al}{n+(3+\al)p}}.
$$
This completes the proof.    
\end{proof}

\begin{theorem}
   $\D w\in C^{\hat\be,\hat\be/2}(K)$ for some $\hat\be=\hat\be(n,p,a)\in(0,1)$. Moreover, 
  $$
  \|\D w\|_{C^{\hat\be,\hat\be/2}(K)}\le C
  $$
  for some universal constant $C>0$.
\end{theorem}

\begin{proof}
Throughout this proof, we write
\begin{align*}
    \e:=\frac{\gamma p}{4(2-\gamma)(n+2)}\in (0,1),\quad \nu:=\frac{\e\al}4\in (0,1).
\end{align*}
We also consider $\theta=\theta(n,p,a)\in[p,2)$ close to $2$ to be determined later. In particular, we ask
\begin{align}
    \label{eq:theta-cond}
    \theta\ge 2(1-\e^2),\quad \frac{\theta-2}{2-p}\ge-\nu.
\end{align}
We also take $r_0\in(0,1)$ small so that $r_0^\e\le 1/4$. For a parabolic cylinder $Q_r^\theta(X^0)\Subset \Omega_T$ with $X^0\in K$ and $0<r<r_0$, let $v$ be the $p$-caloric replacement of $w$ in $Q_r^\theta(X^0)$. For simplicity, we write $X^0=0$. Thanks to Lemma~\ref{lem:par-u-v-est}, we can apply \eqref{eq:par-u-v-diff-est-3} with $\de=\de(p,\gamma)\in(0,1)$ close to $1$ and $\theta'=\theta$ to get
\begin{align}
    \label{eq:par-u-v-diff-est-4}
    \int_{Q_r^\theta}|\D(w-v)|^p\le Cr^{n+\theta+p\frac{\gamma}{2(2-\gamma)}}.
\end{align}
In addition, we have by Lemmas \ref{lem:p-caloric-reg} and \ref{lem:u-v-p-est-comp} that
\begin{align*}
    \|\D v\|_{L^\infty(Q^\theta_{r/2})}&\le Cr^{\frac{n(2-\theta)}{n(p-2)+2p}}\left(\dashint_{Q_r^\theta}|\D v|^p\right)^{\frac2{n(p-2)+2p}}+Cr^{\frac{\theta-2}{2-p}}\le C\left(1+\dashint_{Q_r^\theta}|\D w|^p \right)^{\frac2{n(p-2)+2p}}+Cr^{\frac{\theta-2}{2-p}}.
\end{align*}
We recall $\frac{\theta-2}{2-p}\ge -\nu$ and apply Lemma~\ref{lem:par-u-v-est} with $\sigma\in(0,1)$ close to $1$ so that $\frac{2p(\sigma-1)}{n(p-2)+2p}>-\nu$ and $\theta$ satisfies \eqref{eq:theta-cond}. This is possible since $\theta\to 2$ as $\sigma\to1$. Then we have
$$
\|\D v\|_{L^\infty(Q^\theta_{r/2})}\le Cr^{-\nu}.
$$
We then apply Lemma~\ref{lem:p-caloric-reg} to obtain that for any $X^1,X^2\in Q^\theta_{r^{1+\e}}$,
\begin{align*}
    |\D v(X^1)-\D v(X^2)|&\le C\|\D v\|_{L^\infty(Q_{r/2}^\theta)}\left(\frac{\max\left\{1,\|\D v\|_{L^\infty(Q_{r/2}^\theta)}\right\}^{\frac{2-p}2}|x_1-x_2|+|t_1-t_2|^{1/2} }{\dist(Q^\theta_{r^{1+\e}},\partial_pQ_{r/2}^\theta) } \right)^\al\\
    &\le Cr^{-\nu}\left(r^{\e-\frac{2-p}2\nu}+r^{\frac{\theta(1+\e)}{2}-1} \right)^{\al}\le Cr^{\al\e-\left(1+\frac{2-p}2\al\right)\nu }+Cr^{\left(\frac{\theta(1+\e)}2-1\right)\al-\nu}. 
\end{align*}
Here, from $\nu=\frac{\e\al}4$, we see that $\al\e-\left(1+\frac{2-p}2\al\right)\nu\ge \al\e-2\nu\ge\frac{\al\e}4$. Moreover, since $\theta\ge2(1-\e^2)$ and $\e=\frac{\gamma p}{4(2-\gamma)(n+2)}\le 1/6$, we have $\frac{\theta(1+\e)}2-1\ge (1-\e^2)(1+\e)-1=\e(1-\e-\e^2)\ge\e/2$, thus $\left(\frac{\theta(1+\e)}2-1\right)\al-\nu\ge\frac{\al\e}2-\frac{\al\e}4=\frac{\al\e}4$. Thus,
$$
|\D v(X^1)-\D v(X^2)|\le Cr^{\frac{\al\e}4}.
$$
Now, by this estimate, together with Jensen's inequality and \eqref{eq:par-u-v-diff-est-4}, 
\begin{align*}
    \int_{Q^\theta_{r^{1+\e}}}|\D w-\mean{\D w}_{Q^\theta_{r^{1+\e}}}|^p\le C\int_{Q_{r^{1+\e}}^\theta}|\D v-\mean{\D v}_{Q^\theta_{r^{1+\e}}}|^p+C\int_{Q_{r^{1+\e}}^\theta}|\D(w-v)|^p\le Cr^{(n+\theta)(1+\e)+p\frac{\al\e}{4}}+Cr^{n+\theta+p\frac{\gamma}{2(2-\gamma)}}.
\end{align*}
Letting $\rho=r^{1+\e}$, we get
\begin{align*}
    \dashint_{Q_\rho^\theta}|\D w-\mean{\D w}_{Q_\rho^\theta}|^p\le C\rho^{p\frac{\al\e}{4(1+\e)}}+C\rho^{p\frac{\frac{\gamma}{2(2-\gamma)}-\frac{\e}p(n+\theta) }{1+\e}}.
\end{align*}
It is easily seen from $\e=\frac{\gamma p}{4(2-\gamma)(n+2)}$ that $\frac{\gamma}{2(2-\gamma)}-\frac{\e}p(n+\theta)\ge\frac{\gamma}{4(2-\gamma)}$. Thus,
\begin{align*}
    \left(\dashint_{Q_\rho^\theta}|\D w-\mean{\D w}_{Q^\theta_\rho}|^p\right)^{1/p}\le C\rho^{\hat\be},
\end{align*}
where $\hat\be=\frac1{1+\e}\min\left\{\frac{\al\e}4,\frac{\gamma}{4(2-\gamma)} \right\}\in(0,1)$. This completes the proof.
\end{proof}

\subsection{Degenerate case}\label{subsec:deg}
In this section, we establish regularity estimates when $p>2$. This degenerate case is algebraically simpler than the singular case treated in Section~\ref{subsec:sing}. Since the arguments are largely analogous, we will sometimes omit the repetitive details and focus on the algebraic differences.

\begin{lemma}
    \label{lem:p-caloric-reg-deg}
    For $p>2$, let $v$ be a nonnegative and bounded $p$-caloric function in $\Omega_T$. Then there exist constants $C>0$ and $\al\in(0,1)$, depending only on $n$ and $p$, such that the following hold:
    \begin{enumerate}    
    \item (Theorem~1.1 in Chapter {\rm III}  in \cite{DiB93}) $v\in C^{\al,\al/p}_{\loc}(\Omega_T)$, and for any $K\Subset\Omega_T$
    $$
    |v(X^1)-v(X^2)|\le C\|v\|_{L^\infty(\Omega_T)}\left(\frac{|x_1-x_2|+\|v\|_{L^\infty(\Omega_T)}^{\frac{p-2}p}|t_1-t_2|^{1/p} }{p-\dist(K,\partial_p\Omega_T) }  \right)^{\al}\quad\text{for any }X^1,X^2\in K,
    $$
    where $p-\dist(K,\partial_p\Omega_T):=\inf\left\{|x-y|+\|v\|_{L^\infty(\Omega_T)}^{\frac{p-2}p}|t-s|^{1/p}\,:\, (x,t)\in K,\, (y,s)\in \partial_p\Omega_T \right\}$ is the intrinsic parabolic distance from $K$ to $\partial_p\Omega_T$.

    \item (Theorem~5.1' in Chapter {\rm VIII} in \cite{DiB93}) Whenever $B_\rho(x_0)\times(t_0-\tau,t_0]\Subset \Omega_T$,
    \begin{align*}
        \sup_{B_{\rho/2}(x_0)\times(t_0-\tau/2,t_0]}|\D v|\le C(\tau/\rho^2)^{1/2}\left(\dashint_{B_{\rho}(x_0)\times(t_0-\tau,t_0]}|\D v|^pdz\right)^{1/2}+C(\rho^2/\tau)^{\frac1{p-2}}.
        \end{align*}
        
    \item (Theorem~1.1' in Chapter {\rm IX} in \cite{DiB93}) If $\D v$ is bounded in $\Omega_T$, then $\D v\in C^{\al,\al/2}_{\loc}(\Omega_T)$, and for any $K\Subset\Omega_T$,
    \begin{align*}
    |\D v(X^1)-\D v(X^2)|\le C\|\D v\|_{L^\infty(\Omega_T)}\left(\frac{|x_1-x_2|+\max\left\{1,\|\D v\|_{L^\infty(\Omega_T)}^{\frac{p-2}2}\right\}|t_1-t_2|^{1/2} }{\dist(K,\partial_p\Omega_T)} \right)^\al\quad\text{for any } X^1,X^2\in K.
    \end{align*}
    \end{enumerate}
\end{lemma}

In the remainder of this section, we assume that 
$$
p>2
$$ 
and that $w$ is a solution of \eqref{eq:p-sol-par}, i.e., it belongs to $\cK_{g,p,\Omega_T}$ and satisfies \eqref{eq:par-var-ineq}. As in the previous section, we fix a subset $K\Subset \Omega_T$, and a universal constant depends only on $n,p,C^0,a,\|g\|_\infty,\|f\|_\infty,\Omega_T$ and $K$.

\begin{lemma}
    \label{lem:u-v-p-est-comp-deg}
    If $Q_r^\theta(X^0)\Subset\Omega_T$ and $v$ is the $p$-caloric replacement of $w$ in $Q_r^\theta(X^0)$, then
    \begin{align}
    &\left(\dashint_{Q_r^\theta(X^0)}|\D(w-v)|^p\right)^{1/p}\le Cr^{\frac{\gamma}{p-\gamma}}, \label{eq:u-v-est-0-deg}\\
    &\left(\dashint_{Q_r^\theta(X^0)}|\D v|^p\right)^{1/p}\le C\left(\dashint_{Q_r^\theta(X^0)}|\D w|^p\right)^{1/p}+Cr^{\frac\gamma{p-\gamma}},\label{eq:v-est-0-deg}
    \end{align}
    where $C>0$ are constants depending only on $n$ and $p$.
\end{lemma}

\begin{proof}
We proceed as in the proof of Lemma~\ref{lem:u-v-p-est-comp}, the corresponding result for the singular case, up to \eqref{eq:u-v-est-diff-sing}. In the degenerate case $p\ge2$, instead of \eqref{eq:ineq-sing}, we have
\begin{align*}
    \left(|\xi|^{p-2}\xi-|\eta|^{p-2}\eta\right)\cdot(\xi-\eta)\ge c_0|\xi-\eta|^p
\end{align*}
for any nonzero $\xi,\eta\in\R^n$ and a constant $c_0=c_0(n,p)>0$. By combining this with \eqref{eq:ineq-sing} and recalling $w^s=sw+(1-s)v$, $0\le s\le 1$, we get
\begin{align*}
    \int_{Q_r^\theta}\left(|\D w|^p-|\D v|^p+p(w-v)\partial_tv\right)&\ge c\int_0^1\int_{Q_r^\theta}|\D(w^s-v)|^pdz\frac{ds}s=c\int_0^1\int_{Q_r^\theta}|\D(w-v)|^pdz\,s^{p-1}ds\\
    &\ge c\int_{Q_r^\theta}|\D(w-v)|^pdz.
\end{align*}
By applying \eqref{eq:u-v-min-ineq} which is also valid for $p>2$, we further have
$$
\int_{Q_r^\theta}|\D(w-v)|^p\le Cr^{(n+\theta)(1-\gamma/p)+\gamma}\left(\int_{Q_r^\theta}|\D(w-v)|^p\right)^{\gamma/p}.
$$
Writing $A:=\int_{Q_r^\theta}|\D(w-v)|^p$, this can be rewritten as
$$
A\le Cr^{(n+\theta)(1-\gamma/p)+\gamma}A^{\gamma/p}.
$$
This implies the first inequality in Lemma~\ref{lem:u-v-p-est-comp-deg}. The second one follows from the first one and the triangle inequality.
\end{proof}

\begin{theorem}\label{thm:par-holder-deg}
 $w\in C^{\be,\be/p}(K)$, where $\be:=\frac{\al p^2}{p^2+(p-\gamma)(n+p+\al p) }$. Moreover, 
 $$
 \|w\|_{C^{\be,\be/p}(K)}\le C $$
 for some universal constant $C>0$.
\end{theorem}

\begin{proof}
By using Lemma~\ref{lem:p-caloric-reg-deg} and \eqref{eq:u-v-est-0-deg}, we can argue as in the proof of Theorem~\ref{thm:par-holder-sing} with trivial modifications to get an analogues of \eqref{eq:par-holder-ineq-sing}:
\begin{align*}
    \int_{Q^p_{r^{1+\e}}}|w-\mean{w}_{Q^p_{r^{1+\e}}}|^p&\le C\int_{Q^p_{r^{1+\e}}}|v-\mean{v}_{Q^p_{r^{1+\e}}}|^p+C\int_{Q^p_{r^{1+\e}}}|w-v|^p\le Cr^{(1+\e)(n+p)+p\e\al}+Cr^{n+2p+\frac{\gamma p}{p-\gamma}}.
\end{align*}
We then take $\rho=r^{1+\e}$ and $\e=\frac{p^2}{(p-\gamma)(n+p+\al p)}$ to get
$$
\left(\dashint_{Q_\rho^p}|w-\mean{w}_{Q_\rho^p}|^p\right)^{1/p}\le C\rho^{\frac{\al p^2}{p^2+(p-\gamma)(n+p+\al p) }}.
$$
This completes the proof.   
\end{proof}

\begin{lemma}
    \label{lem:par-u-est-initial-deg}
    There exist universal constants $C>0$ and $r_0>0$ such that
    $$
    \dashint_{Q_r^p(X^0)}|\D w|^p\le Cr^{p\left(\frac{p^2}{p^2+(p+n)(p-\gamma)}-1 \right)}
    $$
    whenever $X^0\in K$ and $0<r\le r_0$.
\end{lemma}

\begin{proof}
Without loss of generality, we may assume $X^0=0$. For a constant $\e\in(0,1)$ to be determined later, we choose $r_0\in(0,1)$ such that $r_0^\e\le 1/4$. By applying \eqref{eq:par-grad-v-sup-est}, which is valid in our degenerate case as well, and Lemma~\ref{lem:p-caloric-reg-deg}, we get
\begin{align*}
    \int_{Q^p_{r^{1+\e}}}|\D v|^p&\le Cr^{(n+p)(1+\e)}\|\D v\|_{L^\infty(Q^p_{r/4})}^p\le Cr^{(n+p)(1+\e)+\frac{p-2}p}\left(\dashint_{Q_{r/2}^p}|\D v|^p\right)^{p/2}+Cr^{(n+p)(1+\e)-p}\\
    &\le Cr^{(n+p)(1+\e)+\frac{(p-2)p}2-\frac p2}+Cr^{(n+p)(1+\e)-p}\le Cr^{(n+p)(1+\e)-p}.
\end{align*}
This, along with \eqref{eq:u-v-est-0-deg}, yields
$$
\int_{Q^p_{r^{1+\e}}}|\D w|^p\le C\int_{Q^p_{r^{1+\e}}}|\D(w-v)|^p+C\int_{Q^p_{r^{1+\e}}}|\D v|^p\le Cr^{n+p+\frac{\gamma p}{p-\gamma}}+Cr^{(n+p)(1+\e)-p}.
$$
Letting $\rho=r^{1+\e}$, we infer
$$
\int_{Q_\rho^p}|\D w|^p\le C\left(\rho^{n+p+\frac{\frac{\gamma p}{p-\gamma}-(n+p)\e }{1+\e}}+\rho^{n+p-\frac{p}{1+\e}}\right).
$$
By taking $\e=\frac{p^2}{(p+n)(p-\gamma)}$ so that $\frac{\gamma p}{p-\gamma}-(n+p)\e=-p$, we conclude that for $\de:=\frac{p^2}{p^2+(p+n)(p-\gamma)}\in(0,1)$,
\begin{equation*}
    \dashint_{Q_\rho^p}|\D w|^p\le C\rho^{p(\de-1) }.\qedhere
\end{equation*}
\end{proof}

\begin{lemma}
    \label{lem:iter-deg-1}
    For $X^0\in K$, assume that there exist universal constants $C_0>0$ and $r_0>0$ such that for any $X^0\in K$ and $0<r<r_0$
    $$
    \dashint_{Q_r^\theta(X^0)}|\D w|^p\le C_0r^{p(\de-1)}.
    $$
    Then there exist universal constants $C_1>0$ and $r_1>0$ such that for any $0<\e<1$, $0<r<r_1$ and $\theta'\in(2,\theta]$,
    \begin{align}
        \label{eq:improved-est-deg}
        \dashint_{Q_r^{\theta'}(X^0)}|\D w|^p\le Cr^{p(\tilde\de-1)}.
    \end{align}
Here, 
\begin{align}
    \label{eq:degree-deg}
    \tilde\de:=\min\{\tilde\de_1,\tilde\de_2,\tilde\de_3\},
\end{align}
where
\begin{align*}
    \tilde\de_1&:=\frac{p}{p-\gamma}-\frac{\e}{1+\e}\left(\frac{\gamma}{p-\gamma}+\frac{n+\theta'}{p}\right),\\
    \tilde\de_2&:=\frac{p\de-(p-\theta')-\frac{n}{\theta}(\theta-\theta') }2+\frac{\e}{2(1+\e)}\left[\frac n\theta(\theta-\theta')+p(1-\de)-(\theta'-2) \right],\\
    \tilde\de_3&:=\frac{p-\theta'}{p-2}+\frac{\e}{1+\e}\cdot\frac{\theta'-2}{p-2}.
\end{align*}
\end{lemma}

\begin{proof}
Without loss of generality, we may assume $X^0=0$. Let $v$ be the $p$-caloric replacement of $w$ in $Q_r^{\theta'}$. Note that by the condition of the lemma,
$$
\int_{Q_r^{\theta'}}|\D w|^p\le \int_{Q^\theta_{r^{\theta'/\theta}}}|\D w|^p\le Cr^{(n+\theta)\frac{\theta'}\theta+p(\de-1) },
$$
which combined with \eqref{eq:v-est-0-deg} yields
$$
\dashint_{Q_r^{\theta'}}|\D v|^p\le C\dashint_{Q_r^{\theta'}}|\D w|^p+Cr^{\frac{p\gamma}{p-\gamma}}\le Cr^{\frac{n}\theta(\theta'-\theta)+p(\de-1)}+Cr^{\frac{p\gamma}{p-\gamma}}.
$$
By using this estimate and applying Lemma~\ref{lem:p-caloric-reg-deg}, we get
\begin{align*}
    \int_{Q_{r^{1+\e}}^{\theta'}}|\D v|^p&\le Cr^{(n+\theta')(1+\e)}\|\D v\|_{L^\infty(Q_{r/2}^{\theta'} )}^p\le Cr^{(n+\theta')(1+\e)}\left[r^{\frac{\theta'-2}2}\left(\dashint_{Q_r^{\theta'}}|\D v|^p\right)^{1/2}+r^{\frac{2-\theta'}{p-2}} \right]^p\\
    &\le Cr^{(n+\theta')(1+\e)+p\frac{\theta'-2+\frac n\theta(\theta'-\theta)+p(\de-1) }2 }+Cr^{(n+\theta')(1+\e)+p\frac{2-\theta'}{p-2} }.
\end{align*}
On the other hand, we have by \eqref{eq:u-v-est-0-deg}
$$
\int_{Q_{r^{1+\e}}^{\theta'}}|\D(w-v)|^p\le C\int_{Q_r^{\theta'}}|\D(w-v)|^p\le Cr^{n+\theta'+\frac{p\gamma}{p-\gamma}}.
$$
By combining the previous two estimates, using the trivial inequality $|\D w|^p\le C|\D(w-v)|^p+C|\D v|^p$ and letting $\rho=r^{1+\e}$, we conclude \eqref{eq:improved-est-deg}.    
\end{proof}

\begin{lemma}
    \label{lem:iter-deg-2}
     Suppose $2<\theta'<\theta\le p$ and $0<\e<1$. We let
    \begin{align}
        \label{eq:del-deg}
        \de:=\frac{p-\theta}{p-2}+\frac{\e}{1+\e}\cdot\frac{\theta'-2}{p-2},\quad \de':=\de+\frac{\theta-\theta'}{(1+\e)(p-2)}.
    \end{align}
    Assume that $\theta-\theta'$ and $\e$ are small with $\theta-\theta'\ll\e$ so that
    \begin{align}
        \label{eq:iter-cond-deg}
        \frac{p}{p-2}(\theta-\theta')+(n+p+1)\e\le\gamma,\quad\de'<1,\quad\theta-\theta'\le \frac{\theta'-2}{4\left(\frac{1}{\theta'-2}+n \right)}\e.
    \end{align}
    If
    $$
    \dashint_{Q_r^\theta(X^0)}|\D w|^p\le C_0r^{p(\de-1)}
    $$
    for any $X^0\in K$ and $0<r<r_0$, where $C_0$ and $r_0>0$ are universal constants, then there exist universal constants $C_1>0$ and $r_1>0$ such that
    $$
    \dashint_{Q_r^{\theta'}(X^0)}|\D w|^p\le C_1r^{p(\de'-1)}
    $$
    whenever $0<r<r_1$.
\end{lemma}

\begin{proof}
In view of Lemma~\ref{lem:iter-deg-1}, it suffices to show that
$$
\de'\le \tilde\de,
$$
where $\tilde\de=\min\{\tilde\de_1,\tilde\de_2,\tilde\de_3\}$ is as in \eqref{eq:degree-deg}. We split its proof into three steps.

\medskip\noindent\emph{Step 1.} In this step, we show that $\de'\le \tilde\de_1$, i.e.,
$$
\de+\frac{\theta-\theta'}{(1+\e)(p-2)}\le\frac{p}{p-\gamma}-\frac{\e}{1+\e}\left(\frac{\gamma}{p-\gamma}+\frac{n+\theta'}p\right),
$$
which is equivalent to
$$
(p-\gamma)\de\le p-\frac{(\theta-\theta')(p-\gamma)}{(1+\e)(p-2)}-\frac{\e}{1+\e}\left[\gamma+\frac{(p-\gamma)(n+\theta')}p \right].
$$
This is true since $\de\le1$ and
$$
\frac{(\theta-\theta')(p-\gamma)}{(1+\e)(p-2)}+\frac{\e}{1+\e}\left(\frac{\gamma}{p-\gamma}+\frac{n+\theta'}p\right)\le \frac{p}{p-2}(\theta-\theta')+(1+n+p)\e\le\gamma.
$$

\medskip\noindent\emph{Step 2.} In this step, we prove $\de'\le\tilde\de_2$.  We have  
$$
\tilde\de_2=1+\frac{\mu}{1+\e}=1+\mu-\frac{\e}{1+\e}\mu,\quad\text{where }\mu=\frac{\theta'-2-p(1-\de)-\frac{n}{\theta}(\theta-\theta') }2.
$$
If $\mu\ge0$, then we readily have $\tilde\de_2\ge1>\de'$. Thus we may assume $\mu<0$. Then we have $\tilde\de_2\ge 1+\mu$, thus it is enough to show $\de'\le 1+\mu$. A direct computation gives that this is equivalent to
$$
\frac{\e}{1+\e}(\theta'-2)\ge\left[\frac2{(1+\e)(p-2)}+\frac{n}\theta+1 \right](\theta-\theta').
$$
This is true by the last condition on \eqref{eq:iter-cond-deg}:
$$
\left[\frac2{(1+\e)(p-2)}+\frac{n}\theta+1 \right](\theta-\theta')\le\left(\frac{2}{\theta'-2}+2n\right)(\theta-\theta')\le\frac{\e}2(\theta'-2)\le\frac{\e}{1+\e}(\theta'-2).
$$

\medskip\noindent\emph{Step 3.} It remains to prove $\de'\le \tilde\de_3$. This simply follows from
\begin{align*}
\de'&=\frac{p-\theta}{p-2}+\frac{\e}{1+\e}\cdot\frac{\theta'-2}{p-2}+\frac{\theta-\theta'}{(1+\e)(p-2)}\\&\le\frac{p-\theta}{p-2}+\frac{\e}{1+\e}\cdot\frac{\theta'-2}{p-2}+\frac{\theta-\theta'}{p-2}=\frac{p-\theta'}{p-2}+\frac{\e}{1+\e}\cdot\frac{\theta'-2}{p-2}=\tilde\de_3.
\end{align*}
This completes the proof.    
\end{proof}

\begin{lemma}
    \label{lem:par-u-v-est-deg}
    For any $\sigma\in(0,1)$, there are $\theta\in(2,p)$, $C>0$ and $r_0>0$, depending only on $n,p,C^0,a,\|g\|_\infty,\|f\|_\infty,\Omega_T,K,\sigma$, such that if $X^0\in K$ and $0<r<r_0$, then
    $$
    \dashint_{Q_r^\theta(X^0)}|\D w|^p\le Cr^{p(\sigma-1)}.
    $$
    Here, for each $\sigma\in(0,1)$, $\theta\in(2,p)$ can be chosen so that $\theta\to2$ as $\sigma\to1$.
\end{lemma}

\begin{proof}
We let
\begin{align*}
&\e:=\min\left\{\frac{p^2}{p^2+(p+n)(p-\gamma)}, \frac{\gamma}{2(n+p+1)},\sigma, 1-\sigma\right\}\in(0,1/8),\\
&\theta^\sharp:=\frac{p-\sigma(p-2)}2+1=1+\sigma+\frac{p}2(1-\sigma)\in(2,p),\\
&\nu:=\min\left\{\frac{\theta^\sharp-2}2, \frac{\theta^\sharp-2  }{8\left(\frac2{\theta^\sharp-2}+n \right)  }\e, \frac{p-2}{2p}\gamma \right\}\in(0,1).
\end{align*}
We then consider an affine function
$$
\psi(\theta):=\frac{-\theta+p+\e(p-\nu-2) }{(1+\e)(p-2) }=\frac{-\theta+p}{p-2}+\frac{\e}{1+\e}\cdot\frac{\theta-\nu-2}{p-2}.
$$
A direct computation gives
$$
\psi^{-1}(\sigma)=p-\sigma(p-2)-\e[(p-2)(1-\sigma)-\nu].
$$
Since $\nu<\frac{\theta^\sharp-2}2=\frac{(p-2)(1-\sigma)}4<(p-2)(1-\sigma)$ and $p-\sigma(p-2)>p-(p-2)=2$,
$$
\psi^{-1}(\sigma)\ge p-\sigma(p-2)>\frac{p-\sigma(p-2)}2+1=\theta^\sharp.
$$
Moreover, since $\psi(p)=\frac{\e(p-\nu-2)}{(1+\e)(p-2)}<\frac{\e}{1+\e}<\sigma$ and $\psi$ is decreasing, $2<\theta^\sharp<\psi^{-1}(\sigma)<p$. We consider sequences $\theta_j:=p-\nu j$ and $\de_j:=\psi(p-\nu j)$, $j\in \mathbb{N}\cup\{0\}$. Since $\nu\le\frac{\theta^\sharp-2}2$, there exists $J\in \mathbb{N}$ such that $\frac{\theta^\sharp+2}2<\theta_J\le \theta^\sharp$. 

We claim that there exist universal constants $C>0$ and $r_0>0$ such that
\begin{align}
    \label{eq:par-grad-u-est-seq-deg}
    \dashint_{Q_r^{\theta_j}(X^0)}|\D w|^p\le Cr^{p(\de_j-1)}
\end{align}
whenever $X^0\in K$, $0<r<r_0$ and $0\le j\le J$. We prove it by induction on $j$. Since $\theta_0=p$ and $\de_0=\psi(p)=\frac{\e(p-\nu-2)}{(1+\e)(p-2)}<\e\le \frac{p^2}{p^2+(p+n)(p-\gamma) }$, \eqref{eq:par-grad-u-est-seq-deg} holds for $j=0$ by Lemma~\ref{lem:par-u-est-initial-deg}. Now we assume that \eqref{eq:par-grad-u-est-seq-deg} holds for $j\in \{0,1,2,\ldots,J-1\}$ and prove it for $j+1$. To this end, we apply Lemma~\ref{lem:iter-deg-2} with $\de=\de_j$, $\de'=\de_{j+1}$, $\theta=\theta_j$ and $\theta'=\theta_{j+1}$. To prove that it is applicable, we need to verify \eqref{eq:del-deg} and \eqref{eq:iter-cond-deg}. For \eqref{eq:del-deg}, we use the definition of $\psi$ and $\theta'=\theta_{j+1}=\theta_j+\nu$ to have
\begin{align*}
    &\de_j=\psi(\theta_j)=\frac{-\theta_j+p}{p-2}+\frac{\e}{1+\e}\cdot\frac{\theta_j-\nu-2}{p-2}=\frac{-\theta+p}{p-2}+\frac{\e}{1+\e}\cdot\frac{\theta'-2}{p-2}=\de,\\
    &\de_{j+1}=\psi(\theta_{j+1})=\psi(\theta_j)+\frac{\theta_{j}-\theta_{j+1}}{(1+\e)(p-2)}=\de+\frac{\theta-\theta'}{(1+\e)(p-2)}=\de'.
\end{align*}
Concerning \eqref{eq:iter-cond-deg}, the first condition follows from $\e\le \frac{\gamma}{2(n+p+1)}$  and $\nu\le \frac{p-2}{2p}\gamma$. The second one follows from 
$$
\de'=\de_{j+1}\le \de_J=\psi(\theta_J)<\psi(2)=1-\frac{\e}{1+\e}\cdot\frac{\nu}{p-2}<1.
$$
Regarding the last one, we use $\nu\le\frac{\theta^\sharp-2 }{8\left(\frac{2}{\theta^\sharp-2}+n \right) }\e$ and $\theta'=\theta_{j+1}\ge \theta_J>\frac{\theta^\sharp+2}2$ to get
$$
\theta'-\theta=\nu\le \frac{\theta^\sharp-2 }{8\left(\frac{2}{\theta^\sharp-2}+n \right) }\e=\frac{\frac{\theta^\sharp +2}2-2 }{4\left(\frac1{\frac{\theta^\sharp+2}{2}-2 }+n \right) }\e\le \frac{\theta'-2 }{4\left(\frac1{\theta'-2}+n \right) }\e .
$$
This completes the proof of the claim \eqref{eq:par-grad-u-est-seq-deg}.

Now, by \eqref{eq:par-grad-u-est-seq-deg}, we have
$$
\dashint_{Q_r^{\theta_J}(X^0)}|\D w|^p\le Cr^{p(\de_J-1)}.
$$
Since $\de_J=\psi(\theta_J)\ge \psi(\theta^\sharp)\ge \psi(\psi^{-1}(\sigma))=\sigma$, we conclude
$$
\dashint_{Q_r^{\theta_J}(X^0)}|\D w|^p\le Cr^{p(\sigma-1)}.
$$

 It remains to show that $\theta\to 2$ as $\sigma\to1$. For this aim, we recall $\psi^{-1}(\sigma)=p-\sigma(p-2)-\e[(p-2)(1-\sigma)-\nu]$. From $0<\e\le 1-\sigma$, we see that $\e\to0$ as $\sigma\to1$. Therefore, $\psi^{-1}(\sigma)\to2$ as $\sigma\to1$. 
\end{proof}

\begin{theorem}
    \label{thm:par-grad-holder}
   $\D w\in C^{\hat\be,\hat\be/p}(K)$ for some $\hat\be=\hat\be(n,p,a)\in (0,1)$. Moreover, 
    $$
    \|\D w\|_{C^{\hat\be,\hat\be/p }(K)}\le C
    $$
    for some universal constant $C>0$.
\end{theorem}

\begin{proof}
Throughout this proof, we write
$$
\e:=\frac{\gamma p}{2(p-\gamma)(n+p)}\in (0,1),\quad \nu:=\frac{\e\al}{4p}\in(0,1).
$$
We also consider $\theta=\theta(n,p,a,\al)\in(2,p)$ close to $2$ to be chosen later. In particular, we require
\begin{align}
    \label{eq:theta-cond-deg}
    \frac{\theta-2}{p-2}>-\nu.
\end{align}
We take $r_0\in(0,1)$ small so that $r_0^\e\le 1/4$. For a parabolic cylinder $Q_r^\theta(X^0)\Subset \Omega_T$ with  $X^0\in K$ and $0<r<r_0$, let $v$ be the $p$-caloric replacement of $u$ in $Q_r^\theta(X^0)$. For the sake of simplicity, we write $X^0=0$. Then, by Lemmas~\ref{lem:p-caloric-reg-deg} and \ref{lem:u-v-p-est-comp-deg},
\begin{align*}
    \|\D v\|_{L^\infty(Q_{r/2}^\theta)}\le Cr^{\frac{\theta-2}2}\left(\dashint_{Q_r^\theta}|\D v|^p\right)^{1/2}+Cr^{\frac{2-\theta}{p-2}}\le C\left(\dashint_{Q_r^\theta}|\D w|^p+Cr^{\frac{p\gamma}{p-\gamma}}\right)^{1/2}+Cr^{\frac{2-\theta}{p-2}}.
\end{align*}
We apply Lemma~\ref{lem:par-u-v-est-deg} with $\sigma\in(0,1)$ close to $1$ so that $\frac{p(\sigma-1)}2>-\nu$ and $\theta$ satisfies \eqref{eq:theta-cond-deg}. This is possible since $\theta\to2$ as $\sigma\to1$. It follows that
$$
\|\D v\|_{L^\infty(Q^\theta_{r/2})}\le Cr^{-\nu}.
$$
This, along with Lemma~\ref{lem:p-caloric-reg-deg}, gives that for any $X^1,X^2\in Q^\theta_{r^{1+\e}}$,
\begin{align*}
    |\D v(X^1)-\D v(X^2)|&\le C\|\D v\|_{L^\infty(Q_{r/2}^\theta)}\left(\frac{|x_1-x_2|+\max\{1,\|\D v\|_{L^\infty(Q_{r/2}^\theta)}\}^{\frac{p-2}2}|t_1-t_2|^{1/2} }{\dist(Q_{r^{1+\e}}^\theta,\partial_pQ_{r/2}^\theta) } \right)^\al\\
    &\le Cr^{-\nu}\left(r^\e+r^{(1+\e)\frac\theta2-\frac{p-2}2\nu }\right)^\al\le Cr^{\e\al-\nu}+Cr^{(1+\e)\frac{\theta\al}2-\left(1+\frac{p-2}2\al\right)\nu }.
\end{align*}
Here, from $\nu=\frac{\e\al}{4p}$, we find $\e\al-\nu\ge\frac{\e\al}2$ and $(1+\e)\frac{\theta\al}2-\left(1+\frac{p-2}2\al\right)\nu\ge (1+\e)\al-\frac{p}2\nu\ge\al$. Thus, we have
$$
|\D v(X^1)-\D v(X^2)|\le Cr^{\frac{\e\al}2}.
$$
By combining this with Jensen's inequality and \eqref{eq:u-v-est-0-deg}, we infer
\begin{align*}
    \int_{Q^\theta_{r^{1+\e}}}|\D w-\mean{\D w}_{Q^\theta_{r^{1+\e}}}|^p&\le C\int_{Q^\theta_{r^{1+\e}}}|\D v-\mean{\D v}_{Q^\theta_{r^{1+\e}}}|^p+C\int_{Q^\theta_{r^{1+\e}}}|\D(w-v)|^p \\
    &\le Cr^{(n+\theta)(1+\e)+p\frac{\e\al}2}+Cr^{n+\theta+\frac{p\gamma}{p-\gamma}}.
\end{align*}
Letting $\rho=r^{1+\e}$, we get
$$
\dashint_{Q_\rho^\theta}|\D w-\mean{\D w}_{Q_\rho^\theta}|^p\le C\rho^{p\frac{\al\e}{2(1+\e)}}+C\rho^{p\frac{\frac{\gamma}{p-\gamma}-\frac\e p(n+\theta) }{1+\e}}.
$$
From $\e=\frac{p\gamma}{2(p-\gamma)(n+p)}$, we have $\frac{\gamma}{p-\gamma}-\frac{\e}p(n+\theta)\ge\frac{\gamma}{2(p-\gamma)}$, thus
$$
\left(\dashint_{Q_\rho^\theta}|\D w-\mean{\D w}_{Q_\rho^\theta}|^p\right)^{1/p}\le C\rho^{\hat\be},
$$
where $\hat\be:=\frac1{1+\e}\min\left\{\frac{\al\e}2,\frac{\gamma}{2(p-\gamma)}\right\}\in(0,1)$. This completes the proof.    
\end{proof}

\section*{Declarations}

\noindent {\bf  Disclosure statement:}  The authors report there are no competing interests to declare.

\medskip
\noindent {\bf  Declaration of generative AI use:}  The authors report generative AI was not used in their research or preparation of this manuscript.


\end{document}